\documentclass[a4paper,
               11pt,
               pdftex,
               normalheadings,
               headsepline,
               footsepline,
               onecolumn,
               headinclude,
               footinclude,
               DIV14,
               abstracton]
{scrartcl}

\usepackage
{
    graphicx,
    amssymb,
    amsmath,
    amsthm,
    xcolor,
    dsfont,
    algpseudocode,
    authblk,
}

\usepackage[ngerman, english]{babel}

\usepackage[
left=25mm,
right=25mm,
top=20mm,
bottom=35mm
]{geometry}

\usepackage{subcaption}

\usepackage[colorlinks,
            pdffitwindow=false,
            plainpages=false,
            pdfpagelabels=true,
            pdfpagemode=UseOutlines,
            pdfpagelayout=SinglePage,
            bookmarks=false,
            colorlinks=true,
            hyperfootnotes=false,
            linkcolor=blue,
            citecolor=green!50!black]
{hyperref}

\usepackage{enumerate}
\usepackage{algorithm}

\usepackage{tabularx}

\newtheorem{theorem}{Theorem}[section]

\newtheorem{definition}[theorem]{Definition}

\let\OLDthebibliography\thebibliography
\renewcommand\thebibliography[1]{
	\OLDthebibliography{#1}
	\setlength{\parskip}{0pt}
	\setlength{\itemsep}{0pt plus 0.3ex}
}

\date{}

\begin{document}

% document title and author
\title{\LARGE Explicit Multi-objective Model Predictive Control for Nonlinear Systems Under Uncertainty}
\author[1]{\large Carlos Ignacio Hern\'andez Castellanos}
\author[2]{Sina Ober-Bl\"obaum}
\author[3]{Sebastian Peitz}
\affil[1]{Department of Computer Science, CINVESTAV-IPN, Mexico City, Mexico}
\affil[2]{Department of Engineering Science, University of Oxford, United Kingdom}
\affil[3]{Department of Mathematics, Paderborn University, Germany}

\maketitle

%%%%%%%%%%%%%%%%%%%%%%%%%%%%%%%%%%%%%%%%%%%%%%%%%%%%%%%%%%%%%%%%%%%%%%%%%%%%%%%%%%%%%%%%%%%%%%
%% Abstract
%%%%%%%%%%%%%%%%%%%%%%%%%%%%%%%%%%%%%%%%%%%%%%%%%%%%%%%%%%%%%%%%%%%%%%%%%%%%%%%%%%%%%%%%%%%%%%
\begin{abstract}
In real-world problems, uncertainties (e.g., errors in the measurement, precision errors) often lead to poor performance of numerical algorithms when not explicitly taken into account. This is also the case for control problems, where optimal solutions can degrade in quality or even become infeasible. Thus, there is the need to design methods that can handle uncertainty. In this work, we consider nonlinear multi-objective optimal control problems with uncertainty on the initial conditions, and in particular their incorporation into a feedback loop via model predictive control (MPC). In multi-objective optimal control, an optimal compromise between multiple conflicting criteria has to be found. For such problems, not much has been reported in terms of uncertainties. To address this problem class, we design an offline/online framework to compute an approximation of efficient control strategies. This approach is closely related to explicit MPC for nonlinear systems, where the potentially expensive optimization problem is solved in an offline phase in order to enable fast solutions in the online phase. In order to reduce the numerical cost of the offline phase -- which grows exponentially with the parameter dimension -- we exploit symmetries in the control problems. Furthermore, in order to ensure optimality of the solutions, we include an additional online optimization step, which is considerably cheaper than the original multi-objective optimization problem. We test our framework on a car maneuvering problem where safety and speed are the objectives. The multi-objective framework allows for online adaptations of the desired objective. Alternatively, an automatic scalarizing procedure yields very efficient feedback controls. Our results show that the method is capable of designing driving strategies that deal better with uncertainties in the initial conditions, which translates into potentially safer and faster driving strategies.
\end{abstract}

%%%%%%%%%%%%%%%%%%%%%%%%%%%%%%%%%%%%%%%%%%%%%%%%%%%%%%%%%%%%%%%%%%%%%%%%%%%%%%%%%%%%%%%%%%%%%%
%% Introduction
%%%%%%%%%%%%%%%%%%%%%%%%%%%%%%%%%%%%%%%%%%%%%%%%%%%%%%%%%%%%%%%%%%%%%%%%%%%%%%%%%%%%%%%%%%%%%%
\section{Introduction}
    
    In many real-world engineering problems, one is faced with the problem that several objectives have to be optimized concurrently, leading to a multi-objective optimization problem (MOP). The typical goal for such problems is to identify the set of optimal tradeoff solutions (the so-called \emph{Pareto set}) and its image in objective space, the \emph{Pareto front}. 
    Similar problems occur in the context of \emph{control}, where an input function has to be computed such that a dynamical system behaves optimally with respect to multiple objectives. One particularly successful approach for feedback control with conflicting criteria is \emph{model predictive control (MPC)}, where an open loop optimal control problem is solved repeatedly over a finite horizon and then directly applied to the real system, which is running in parallel. As this requires the solution of MOPs in real-time, special measures need to be taken in the case of multiple criteria. Possible approaches are the \emph{weighting} of objectives \cite{BP09} or \emph{reference point tracking} \cite{ZF12}, see also \cite{PD18} for an overview. An alternative approach is \emph{explicit MPC} \cite{bemporad2002explicit}, where instead of solving an optimization problem online, the optimal input is selected from a library of optimal inputs which is computed in an offline phase. In the multi-objective case, this was used in \cite{PSO+17} and \cite{ober2018explicit}.
    
    An additional difficulty for feedback control is the issue of model inaccuracies.
    Thus, the decision maker may, in practice, not always be interested in the exact Pareto optimal solutions, in particular, if these are sensitive to perturbations \cite{Beyer20073190}. Instead, solutions which are more robust are preferable, which leads to \emph{robust multi-objective optimization problems (RMOPs)} \cite{ehrgott2014minmax}. In this case, the definition of a robust solution is not unique since it depends on the information available and the type of uncertainty present in the problem \cite{kuroiwa2012robust,doolittle2018robust,fliege2014robust,Eichfelder2017,PD18b}. The interested reader is referred to \cite{ide2016robustness} for a survey of the different robustness definitions.
    
    In this work, we will study nonlinear uncertain multi-objective optimal control problems in the sense of \emph{set-based minmax robustness} \cite{ehrgott2014minmax}. This definition is the natural extension of the minmax from single-objective optimization (also called \emph{worst-case optimization}). There exist a few applications of worst-case multi-objective optimization. Examples can be found in \cite{doolittle2018robust}, where the authors solved an internet routing problem, and \cite{fliege2014robust}, where it was used for portfolio optimization.
    
    Until now, worst-case optimization has not drawn much attention in multi-objective optimal control. However, there exist multiple studies for the single-objective case. Lofberg \cite{lofberg2003approximations} introduced an approach for solving closed-loop minimax problems for linear-time discrete systems in an MPC framework to avoid the controller to be over-conservative. \cite{bemporad2003min} addressed discrete-time uncertain linear systems with polyhedral parametric uncertainty. In \cite{walton2016numerical}, the authors described a numerical method to solve nonlinear control problems with parameter uncertainty. In this case, the problem was solved by a sequence of multi-parametric linear programs. Hu and Ding \cite{hu2019efficient} presented an offline MPC approach to reduce the online computational burden on discrete-time uncertain LPV systems. 
    
    In this article, we build on these ideas in order to construct a multi-objective MPC framework for nonlinear systems with uncertainties, which extends the work from \cite{ober2018explicit}, where the deterministic case was considered. In particular, we consider uncertainties in the initial conditions which might arise from inaccurate sensor measurements in each MPC iteration. As multi-objective optimization problems usually cannot be solved in real-time, we use ideas from explicit MPC for nonlinear systems, where a library of solutions is computed in an offline phase for many different initial conditions. The offline phase thus requires the solution of a parametric multi-objective optimal control problem with uncertainties. In the online phase, the problem is then reduced to selecting a solution from a library (and potentially interpolation between multiple entries).
    %To address this problem, we extend the recently developed framework in \cite{ober2018explicit} for multi-objective MPC of nonlinear systems and design an offline/online framework to compute approximations of efficient control strategies. In particular, we construct a library of efficient solutions for a discretization of initial conditions in an offline phase. This library allows for the efficient selection of the solutions during the simulation.
    To avoid feasibility issues due to the interpolation, another extension is proposed where
    %Moreover, it might be infeasible to precompute the efficient solutions for all initial conditions beforehand. Thus, 
    during the simulation, the solution from the library is further refined (in a comparably cheap optimization step) in order to match the exact initial conditions. Finally, we exploit symmetries in the control problems 
    %and corresponding motion primitives 
    to reduce the complexity of the  offline phase and increase the efficiency of the proposed methods.
    
    The article is organized as follows. In Section 2, we present the basic definitions of multi-objective optimal control under uncertainty, symmetries, and model predictive control. In Section 3, we then extend the result on symmetries in nonlinear control systems from \cite{ober2018explicit} to uncertainties before introducing our framework for solving multi-objective optimal control under uncertainty in the initial conditions in Section 4. We then study an example from autonomous driving in Section 5, and we present our conclusions and future paths for research in Section 6.

\section{Background}
    In this section, we introduce the basic concepts that are utilized in the consecutive sections. First, we introduce the multi-objective optimal control problem. Next, we introduce the concept of uncertainty and efficiency that we will use in this work.
    %before extending the concept of symmetries to multi-objective optimal control problems under uncertainty. 
    Finally, we present some basic concepts of model predictive control.
    \subsection{Multi-objective optimal control}
    The basis for all considerations is the following general nonlinear multi-objective optimal control problem:
    \begin{definition}[Multi-objective optimal control problem]
    \begin{equation}
        \label{eq:mocp}
        \begin{split}
            \min\limits_{x\in\mathcal{X}, u\in\mathcal{U}} J(x,u) &= 
                \begin{pmatrix}
                    \int^{t_e}_{t_0} C_1(x(t),u(t))dt + \Phi_1(x(t_e))\\
                    \vdots\\
                    \int^{t_e}_{t_0} C_k(x(t),u(t))dt + \Phi_k(x(t_e))
                \end{pmatrix} \\
                s.t. \quad\quad\quad\dot x(t) &= f(x(t),u(t)),\\
                x(t_0) &= x_0,\\
                g_i(x(t),u(t)) &\leq 0,\;  i=1,\ldots,l,\\
                h_j(x(t),u(t)) &= 0,\;  j=1,\ldots,m, \\
        \end{split}
    \end{equation}
    %where $t\in (t_0,t_e]$, $x(t) \in \mathbb{R}^{n_x}$ is the system state, and $u(t) \in U \subseteq \mathbb{R}^{n_u}$ is the control variable. Moreover, $U$ is closed and convex. 
    where $t\in (t_0,t_e]$, $x \in \mathcal{X} = W^{1,\inf}([t_0,t_e], \mathbb{R}^{n_x})$ is the system state, and $u \in \mathcal{U} = L^{\infty}([t_0,t_e],U)$ is the control trajectory with $U$ being closed and convex.
    $J:\mathcal{X}\times\mathcal{U}\rightarrow\mathbb{R}^k$ denotes the objective function with $k$ objectives in conflict, 
    $f$ describes the system dynamics, and $g = (g_1,\ldots,g_l)^T$ and $h = (h_1,\ldots,h_m)^T$ are the inequality and equality constraint functions, respectively.
    The functions $C_i:\mathbb{R}^{n_x}\times U \rightarrow \mathbb{R}, \Phi_i: \mathbb{R}^{n_x}\times U \rightarrow \mathbb{R}$ are continuously differentiable for $i=1,\ldots,k$. Moreover, $f:\mathbb{R}^{n_x}\times U \mapsto \mathbb{R}^{n_x}$ is Lipschitz continuous, and $g,h: \mathbb{R}^{n_x}\times U \mapsto \mathbb{R}^l$ are continuously differentiable. The pair $(x,u)$  is called a \emph{feasible pair} if it satisfies the constraints of Problem \ref{eq:mocp}. The space of the control trajectories $\mathcal{U}$ is also known as the decision space and its image is the so-called objective space.
    \end{definition}
    %The sets $\mathcal{X} = W^{1,\inf}([t_0,t_e], \mathbb{R}^{n_x})$ and $\mathcal{U} = L^{\infty}([t_0,t_e],U)$ are function spaces. 
    %$J:\mathcal{X}\times\mathcal{U}\rightarrow\mathbb{R}^k$ is the objective function with $k$ objectives in conflict. The functions $C_i:\mathbb{R}^{n_x}\times U \rightarrow \mathbb{R}, \Phi_i: \mathbb{R}^{n_x}\times U \rightarrow \mathbb{R}$ are continuously differentiable for $i=1,\ldots,k$. Next, $f:\mathbb{R}^{n_x}\times U \mapsto \mathbb{R}^{n_x}$ is Lipschitz continuous. Further, $g: \mathbb{R}^{n_x}\times U \mapsto \mathbb{R}^l$, $g = (g_1,\ldots,g_l)^T$ and $h: \mathbb{R}^{n_x}\times U \mapsto \mathbb{R}^m$, $h = (h_1,\ldots,h_m)^T$ are continuously differentiable. $g$ and $h$ are the inequality and equality constraint functions, respectively. The pair $(x,u)$  is called a \emph{feasible pair} if it satisfies the constraints of Problem \ref{eq:mocp}. The space of the control trajectories $\mathcal{U}$ is also known as the decision space and its image is the so-called objective space.

    Problem \ref{eq:mocp} can be simplified by introducing the flow of the dynamical system:
    \begin{equation}
        \varphi_u(x_0,t) = x_0 + \int_{t_0}^{t}f(x(t),u(t))dt.
    \end{equation}
    
    As a consequence, the explicit dependency of $J$, $g$ and $h$ on $x$ can be removed and the formulation replaced by a parametric problem with $x_0$ being the parameter:
    \begin{equation}
        \label{eq:smocp}
        \begin{split}
            \min\limits_{u\in\mathcal{U}}& \hat J(x_0,u)\\
            &\hat g_i(x_0,\mathfrak{u}) \leq 0, i=1,\ldots,l,\\
            &\hat h_j(x_0,\mathfrak{u}) = 0, j=1,\ldots,m,
        \end{split}
    \end{equation}
    where $t\in (t_0,t_e]$ and
    \begin{equation}
        \hat J_i(x_0,u) = \int_{t_0}^{t_e}\hat C_i(x_0,\mathfrak{u})dt + \hat\Phi(x_0,\mathfrak{u})
    \end{equation}
    with $\hat C_i(x_0,\mathfrak{u}):=C_i(\varphi_u(x_0,t),u(t))$ and $\hat \Phi_i(x_0,\mathfrak{u}):=\Phi(\varphi(x_0,t_e))$ for $i=1,\ldots,k$. Here, $\mathfrak{u}:=u|_{[t_0,t]}$ is introduced to preserve the time dependency. The constraints $\hat g(x_0,\mathfrak{u})$ and $\hat h(x_0,\mathfrak{u})$ are defined accordingly. $u$ is called a \emph{feasible curve} if it satisfies the equality and inequality constraints $\hat g_i, i = 1,\ldots,l$, and $\hat h_j,j=1,\ldots,m$. In the remainder of the manuscript, we will restrict ourselves to inequality constraints.
    
    \subsection{Introducing uncertainty to MOCP}
        The problem formulation \eqref{eq:smocp} allows us to treat uncertainties in the initial conditions as parameter uncertainties:
        \begin{definition}[Multi-objective optimal control problem  with uncertainty in the initial state (uMOCP)]
            Given a set of initial conditions $\mathcal{Y} \subseteq \mathbb{R}^{n_x}$ and a set of control variables $U \subseteq \mathbb{R}^{n_u}$, a known uncertain set $\mathcal{Z} \subseteq \mathbb{R}^{n_x}$ and an objective function $\hat J: \mathcal{Y}\times \mathcal{U}\times\mathcal{Z} \rightarrow \mathbb{R}^k$, a multi-objective optimal control problem with uncertainty in the initial state
            \begin{equation}
                \mathcal{P}(\mathcal{Z}) := (\mathcal{P}(\alpha),\alpha \in \mathcal{Z})
            \end{equation}
            is defined as the family of parametrized problems
            \begin{equation}
                \begin{split}                    
                    \mathcal{P}(\alpha) &:= \min \hat J(x_0+\alpha, u)\\
                  s.t.\quad  & \hat g_i (x_0 + \alpha, u) \leq 0, \;i=1,\ldots,l.
                \end{split}
            \end{equation}
        \end{definition}
        
        Note that the solution to such a problem is not uniquely defined. In this work, we use the definition of set-based minmax robustness (SBR) proposed in \cite{ehrgott2014minmax} in the context of multi-objective optimal control (see \cite{ide2016robustness} for other interpretations of efficiency):
        
        \begin{equation}
        \label{eq:usmocp}
            \begin{split}
                \min\limits_{u \in \mathcal{U}} &\sup\limits_{\alpha \in \mathcal{Z}} \hat J(x_0+\alpha, u)\\
                s.t.\quad & \sup\limits_{\alpha\in\mathcal{Z}} \hat g_i (x_0 + \alpha, u) \leq 0, \;i=1,\ldots,l.
            \end{split}            
        \end{equation}
        
        The authors generalize the definition of efficiency from classical multi-objective optimization problems \cite{pareto:71} by replacing the single points $\hat J(x_0, u) \in \mathbb{R}^k$ in Problem \eqref{eq:smocp} by the sets
        \begin{equation}
            \hat J_\mathcal{Z}(x_0,u) = \{\hat J(x_0+\alpha, u): \alpha \in \mathcal{Z}\}
        \end{equation}
        of all possible objective values under all scenarios. Similarly to \cite{Eichfelder2017}, the true initial condition $x_0 \in \mathbb{R}^{n_x}$  is an element of the set $\{x_0\}+\mathcal{Z} = \{x_0+\alpha: \alpha\in \mathcal{Z}\}$\footnote{Note that $\hat J$ is now of the form $\hat J:\mathbb{R}^{n_x}\times\mathcal{U}\rightarrow\mathbb{R}^k$ since we consider decision uncertainty.}. Furthermore, we require $0 \in \mathcal{Z}$ to include the exact initial condition.

        \begin{definition}[Set-based minmax robust efficiency]
            Given an uncertain multi-objective optimal control problem $\mathcal{P}(\mathcal{Z})$ and an initial condition $x_0$, a feasible curve $\bar u \in \mathcal{U}$ is called set-based minmax robust efficient (re) if there is no $u' \in \mathcal{U}\backslash{\bar u}$ such that
            \begin{equation}
                \label{eq:sbr}
                \hat J_\mathcal{Z}(x_0, u') \subseteq \hat J_\mathcal{Z}(x_0, \bar u) - \mathbb{R}^k_\succeq,
            \end{equation}
            where $\mathbb{R}^k_\succeq$ denotes the set $\{z\in \mathbb{R}^k: z \geq 0, i = 1,\ldots,k\}$ and the relation $\geq$ is defined as presented in \cite{ehrgott:05}.
        \end{definition}
        Note, that in the remainder of the paper, we use the term efficient curve (resp.~set) to refer to the set-based minmax robust efficient curve (resp.~set). Finally, we will define as $\mathcal{R}$ the efficient set and $J_\mathcal{Z}(\mathcal{R})$ its image. In the following example, we visualize the previous definitions. Consider the following uMOCP with $J: \mathbb{R}^2\times U \rightarrow\mathbb{R}^2$:
       
        \begin{equation}
            \label{eq:lss25}
            \min\limits_{u\in U} J(\alpha, u) = 
                \min\limits_{u\in U}
                \begin{pmatrix}
                    \frac{1}{n^{0.25}}((u_1+\alpha_1)^2+(u_2+\alpha_2)^2)^{0.25} \\
                    \frac{1}{n^{0.25}}((1-(u_1+\alpha_1))^2+(1-(u_2+\alpha_2))^2)^{0.25}
                \end{pmatrix}.
        \end{equation}
        The set $U$ consists of four feasible points, i.e., $U=\{u_{I},u_{II},u_{III},u_{IV}\}$ with
        \begin{align*}
            u_I = \begin{pmatrix} -0.3545 \\ 1.3044 \end{pmatrix},~
            u_{II} = \begin{pmatrix} 0.6445 \\ 0.2392 \end{pmatrix}, ~
            u_{III} =  \begin{pmatrix} 0.3760 \\ -0.7945 \end{pmatrix}, ~
            u_{IV} = \begin{pmatrix} 1.7017 \\ 0.6869 \end{pmatrix}
        \end{align*}

        %$U=\{u_{I},u_{II},u_{III},u_{IV}\}$, $u_{I}=[-0.3545, 1.3044]^T, \linebreak u_{II}=[0.6445, 0.2392]^T, u_{III}=[0.3760   -0.7945]^T,$ $u_{IV}=[1.7017 ,0.6869]^T$ 
        and $-0.2 \leq \alpha_i \leq 0.2$, $i=1,2$.
        
        Figure \ref{fig:ex5} shows an example of minmax efficiency, where the set of feasible points\footnote{Note since we consider a finite-dimensional problem, i.e., $u\in \mathbb{R}^2$, feasible curves reduce to feasible points in $\mathbb{R}^2$.} is shown in \ref{fig:ex5s1}. Figure \ref{fig:ex5s2} shows all possible realizations of the feasible points when considering the uncertainty $\alpha$. Next, Figure \ref{fig:ex5s3} shows the supremum sets for each $\hat J_{\cal{Z}}(x_0, u)$, $u\in U$, and we can observe in Figure~\ref{fig:ex5s4} that the $\hat J_{\mathcal{Z}}(0,u_{I}) - \mathbb{R}^2$ (blue) and $\hat J_{\mathcal{Z}}(0,u_{III}) - \mathbb{R}^2$ (purple) contain the $\hat J_{\mathcal{Z}}(0,u_{II})$ (red). Thus, they are not efficient. Further, the $u_{II}$ (red) is efficient since $\hat J_{\mathcal{Z}}(0,u_{II}) - \mathbb{R}^2$ does not contain $\hat J_{\mathcal{Z}}(0,u_I) (blue), \hat J_{\mathcal{Z}}(0,u_{III})$ (purple) nor $\hat J_{\mathcal{Z}}(0,u_{IV})$ (yellow). Finally, $u_{IV}$ is also an efficient solution.
                
        \begin{figure}
        	\centering
            \begin{subfigure}[t]{.4\textwidth}
                \includegraphics[width=\columnwidth]{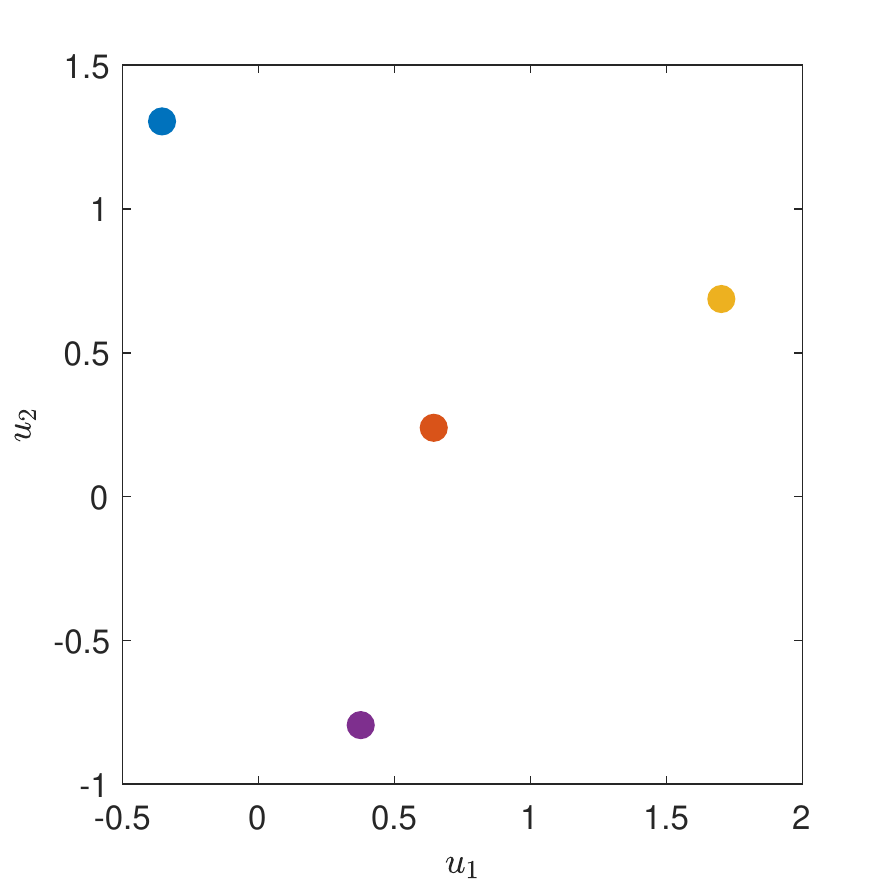}
                \caption{Feasible points in $U$.}
                \label{fig:ex5s1}
            \end{subfigure}
        	\hfil
            \begin{subfigure}[t]{.4\textwidth}
                \includegraphics[width=\columnwidth]{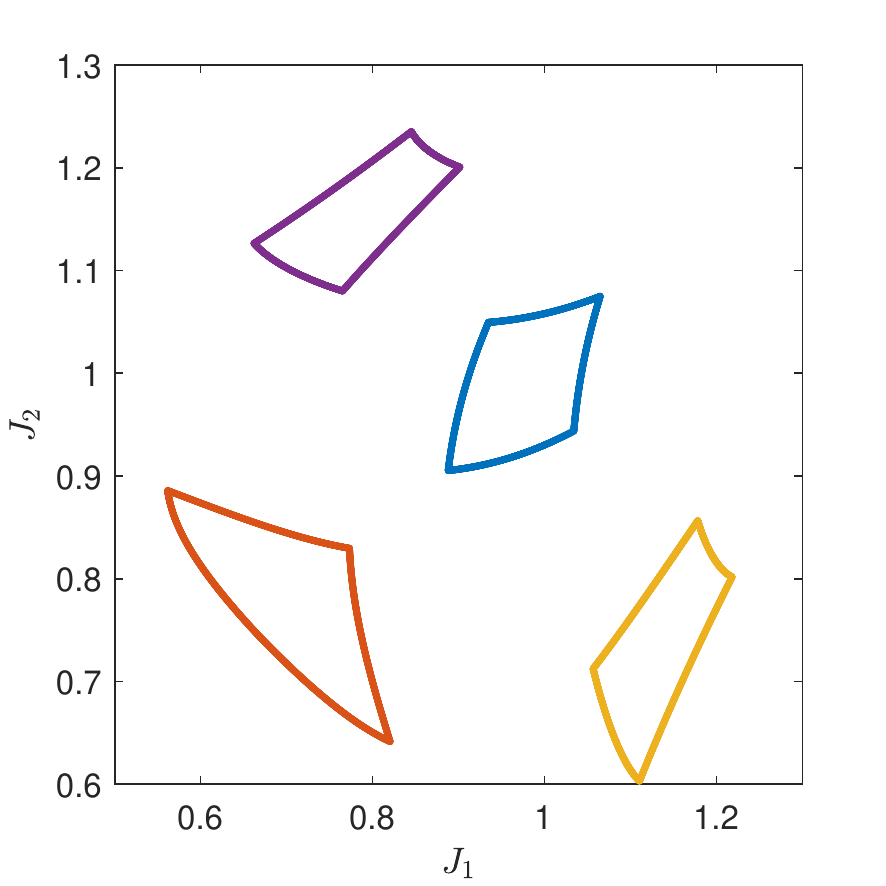}
                \caption{$\hat J_Z(x_0,u)$ for all $u \in U$.}
                \label{fig:ex5s2}
            \end{subfigure}\\
            \begin{subfigure}[t]{.4\textwidth}
                \includegraphics[width=\columnwidth]{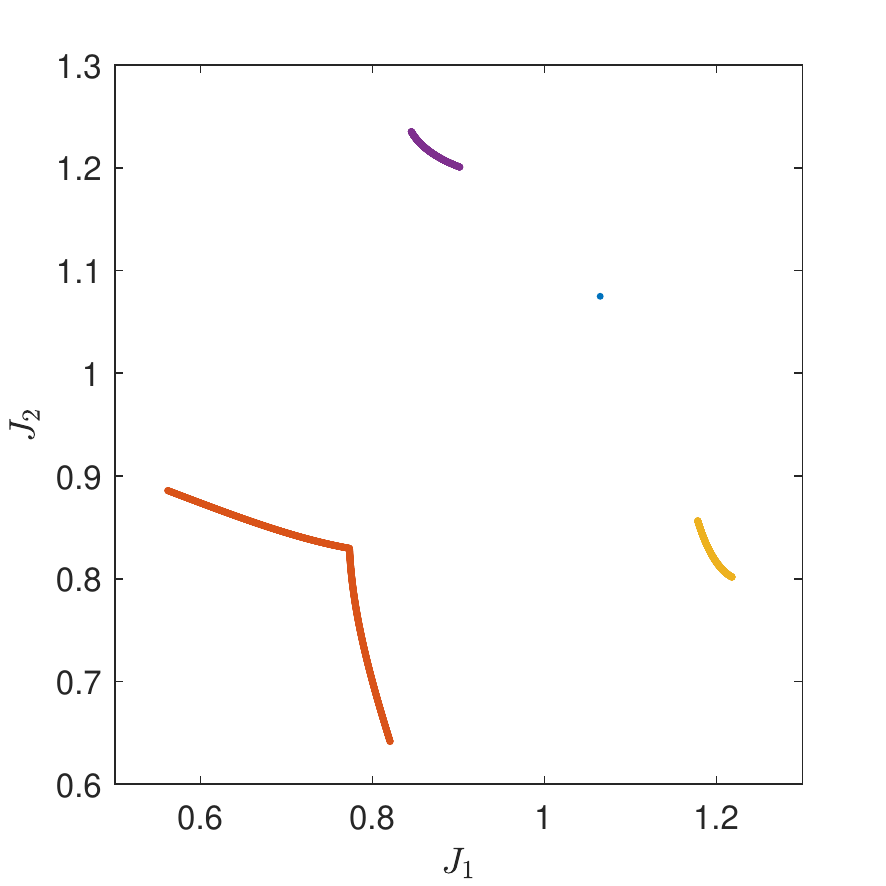}
                \caption{$\sup \hat J_Z(x_0,u)$.}
                \label{fig:ex5s3}
            \end{subfigure}
        	\hfil
            \begin{subfigure}[t]{.4\textwidth}
                \includegraphics[width=\columnwidth]{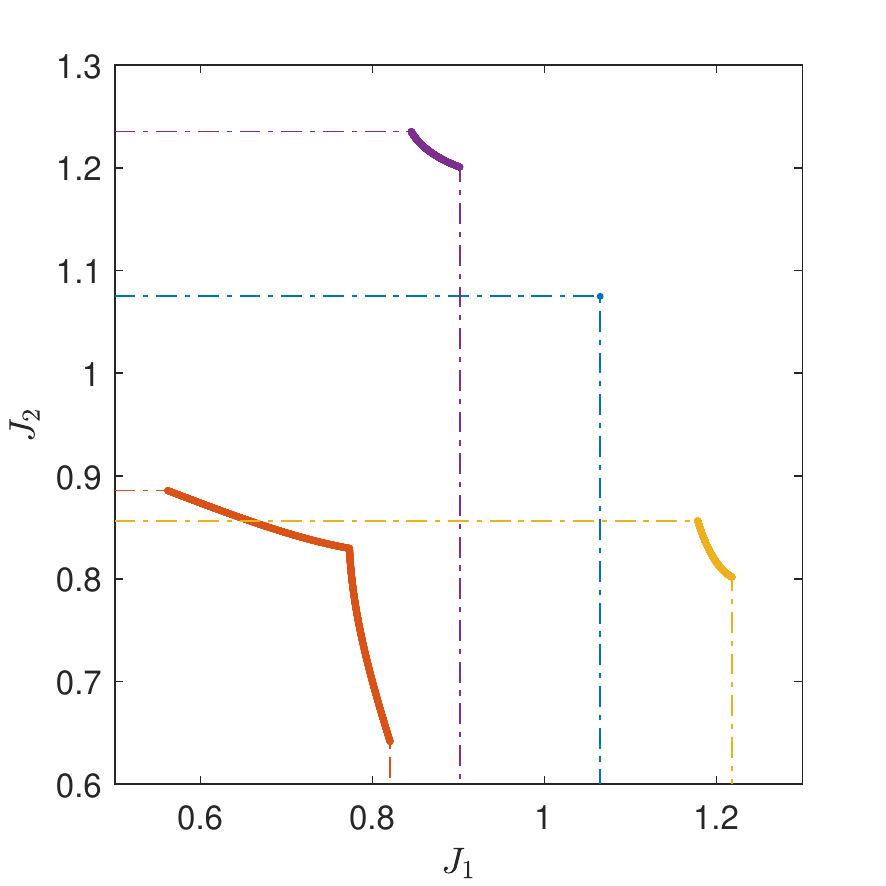}
                \caption{$\hat J_Z(x_0,u)-\mathbb{R}^2_{\succeq}$.}
                \label{fig:ex5s4}
            \end{subfigure}
            \caption{Example of minmax robust efficient solutions of Problem \eqref{eq:lss25}. In all cases, the colors identify the same solutions.}
            \label{fig:ex5}
        \end{figure}
    
    \subsection{Model predictive control}
    
    Model predictive control (MPC, see \cite{grune2017nonlinear} for a detailed introduction) is a very popular and highly flexible framework to construct a feedback control law using a model of the system dynamics. In order to obtain a feedback signal, the model-based open loop problem \eqref{eq:smocp} is solved on a finite-time \emph{prediction horizon} of length $t_p$. This means that we set $t_e = t_0 + t_p$. Then a small part of length $t_c \leq t_p$ is applied to the real system, and the problem has to be solved again on a shifted time horizon, i.e., for $t_0 = t_0 + t_c$ and $t_e = t_e + t_p$.
    Several extensions to multiple objectives have been presented \cite{PD18}. Well-known approaches are the \emph{weighted sum method} \cite{BP09} or \emph{reference point tracking} \cite{ZF12}.
    
    The MPC framework allows for easy incorporation of nonlinear dynamics as well as constraints. However, the solution has to be obtained within the control horizon $t_c$, which is particularly challenging in the presence of multiple objectives.
    A remedy to this issue is \emph{explicit MPC} where -- instead of solving the control problem online -- the optimal input is selected from a library of optimal inputs which was computed in an offline phase for all possible initial conditions. In the linear-quadratic case\cite{bemporad2002explicit}, this library is computed using multi-parametric programming. The explicit solution is exact for all inputs even though only a finite number of problems has to be solved. As the number of problems can quickly become prohibitively expensive, an extension to exploit symmetries has been proposed in Danielson and Borrelli\cite{DB12} for linear-quadratic problems.
    
    If the dynamics are nonlinear, then the explicit MPC procedure requires interpolation between different library entries. In this situation, the optimal inputs cannot be characterized by polygons as before. Instead, the space of initial conditions is discretized, and an optimal control problem is solved for each initial condition within this grid \cite{Joh02,BF06}. To obtain controls for intermediate values, interpolation is used, which results in an additional error that has to be taken into account. In \cite{ober2018explicit}, a very similar path was taken for multi-objective optimal control problems, where the interpolation is performed between elements of Pareto sets.
    
    \section{Symmetries in uMOCP}
    Symmetry identification in dynamical systems and optimal control can help to simplify the problem at hand, as invariances in a system lead to several (possibly infinitely many) identical solutions. These can then be replaced by one representative solution, which leads to faster and more efficient numerical computations, as the representative has to be computed only once. This concept was used in the control context in \cite{frazzoli2001robust,frazzoli2005maneuver}, where solutions of optimal control problems are obtained very efficiently by combining elements from a pre-computed library of so-called \emph{motion primitives}. Inspired by the concept of motion primitives, symmetries in dynamical control systems were exploited in \cite{ober2018explicit} to design explicit MPC algorithms for nonlinear MOCPs to reduce the computational effort to solve the problem.
    
    In the following, we extend this last study on symmetries in MOCPs to problems with uncertainty in the initial state. Formally, we describe symmetries by a finite-dimensional Lie group G and its group action $\psi:\mathbb{R}^{n_x}\times G \rightarrow \mathbb{R}^{n_x}$. For each $g\in G$, we denote by $\psi_g:\mathbb{R}^{n_x} \rightarrow \mathbb{R}^{n_x}$ the diffeomorphism defined by $\psi_g:= \psi(\cdot,g)$.
    
    We want to identify efficient solutions to Problem \eqref{eq:usmocp} that remain efficient when the initial conditions are transformed by the symmetry group action such that
    \begin{equation}
        \label{eq:symre}
        \arg\min\limits_u\sup_\alpha \hat J(x_0+\alpha, u) = \arg\min\limits_u\sup_\alpha \hat J(\psi_g(x_0+\alpha), u) \;\; \forall g\in G.
    \end{equation}
    
    The following theorem provides conditions under which Equation \eqref{eq:symre} holds.
    
    \begin{theorem}[Symmetry of uMOCP]
        \label{thm:sumocp}
        Let $\mathcal{X} = W^{1,\infty}([t_0,t_e],\mathbb{R}^{n_x})$ and $\mathcal{U} = L^{\infty}([t_0,t_e],\mathbb{R}^{n_u})$. If
        \begin{enumerate}
         \item the dynamics are invariant under the Lie Group action $\psi$, i.e. $\psi_g(\varphi_u(x_0,t)) = \varphi_u(\psi_g(x_0),t)$ for all $g\in G, x_0 \in \mathbb{R}^{n_x}$, $t\in [t_0,t_e]$ and $u\in\mathcal{U}$;
         \item there exist $\eta,\beta,\delta \in \mathbb{R}, \eta \neq 0$, such that the cost functions $C_i$ and the Mayer terms $\Phi_i,i=1,\ldots , k$, are invariant under the Lie Group action $\psi$ up to linear transformation, i.e.,
         \begin{equation}\label{eq:inv_cost}
            C_i(\psi_g(x),u) = \eta C_i(x,u)+\beta
         \end{equation}
         and
         \begin{equation}\label{eq:inv_Mayer}
            \Phi_i(\psi_g(x_e)) = \eta\Phi(x_e)+\delta\; \text{ for } i=1,\ldots,k;
         \end{equation}
         \item the constraints $g_i, i=1,\ldots,l$ are invariant under the Lie Group action $\psi$, i.e.,
         \begin{equation}\label{eq:inv_con}
            \begin{split}
                g_i(\psi_g(x),u) = g_i(x,u)\; \text{ for } i=1,\ldots,l,\\
            \end{split}
         \end{equation}

          then we have
          \begin{equation}
            \arg\min\limits_u\sup_\alpha \hat J(x_0+\alpha, u) = \arg\min\limits_u\sup_\alpha \hat J(\psi_g(x_0+\alpha), u) \; \forall g\in G.
            \label{eq:prf}
          \end{equation}
        \end{enumerate}
        
        We say that problem (uMOCP) is invariant under the Lie group action $\psi_g$, or equivalently, $G$ is a symmetry group for problem (uMOCP).

        \begin{proof}
            In order to prove the theorem, we will first show feasibility and then optimality. In both cases, a key component of the proof is the fact that $x_0 + \alpha \in \mathbb{R}^{n_x}$. Thus, the invariance property of the dynamics also hold for uncertainty in the initial condition as
            \begin{equation}\label{eq:inv_unc}
            \psi_g(\varphi_u(x_0+\alpha,t)) = \varphi_u(\psi_g(x_0+\alpha),t) 
            \end{equation}
            for all $g\in G, x_0+\alpha \in \mathbb{R}^{n_x}$, $t\in [t_0,t_e]$ and $u\in\mathcal{U}$.
            
            \underline{Feasibility.} Let $u$ be a feasible curve of problem (uMOCP) and let $\varphi_u(x_0+\alpha,t)$ be the solution of the initial value problem. We now consider problem (uMOCP) with initial value $\psi_g(x_0+\alpha)$. Substituting $u$ into the inequality constraints of the transformed (MOCP):
            \begin{equation}
                \begin{split}
                    \hat g_i(\psi_g(x_0+\alpha), \mathfrak{u}) &= g_i(\varphi_u(\psi_g(x_0+\alpha), t), u(t)) \\
                    &\stackrel{\eqref{eq:inv_unc}}{=} g_i(\psi_g(\varphi_u(x_0+\alpha, t)), u(t)) \\
                    &\stackrel{\eqref{eq:inv_con}}{=} g_i(\varphi_u(x_0+\alpha, t), u(t)) \\
                    &= \hat g_i(x_0+\alpha, \mathfrak{u}) \leq 0 \\
                \end{split}
            \end{equation}
            for $i=1,\ldots,l$ and for all $\alpha \in \mathcal{Z}$.
            
            \underline{Optimality.} First, we prove that solutions to the supremum problem are invariant the under group actions on initial conditions. Then, we prove the same follows for the minimization problem. Assume the maximum exists for all $u\in \mathcal{U}$ in Equation (\ref{eq:prf}), let $\alpha \in \arg\max_{\alpha} \hat J(x_0+\alpha, u)$, and assume there exists an $\tilde\alpha$ such that
            \begin{equation}
                \begin{split}                                 
                    &\hat J(\psi_g(x_0+\tilde\alpha), u) > \hat J(\psi_g(x_0+\alpha),u) \\
                    &\stackrel{\eqref{eq:inv_unc}}{\Leftrightarrow} J(\varphi_u(\psi_g(x_0+\tilde\alpha),\cdot),u) >
                    J(\varphi_u(\psi_g(x_0+\alpha),\cdot),u) \\
                    &\stackrel{\eqref{eq:inv_cost},\eqref{eq:inv_Mayer}}{\Leftrightarrow} J(\varphi_u(x_0+\tilde\alpha),u) >
                    J(\varphi_u(x_0+\alpha),u) \\
                    &\Leftrightarrow \hat J(x_0+\tilde\alpha,u) > \hat J(x_0+\alpha,u)
                \end{split}
            \end{equation}

            This is a contradiction to $\alpha \in \arg\max_\alpha \hat J(x_0+\alpha,u)$. Thus, the Pareto set for the maximization problem with initial values $x_0+\alpha$ and $\psi(x_0+\alpha)$ are identical. 
            
            Next, let $u \in \arg\min_{u} \max_{\alpha}\hat J(x_0+\alpha, u)$, and assume there exists an $\tilde u$ such that
            \begin{equation}
                \begin{split}                                 
                    &\max\limits_{\alpha} \hat J(\psi_g(x_0+\alpha), \tilde u) \subseteq \max\limits_{\alpha} \hat J(\psi_g(x_0+\alpha),u) - \mathbb{R}^k \\
                    &\stackrel{\eqref{eq:inv_unc}}{\Leftrightarrow} \max\limits_{\alpha} J(\varphi_u(\psi_g(x_0+\alpha),\cdot),\tilde u) \subseteq 
                    \max\limits_{\alpha} J(\varphi_u(\psi_g(x_0+\alpha),\cdot),u)  - \mathbb{R}^k \\
                    &\stackrel{\eqref{eq:inv_cost},\eqref{eq:inv_Mayer}}{\Leftrightarrow} \max\limits_{\alpha} J(\varphi_u(x_0+\alpha),\tilde u) \subseteq 
                    \max\limits_{\alpha} J(\varphi_u(x_0+\alpha),u)  - \mathbb{R}^k \\
                    &\Leftrightarrow \max\limits_{\alpha} \hat J(x_0+\alpha,\tilde u) \subseteq  \max\limits_{\alpha} \hat J(x_0+\alpha,u)  - \mathbb{R}^k \\
                \end{split}
            \end{equation}            
            This is a contradiction to $u \in \arg\min_{u} \max_{\alpha}\hat J(x_0+\alpha, u)$.  Thus, it follows that the efficient set is invariant under group actions on initial conditions.                        
        \end{proof}
    \end{theorem}
    
    By Theorem \ref{thm:sumocp}, efficient sets are valid in multiple situations. Thus, identifying symmetries in the uMOCP allows reducing the search space and the computational effort since one only needs to compute one representative efficient set.
    
    Further, if the group action acts in the same way as the uncertainty, that is, $\psi_g(x_0+\alpha) := x_0 + \alpha + g$, $g \in \mathbb{R}^{n_x}$ such that $\alpha + g \in \mathcal{Z}$ then there exists a relationship with other definitions of robustness for MOCPs. As shown in \cite{ober2018explicit}, the Pareto sets with initial conditions $x_0$ and $\psi_g(x_0)$ are identical for all $g \in G$. This leads to what is known as highly robust efficiency \cite{ide2016robustness}.
    
    \begin{definition}[Highly robust efficiency]
        Given an uncertain multi-objective optimization problem $\mathcal{P}(\mathcal{Z})$, a feasible curve $\bar u \in \mathcal{U}$ is called highly robust efficient for $\mathcal{P}(\mathcal{Z})$ if it is efficient for $\mathcal{P}(\alpha)$ for all $\alpha \in \mathcal{Z}$.
    \end{definition}
    
    This means that any Pareto solution of $\mathcal{P}(\alpha)$ for $\alpha \in \mathcal{Z}$ will remain optimal in the Pareto sense regardless of the presence of uncertainty. Thus, if one is interested in highly robust efficiency and the group action acts in the same way as the uncertainty, it suffices to solve one representative MOCP leading to further computational savings.
    
    % Removing the inner maximization problem
    % Additionally, if the group action acts in the same way as the uncertainty then
    % \begin{equation}
    %     \begin{split}
    %         \arg\min\limits_u\sup_\alpha \hat J(\psi_g(x_0+\alpha), u) 
    %         &\Leftrightarrow \arg\min\limits_u \max\limits_\alpha\hat J(\psi_g(x_0), u) \\
    %         &\Leftrightarrow \arg\min\limits_u \hat J(\psi_g(x_0), u) \\
    %     \end{split}
    % \end{equation}
    % for all $g\in G$. Since, $\sup_\alpha$ would be invariant under the group action $G$ and thus have the same optimal $\alpha$ for all $u$. This means that we can remove the inner maximization problem, which significantly speedups the computations.
    
\section{The Method}    

    There exist two fundamentally different -- namely indirect and direct -- approaches to solve optimal control problems. In a direct approach, a discretization is introduced for both the state and the input,
    by which the optimal control problem is transformed into a high- yet finite-dimensional MOP \cite{logist2010efficient,ober2012solving} such that methods from multi-objective optimization can be used.
    In the following, we present a classification of solution methods for MOPs based on when the decision-maker participates in expressing his/her preferences \cite{hwang1979multiple}.
    
    \begin{itemize}
        \item {\bf A priori:}
        the decision maker has to define the preferences of the objective functions before starting the search \cite{Zadeh63,Bowman76,wierzbicki1980use,das:98}.
        \item {\bf A posteriori:}
        first, an approximation of the entire Pareto set and front is generated. It is then presented to the decision maker, who selects the most preferred compromise according to her/his preferences \cite{dsh:05,deb:01,Coello:07,hernandez2017global}.
        \item {\bf Interactive:}
        both the optimizer and decision maker work progressively. The optimizer produces feasible points and the decision maker provides preference information \cite{miettinen1995interactive,schutze2019pareto}.
    \end{itemize}
    
    In this work, we investigate a hybrid approach. First, we compute an approximation to the efficient set of the family of uMOCP in an offline fashion. Then, at each step of the simulation, an algorithm selects an adequate, efficient feasible curve for the given context with the help of the efficient solutions computed offline. Figure \ref{fig:bd} shows the general flow for the simulation. There exist two main blocks, the first one is the online phase which optimizes the control strategy given the initial conditions (subject to uncertainty). Note that the online phase can make use of the library of solutions computed offline. The second block is the model that receives the optimized control strategy and returns the initial conditions for the next iteration. In the following, we describe in detail the offline and online phases.
    
    \begin{figure}
    	\centering
        \includegraphics[trim={1cm 15.5cm 4cm 1cm},clip,width=.8\textwidth]{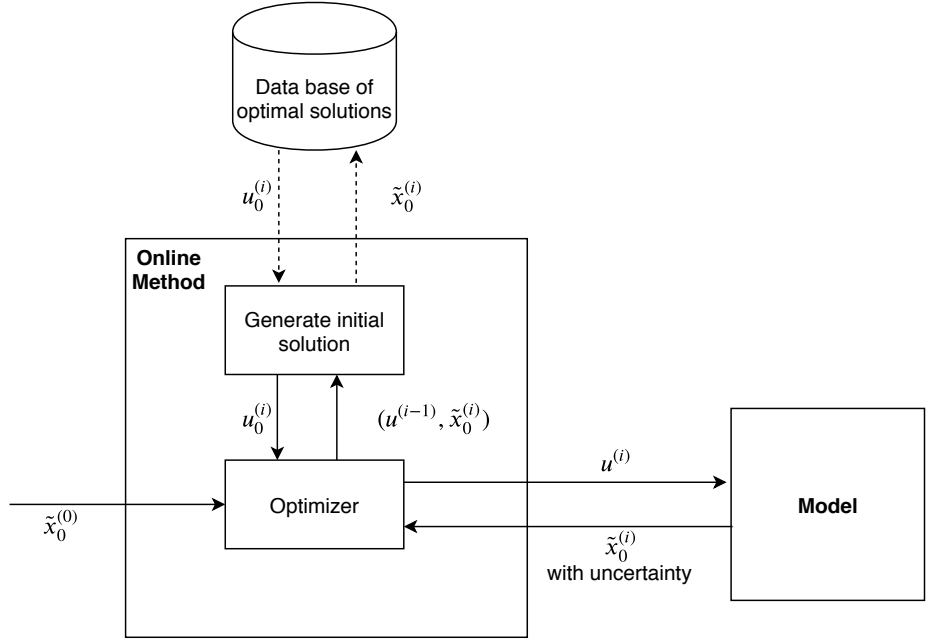}
        \caption{Block diagram of the simulation. Given the initial conditions $\tilde x_0$ (after exploiting the symmetry group $G$) at iteration $i$, the online method approximates the optimal control strategy. This process can be with or without the information from the offline phase. The control strategy is applied to the  model over the control horizon and  returns the initial conditions for the next iteration.}
        \label{fig:bd}
    \end{figure}    
    
    \subsection{Offline phase}
        From the formulation of the uMOCP, it can be observed that the problem is parameterized with respect to the initial value $x_0 \in  \mathcal{Y}$, i.e., we have to solve a so-called multi-objective parametric optimization problem. For this class of problems, a feasible curve $\bar u \in \mathcal{U}$ is called efficient for $\mathcal{P}(\mathcal{Y})$ if it is efficient for $\mathcal{P}(x_0)$ for at least one $x_0 \in \mathcal{Y}$. 
        
        To avoid solving infinitely many uMOCPs, we introduce a discretization of the set $\mathcal{Y}$ and solve the resulting uMOCPs. A potential drawback of the approach is that the number of problems increases exponentially with the state dimension $n_x$. However, by exploiting symmetries in the problem, the number of uMOCPs to solve by be significantly reduced. This reduction is a consequence of the change in the dimension of the initial conditions from $x_0 \in \mathbb{R}^{n_x}$ to $\tilde x_0 \in \mathbb{R}^{\tilde n_x}$ where $\tilde n_x = n_x - dim(G)$.
        
        Algorithm \ref{alg:op} shows the steps to follow to solve the family of uMOCPs. This algorithm follows the framework proposed in \cite{ober2018explicit}.        
        
        \begin{algorithm}
            \caption{Offline phase}
            \label{alg:op}
            \begin{algorithmic}[1]
                \Require Lower and upper bounds $x_{0,\min}, x_{0,\max} \in \mathbb{R}^{n_x}$.
                \State Dimension reduction: decrease dimension of the parameter $x_0 \in \mathbb{R}^{n_x}$, to $\tilde x_0\in \mathbb{R}^{\tilde n_x}$ by exploiting the symmetry group $G$.
                \State Construction of library: create an $\tilde n_x$-dimensional grid $\mathcal{L}$ for the parameter $\tilde x_0$ between $\tilde x_{0,\min}$ and $\tilde x_{0,\max}$ with $\delta_i$ points in the $i^{th}$ direction. This results in $N = \prod_{i=1}^{\tilde n_x} \delta_i$ parameters.
                \State Compute the efficient sets $\mathcal{R}_{\tilde n_x}$ for all $\tilde n_x\in \mathcal{L}$
            \end{algorithmic}
        \end{algorithm}
        
        Algorithm \ref{alg:op} is highly parallelizable. The maximum possible speedup of a parallel program to its sequential counterpart according to Amdahl's law \cite{Amdahl:1967} is given by
        \begin{equation}
            S_{A1}(n_c)=\frac{1}{\frac{1}{n_c}+\frac{1}{n_c}(1-\frac{1}{n_c})},
        \end{equation}
        where $n_c$ denotes the number of cores. This is the case, since typically the number of subprobles is much larger than the number of core available. Thus, the maximum acceleration is be given by the number of cores $n_c$.
        
        In Algorithm \ref{alg:op} many sets of efficient solutions in the sense of minmax robustness of the uMOCP need to be computed. We here use the generic stochastic search algorithm \cite{LTDZ2002b}, see Algorithm \ref{alg:generic_emo}. The algorithm starts by generating a random sample from $\mathbb{R}^{n_u}$ ($P_0$) and then filters those points that are efficient with respect to the set $P_0$. Then, in each iteration $j$, the algorithm generates new feasible points from those in the archive $A_j$ through evolutionary operators \cite{deb:01}. In the next iteration, the new feasible points are again filtered, and the efficient solutions are stored in the archive $A_{j+1}$.
                        
        \begin{algorithm}
            \caption {Generic Stochastic Search Algorithm}
            \label{alg:generic_emo}
            \begin{algorithmic}[1]
            \State $P_0\subset \mathbb{R}^{n_u}$ drawn at random
            \State $A_0 = ArchiveUpdate\mathcal{R}(P_0,\emptyset)$
            \For{$j=0,1,2, \ldots$}
            \State $P_{j+1} = Generate (A_j)$
            \State $A_{j+1} = ArchiveUpdate\mathcal{R} (P_{j+1}, A_j)$
            \EndFor
            \end{algorithmic}
        \end{algorithm}        
        
        The update of the archive $A_j$ is realized using the function $ArchiveUpdate\mathcal{R}$ (cf.~Algorithm \ref{alg:PQe1} for the pseudo code), which computes all efficient solutions in a probabilistic sense.
        %In the following, we present an archiver that aims to compute all efficient solutions in a probabilist sense. Algorithm \ref{alg:PQe1} shows the pseudocode of $ArchiveUpdate\mathcal{R}$. 
        The archiver will go through all candidate points, and it will add a candidate $p \in P$ if there is no solution $a\in A$ in the archiver that dominates $p$. Further, a solution in the archiver will be removed if there is a candidate solution $p$ that dominates $a$. The complexity in terms of the number of the comparisons is given by $O(|A||P|)$.

        \begin{algorithm}
        \caption {$A := ArchiveUpdate\mathcal{R}\; (P, A_0)$}
        \label{alg:PQe1}
        \begin{algorithmic}[1]
        \Require{population $P$, archive $A_0$}
        \Ensure{updated archive $A$}
        \State $A := A_0$
        \ForAll{$p\in P$}
        \If{$\mbox{$\nexists$}a\in A: \hat J_\mathcal{Z}(\tilde x_0, a) \subseteq \hat J_\mathcal{Z}(\tilde x_0, p) - \mathbb{R}^k_\succeq$ } 
        \State $A:=A\cup \{p\}$
        \EndIf
        \ForAll{$a\in A$} 
        \If{$\hat J_\mathcal{Z}(\tilde x_0, p) \subseteq \hat J_\mathcal{Z}(\tilde x_0, a) - \mathbb{R}^k_\succeq$}
        \State $A:=A\backslash\{a\}$
        \EndIf
        \EndFor
        \EndFor
        \end{algorithmic}
        \end{algorithm}                
        
        If we assume Algorithm \ref{alg:generic_emo} as a generator, then it is possible to prove that the archiver does not cycle or deteriorate in their entries \cite{Hanne1999}.
                
        \begin{theorem}
        \label{thm:monotonicity1}
        Let $l\in\mathbb{N}$, $P_1,\ldots,P_l\subset\mathbb{R}^{n_x}$ be finite sets, and $A_i, i=1,\ldots ,l$, be obtained by $ArchiveUpdate\mathcal{R}$ using a generator, and $C_l = \bigcup_{i=1}^l P_i$. Then
        \begin{equation}
        A_l = \{x\in C_l\; : \; \nexists y\in C_l: y\preceq x\}.
        \end{equation}
        \end{theorem}
        \begin{proof}
            Let $C_l = \{u_1,\ldots,u_m\}$, where the $u_i$'s are considered by the archiver in this order. Let $u\in C_l$, then there exists an index $j\in\{1,\ldots,m\}$ such that $u=u_j$. First, let $u\in A_l$. Then $u$ will be added to the archive in the $j$-th step and not discarded further on (line 3 respectively line 7 of Algorithm \ref{alg:PQe1}). Next, let $u\not\in A_l$. That is, there exists a point $u_i\in P_{C_l}$ such that $\hat J(\tilde x_0+\alpha,u_i)\subseteq \hat J(\tilde x_0+\alpha,u) - \mathbb{R}^{k}$. Hence, $u$ is either not added to the archive, or discarded if added previously.
        \end{proof}

        Next, we investigate the limit behavior of the sequence $A_i$ of archives. To guarantee convergence, we have to assume the following (see also \cite{schuetze_ecj:10,schutze2019pareto}): 
        \begin{equation}
        \label{eq:P=1} 
        \forall u\in U \;\mbox{and}\; \forall \delta > 0:
        \quad
        P\left(\exists l\in\mathbb{N}\; : \; P_l\cap B_\delta(u)\cap U\neq 
        \emptyset\right) = 1,
        \end{equation}
        where $P(A)$ denotes the probability for event $A$ and $B_\delta(u)$ the $n$-dimensional sphere with center $u$ and radius $\delta$. The following theorem shows that the sequence of archives converges under these conditions with probability one to $\mathcal{R}$ in the Hausdorff sense (see Definition \ref{def:hd}). 

        \begin{theorem}
        \label{thm:convergence}
        Let Problem (\ref{eq:smocp}) be given, where $\hat J$ is continuous and $\mathbb{R}^{n_x}$ is compact. Further, let Assumption (\ref{eq:P=1}) be fulfilled. Then, using a generator where $ArchiveUpdate\mathcal{R}$ is used to update the archive, leads to a sequence of archives $A_l, l\in\mathbb{N}$, with
        \begin{equation}
        \lim_{l\to\infty} d_H(\mathcal{R},A_l) = 0,\quad \mbox{with probability one.}
        \end{equation}
        \end{theorem}
        \begin{proof}
            First, we show that $dist(A_l,\mathcal{R}))\to 0$ with probability one for $l\to\infty$. We have to show that every $u\in \mathbb{R}^{n_x} \backslash \mathcal{R}$ will be discarded (if added before) from the archive after finitely many steps, and that this point will never be added further on. Let $u\in \mathbb{R}^{n_x}\backslash \mathcal{R}$. Since there exists $p\in \mathcal{R}$ with $\hat J(\tilde x_0,p) \subseteq \hat J(\tilde x_0,u) - \mathbb{R}^k$. Further, since $\hat J$ is continuous there exists a neighborhood $U$  of $u$ with $\hat J(\tilde x_0,p) \subseteq \hat J(\tilde x_0,u)$, $\forall u\in U$. By assumption, there exists a number $l_0\in\mathbb{N}$ such that there exists $u_{l_0}\in P_{l_0}\cap \mathbb{R}^{n_x}$. Thus, by construction of $ArchiveUpdate\mathcal{R}$, the point $u$ will be discarded from the archive if it is a member of $A_{l_0}$, and never be added to the archive further on.
            
            Now we consider the limit behavior of $dist(\mathcal{R},A_l)$. Let $\bar{p}\in \mathcal{R}$. For $i\in\mathbb{N}$ there exists by assumption, a number $l_i$ and a point $p_i\in P_{l_i}\cap B_{1/i}(\bar{p})\cap \mathbb{R}^{n_x}$, where $B_\delta(p)$ denotes the open ball with center $p$ and radius $\delta\in\mathbb{R}_+$. Since $\lim_{i\to\infty} p_i = \bar{p}$ and since $\bar{p}\in \mathcal{R}$ there exists an $i_0\in\mathbb{N}$ such that $p_i\in \mathcal{R}$ for all $i\geq i_0$. By construction of $ArchiveUpdate\mathcal{R}$, all the points $p_i, i\geq i_0$, will be added to the archive (and never discarded further on). Thus, $dist(\bar{p},A_l) \to 0$ for $l\to\infty$.
        \end{proof}                
        
        We now investigate the performance of Algorithm \ref{alg:op} using the following example problem with $J:\mathbb{R} \times \mathbb{R}^2 \rightarrow \mathbb{R}^2$: 
        \begin{equation}
            \label{eq:witting}
            %\min\limits_{u} J(\alpha,u) = 
                \min\limits_{u}
                \begin{pmatrix}
                    \frac{1}{2}(\sqrt{1+(u_1+u_2)^2} + \sqrt{1+(u_1-u_2)^2} + u_1 - u_2) + \alpha e^{-(u_1-u_2)^2} \\
                    \frac{1}{2}(\sqrt{1+(u_1+u_2)^2} + \sqrt{1+(u_1-u_2)^2} - u_1 + u_2) + \alpha e^{-(u_1-u_2)^2}               
                \end{pmatrix},
        \end{equation}        
        where $-2 \leq u_1,u_2 \leq 2$ and $0.1\leq \alpha \leq 1.5$. Note that for $0 < \alpha < 1$ the problem's Pareto set is invariant under translations in $\alpha$ \cite{ober2018explicit}. Thus, the solutions in the Pareto set are highly robust efficient as well as set-based minmax robust efficient \cite{ide2016robustness}.
        However, this is not longer the case for $1\leq \alpha \leq 1.5$. Figure \ref{fig:exWit} shows the efficient set and its image for different values of $\alpha$.
        
        \begin{figure}
        	\centering
            \includegraphics[width=.4\textwidth]{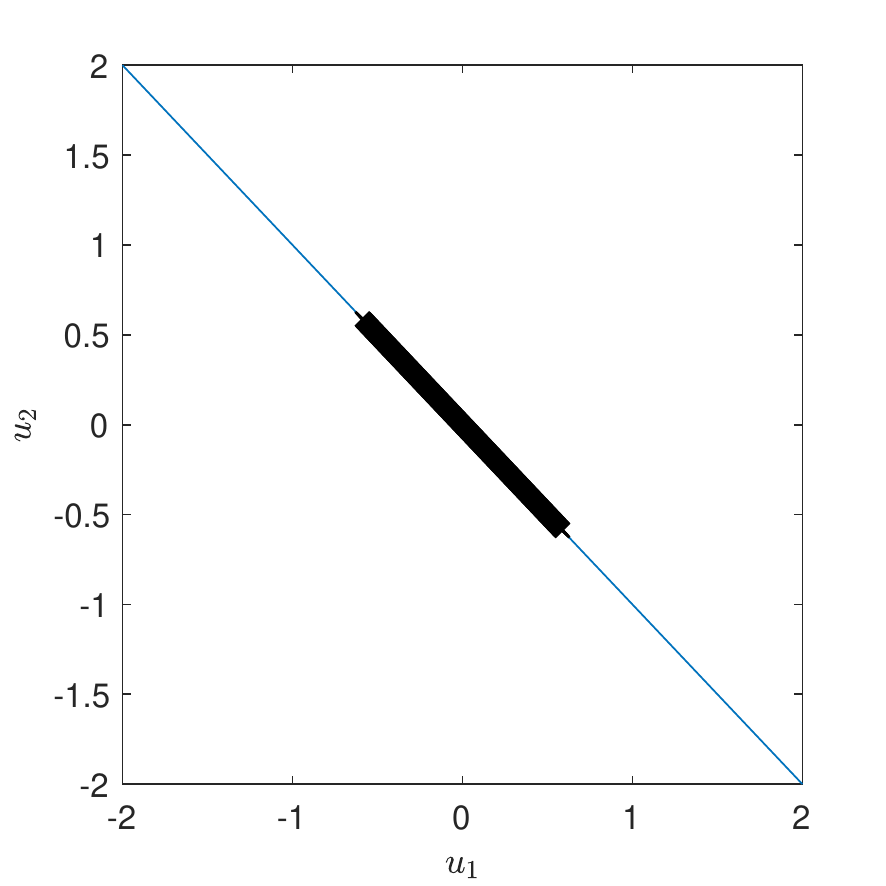} \hfil
            \includegraphics[width=.4\textwidth]{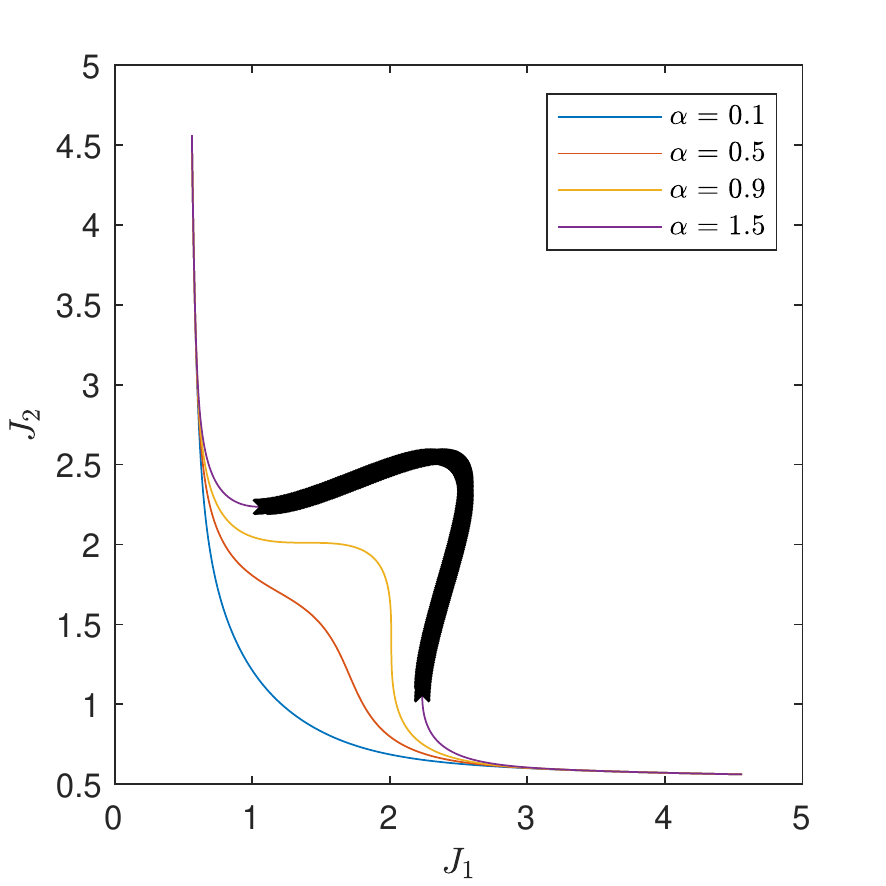}
            \caption{Invariant Pareto set (left) and its image for Problem \eqref{eq:witting} with $\alpha \in (0.1, 0.5, 0.9)$. Note that for $\alpha = 1.5$ the region in bold it is no longer optimal.}
            \label{fig:exWit}
        \end{figure}
    
		\begin{figure}
			\centering
			\begin{subfigure}[t]{.40\textwidth}                         
				\includegraphics[width=\textwidth]{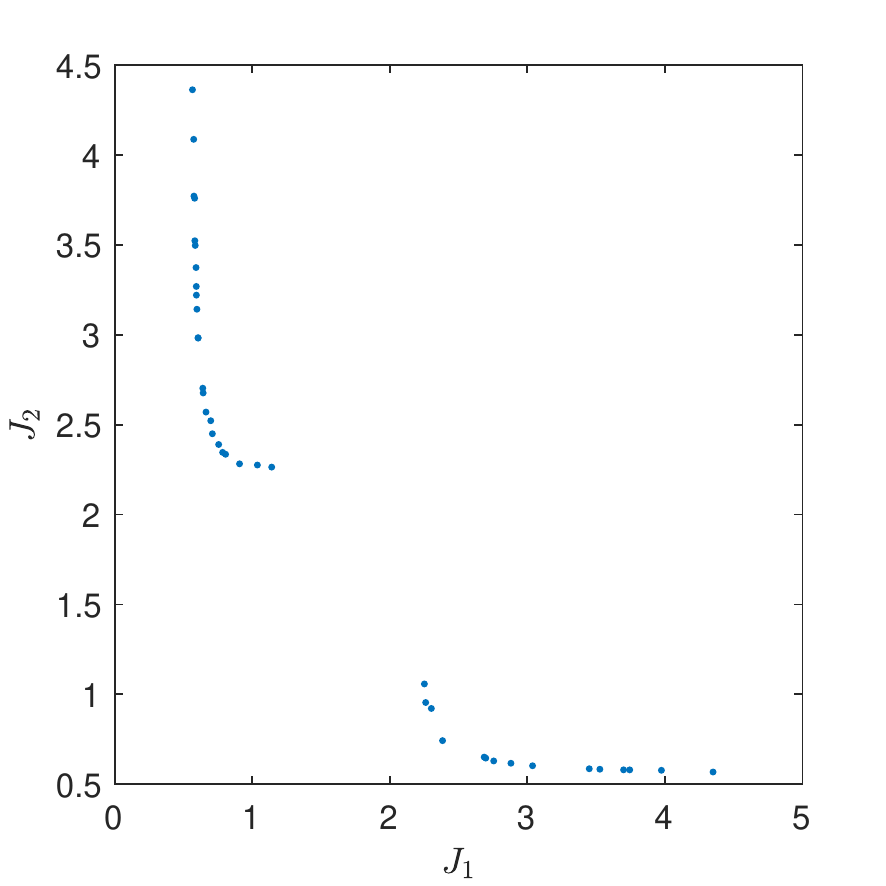}
				\caption{$500$ iterations}
				\label{fig:exA1}
			\end{subfigure}\hfil
			\begin{subfigure}[t]{.40\textwidth}
				\includegraphics[width=\textwidth]{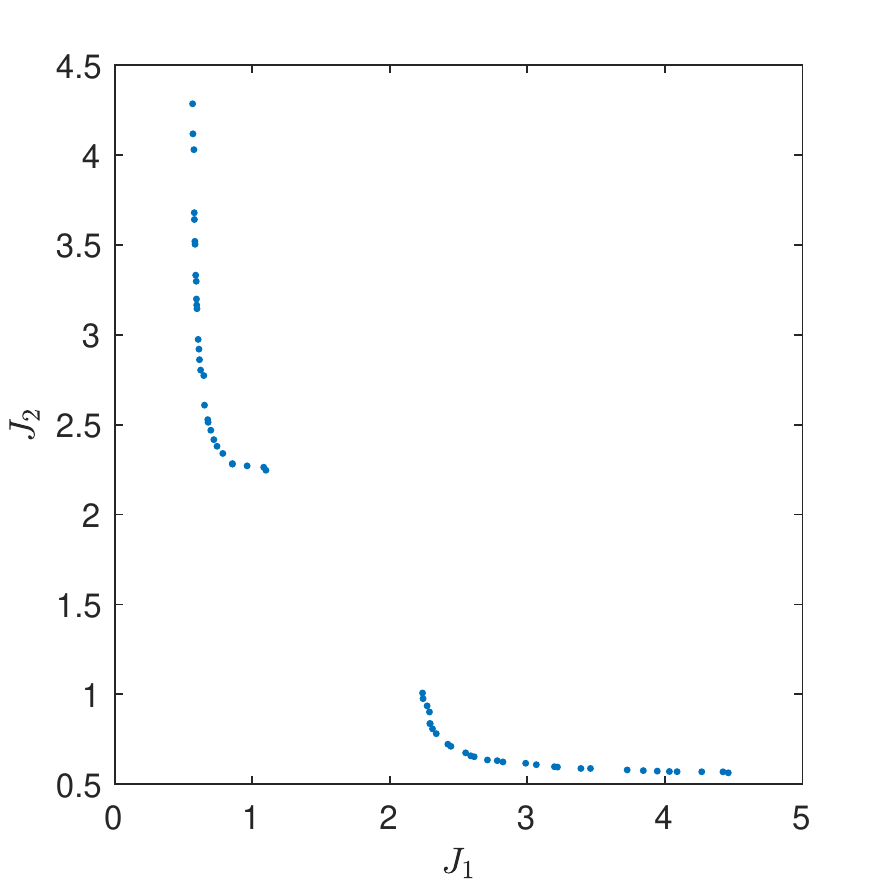}        
				\caption{$1,000$ iterations}
				\label{fig:exA2}
			\end{subfigure}\\
			\begin{subfigure}[t]{.40\textwidth}
				\includegraphics[width=\textwidth]{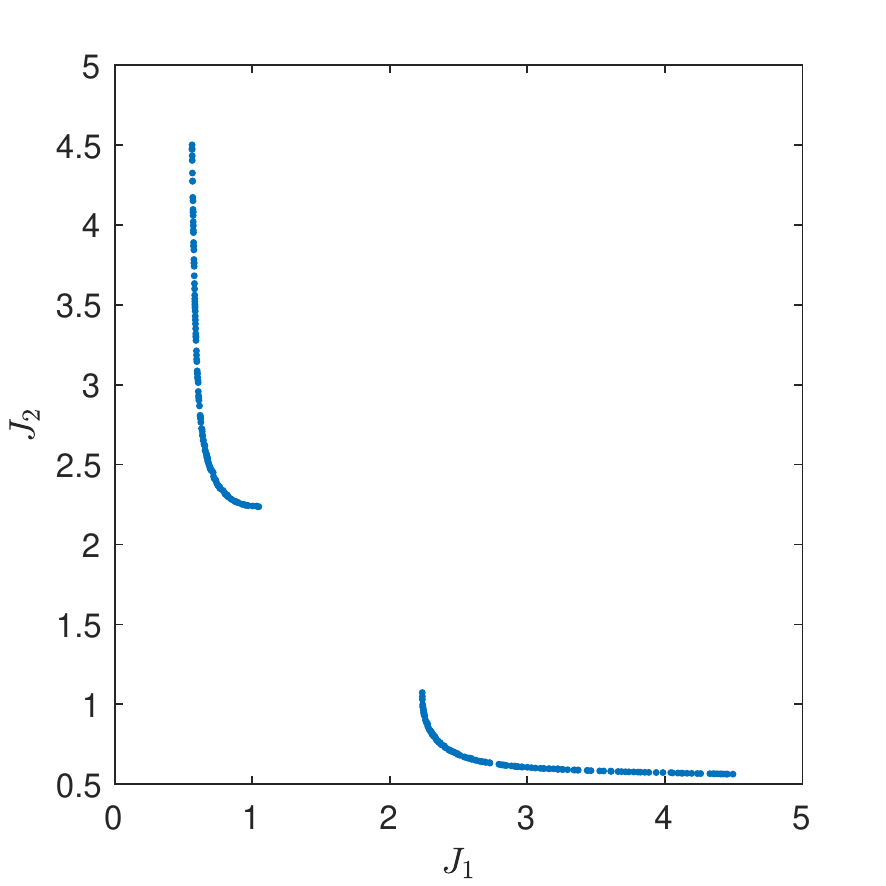}
				\caption{$10,000$ iterations}
				\label{fig:exA3}
			\end{subfigure}\hfil
			\begin{subfigure}[t]{.40\textwidth}
				\includegraphics[width=\textwidth]{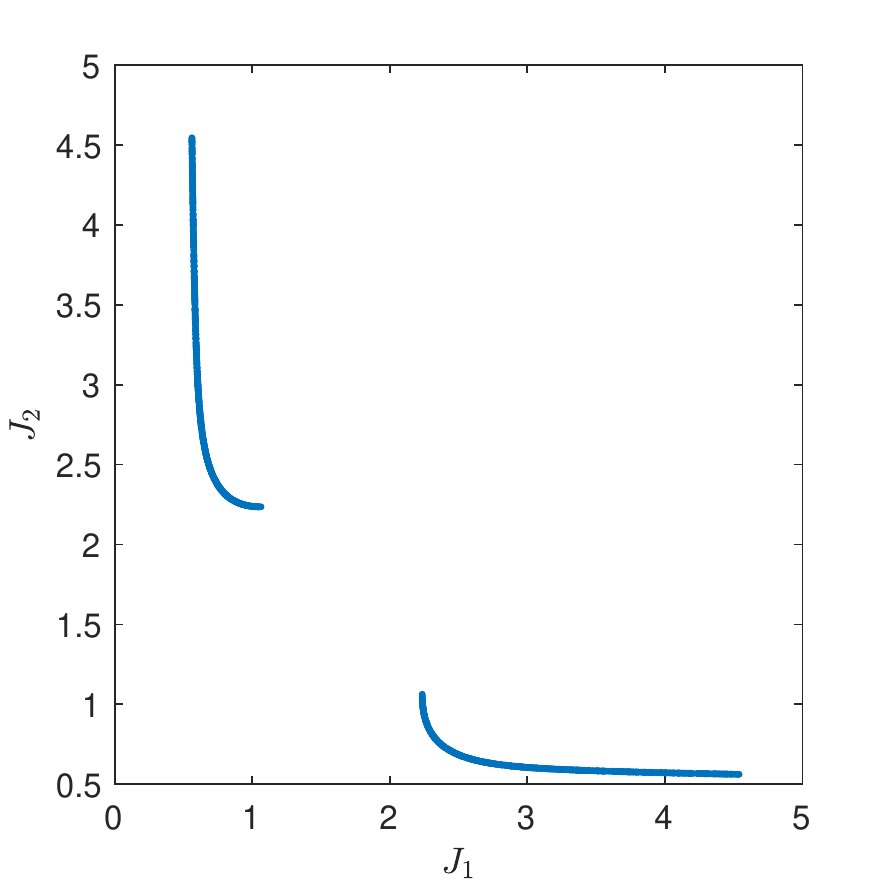}
				\caption{$100,000$ iterations}
				\label{fig:exA4}
			\end{subfigure}
			\caption{Approximations of the efficient set for Problem \eqref{eq:witting} with different number of random samples.}
			\label{fig:exA}
		\end{figure}

        To solve this problem, we applied Algorithm \ref{alg:op} on $30$ archives of $100,000$ uniform random feasible points each. Further, we measured the quality of the feasible points in the archiver in terms of the $\Delta_2$ indicator \cite{SELC12}. This indicator measures the distance between a reference set and an approximation using averaged Hausdorff distance with the $L_2$ norm. The quality was measured after $500$, $1,000$, $10,000$ and $100,000$ iterations. Figure~\ref{fig:exA} shows the result of applying  Algorithm~\ref{alg:PQe1} to Problem \eqref{eq:witting}. Figure~\ref{fig:exAdp} shows the convergence of the method for the $\Delta_2$ indicator. For all cases, we show the median. From the results, we can observe that the archiver can keep efficient solutions and that the method converges as the number of samples increases.

        \begin{figure}
        	\centering
            \begin{subfigure}[t]{.45\textwidth}                         
                \includegraphics[width=\textwidth]{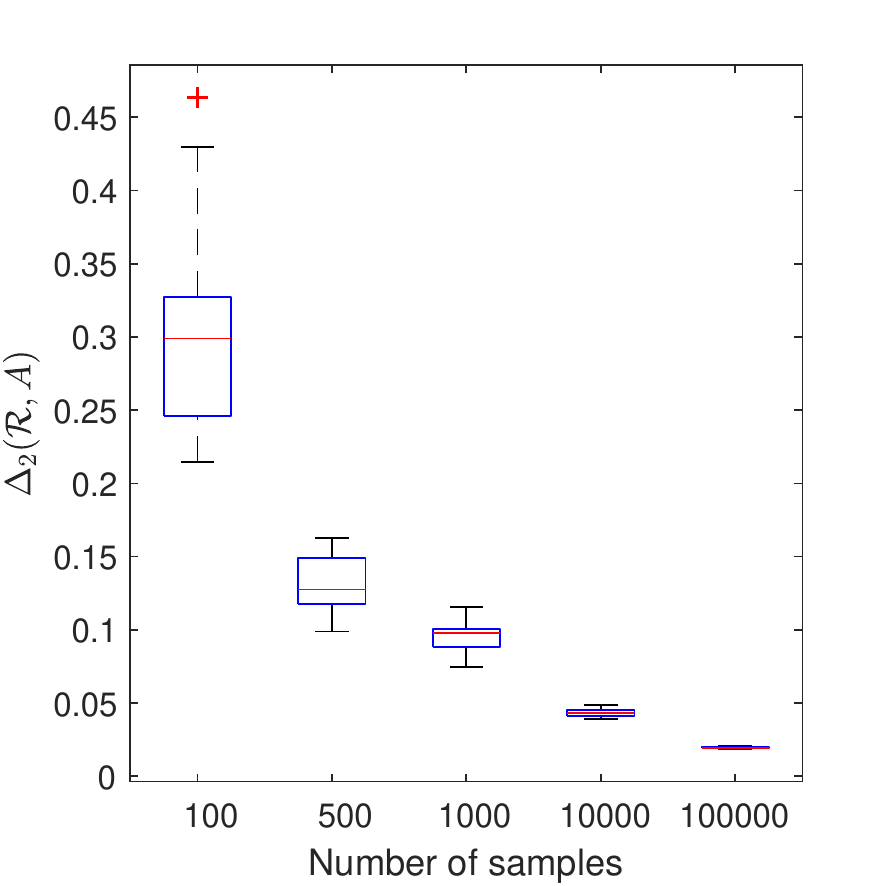}
                \caption{$\Delta_2(\mathcal{R},A)$}
                \label{fig:exAdp1}
            \end{subfigure}\hfil
            \begin{subfigure}[t]{.45\textwidth}
                \includegraphics[width=\textwidth]{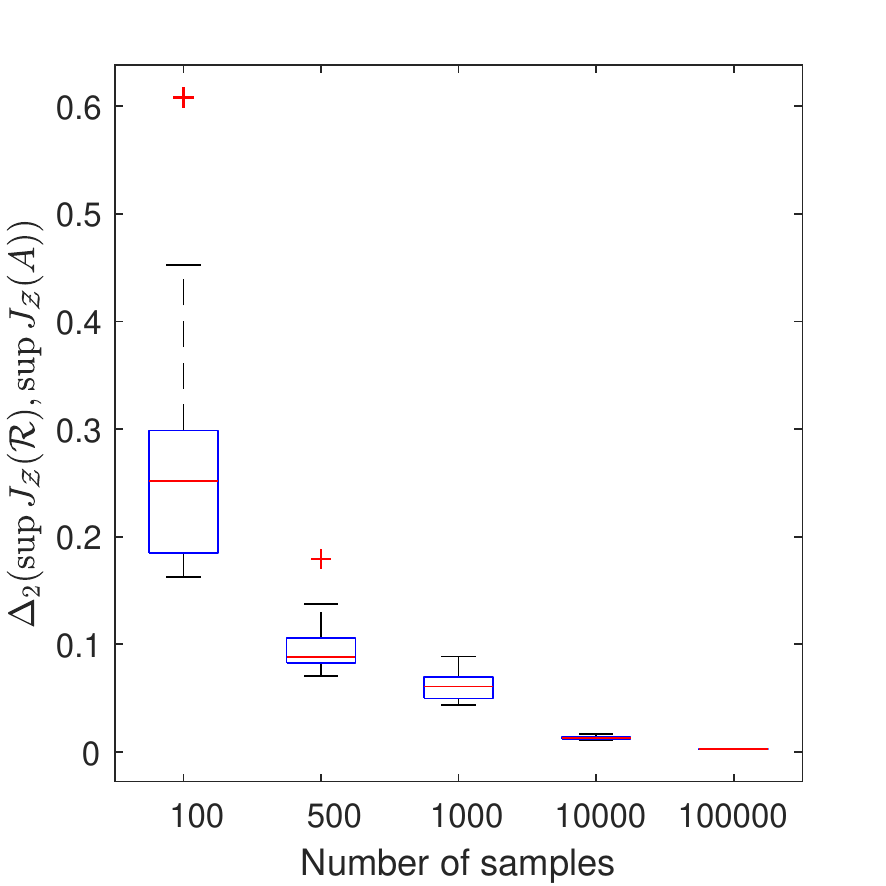}        
                \caption{$\Delta_2(\sup J_\mathcal{Z}(\mathcal{R}), \sup J_\mathcal{Z}(A))$}
                \label{fig:exAdp2}
            \end{subfigure}
            \caption{Convergence of the averaged Hausdorff distance between the approximation found by the archiver ($A$) and a discretization of the efficient set ($\mathcal{R}$) in decision space (left) and objective space (right).}
            \label{fig:exAdp}
        \end{figure}
        
    \subsection{Online phase}
        The online phase consists of two steps. First, we consider the approach proposed in \cite{Joh02,BF06,ober2018explicit} to obtain a promising initial feasible control input (We do not select a state trajectory from the library but only an input control trajectory $u$, as the state is generated by the system which is running in parallel) by exploiting the library generated in the offline phase. 
        Next, starting from this solution, we compute a feasible optimal compromise (i.e., a solution that complies with the constraints) using a \emph{reference point method} \cite{wierzbicki1980use,POBD19}, and apply it to the system.
        
        In the first step, the task is to identify the current values for $\tilde x_0$. We then select the corresponding efficient set $\mathcal{R}_{\tilde x_0}$ from the library. 
        In case the current initial condition is not contained in the library (which occurs with probability one), we use interpolation between different entries, cf.~Algorithm~\ref{alg:online}.
        Next, given a preference of the decision-maker $\rho\in [0,1]^k$ s.t. $\sum_{i=1}^k \rho_i = 1$, an efficient feasible curve is chosen and used as an initial condition for the second step. 
        Therein, we transform the uMOCP to a single-objective optimal control problem with uncertainty by means of a reference point formulation \cite{wierzbicki1980use}. 
        In this method, the distance to an infeasible \emph{reference point} $Z \in \mathbb{R}^k$ in image space is minimized, which leads to an efficient point. Hence, 
        the decision-maker has an influence on the performance that should be achieved, as the``closest'' realization to the reference point $Z$ will be selected. Formally, the problem can be stated as: 
        \begin{equation}
            \min_{u\in \mathcal{X}} dist\left(\max_{\alpha\in \mathcal{Z}} \{\hat J(\tilde x_0+\alpha,u)\}, Z\right),
            \label{eq:rhd}
        \end{equation}
        where $Z \in \mathbb{R}^k$ and $dist$ is a distance between a set and a vector. For this purpose, we use the Hausdorff distance.        
        \begin{definition}[Hausdorff distance]
            Let $u, v \in \mathbb{R}^n$ and $A,B \subset \mathbb{R}^n$. The maximum norm distance $d_\infty$, the semi-distance $dist(\cdot, \cdot)$ and the Hausdorff distance $d_H(\cdot, \cdot)$ are defined as follows:
            \begin{enumerate}
                \item $d_\infty = \max \limits_{i=1,\ldots,n} \left|u_i-v_i\right|$ 
                \item $\text{dist}(u,A) = \inf \limits_{v\in A} d_\infty(u,v)$
                \item $\text{dist}(B,A) = \sup \limits _{u\in B} d_\infty (u, A)$
                \item $d_H(A,B) = \max \{ \text{dist}(A,B), \text{dist}(B,A) \}$
            \end{enumerate}
            \label{def:hd}
        \end{definition}
        
        \begin{theorem}
        Given an uncertain multi-objective optimal control problem, $u^*$ is a set-based minmax robust efficient solution, if and only if $u^*$ is an optimal solution to Problem \eqref{eq:rhd} for some utopic $Z$ and if $\max_{\alpha \in U} \hat J_i(\tilde x_0 + \alpha, u)$ exists for all $u\in U$ and for all $i=1,\ldots,k$.
        \end{theorem}
        
        \begin{proof}
            Assume that $u^*$ is not a robust efficient solution for Problem \ref{eq:smocp}. Then there exists an $u' \in U$ such that $\hat J_\mathcal{Z}(\tilde x_0,u') \subset \hat J(\tilde x_0, u^*) - \mathbb{R}^k_\geq$. Based on \cite[Lemma 3.4]{ehrgott2014minmax}, for all $\alpha \in U$, there exists $\alpha \in \mathcal{Z}$ such that $\hat J(\tilde x_0+\alpha, u') \preceq F(\tilde x_0+\alpha, u^*)$.
            
            From this fact it follows that
            \begin{equation}                
                d_H\left(\max_{\alpha\in U} \{\hat J(\tilde x_0+\alpha, u')\}, Z\right)
                < 
                d_H\left(\max_{\alpha\in U} \{\hat J(\tilde x_0+\alpha, u^*)\}, Z\right).
            \end{equation}
            which contradicts the assumption that $u^*$ is a solution of Problem \eqref{eq:rhd}.
        \end{proof}
        
        It should be noted that this is a particular case of the result found in \cite{Zhou:2017}. However, this formulation allows using other set distances such as the generalized averaged Hausdorff distance \cite{vargas2018generalization}. 
        The entire online phase is summarized in Algorithm \ref{alg:online}. % shows the steps of the online phase.
        
        \begin{algorithm}[t]
            \begin{algorithmic}
                \Require Weight $\rho\in\mathbb{R}^k$ with $\sum_{i=1}^k \rho_i=1$ and $\rho\geq 0$.
                \For{$t=t_0,t_1,t_2,\ldots$}
                    \State Obtain the current initial condition $\tilde x_0 = \tilde x(t)$.
                    \State Identify the $2\tilde n_x$ neighboring grid points of $\tilde x_0$ in $\mathcal{L}$. These points are collected in the index set $\mathcal{I}$.
                    \State From each of the corresponding efficient sets $\mathcal{R}_{(\tilde x_0)_i}, i \in \mathcal{I}$, select an efficient solution $u_i$ according to the weight $\rho$.
                    \State Compute the distances $d_i$ between the entries of the library and $\tilde x_0$:
                    $$d_i = \|(\tilde x_0)_i - \tilde x_0\|_2.$$
                    \If{$\exists j \in\mathcal{I}$ with $d_j=0$}
                        \State $u=u_j$
                    \Else
                        \State $$\hat u = \frac{\sum_{i=1}^{|\mathcal{I}|}\frac{1}{d_i}u_i}{\sum_{i=1}{|\mathcal{I}|}\frac{1}{d_i}}$$
                        \State Solve Problem (\ref{eq:rhd}) with $\hat u$ as initial point to find $u$.
                    \EndIf
                    \State Apply $u$ to the plant for the control horizon length $t_c$.
                \EndFor
            \end{algorithmic}
            \caption{Online phase}
            \label{alg:online}
        \end{algorithm}
        
        Now, we present an example of the reference point method (\ref{eq:rhd}) on Problem (\ref{eq:witting}) using a standard SQP solver for the single-objective optimization problems. We have taken $u_0 = (-1.8, -1.6)^T$ and $Z = (0, 0)^T$. Figure \ref{fig:exCS4} shows the landscape and the contour plot for the problem. It is possible to observe that for this particular instance, it is a multi-modal problem. Figure \ref{fig:exCS1} shows the path of the algorithm in decision space and Figure \ref{fig:exCS2} shows the path in objective space. From the figures, it is possible to observe the method converges to the optimal solution in this scenario within a few iterations.
                        
        \begin{figure}[h]
        	\centering
            \begin{subfigure}[t]{.4\textwidth}
                \includegraphics[width=\textwidth]{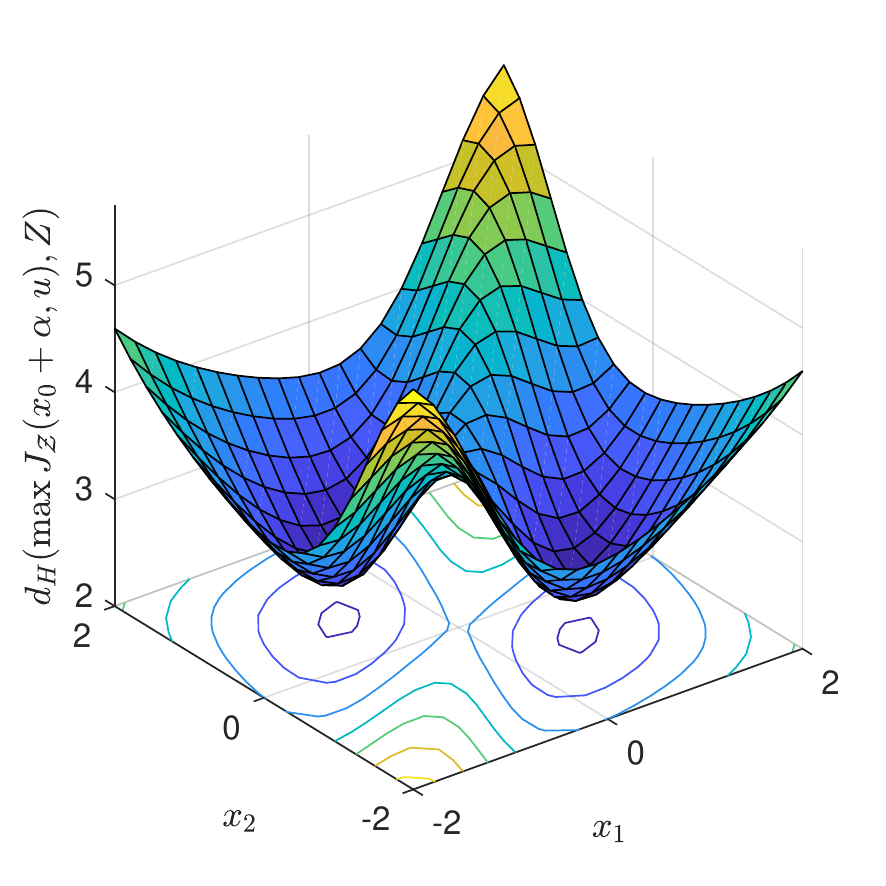}
                \caption{Landscape of the problem.}
                \label{fig:exCS4}
            \end{subfigure}\hfil
            \begin{subfigure}[t]{.4\textwidth}                         
                \includegraphics[width=\textwidth]{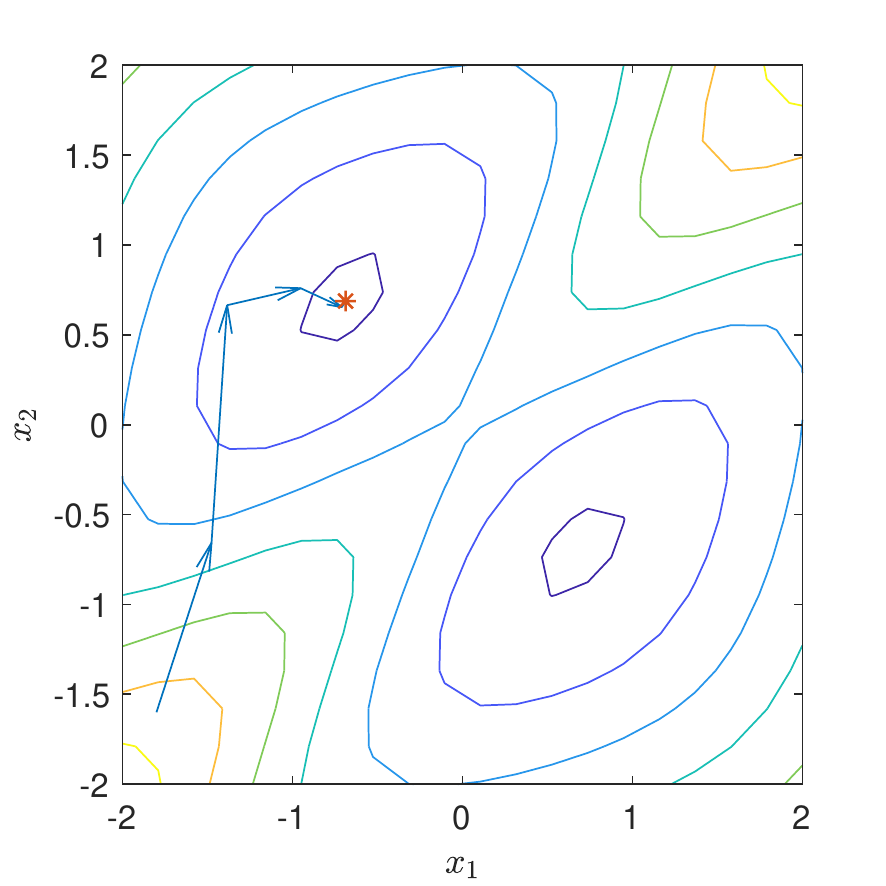}
                \caption{Path in decision space.}
                \label{fig:exCS1}
            \end{subfigure}\\
            \begin{subfigure}[t]{.4\textwidth}
                \includegraphics[width=\textwidth]{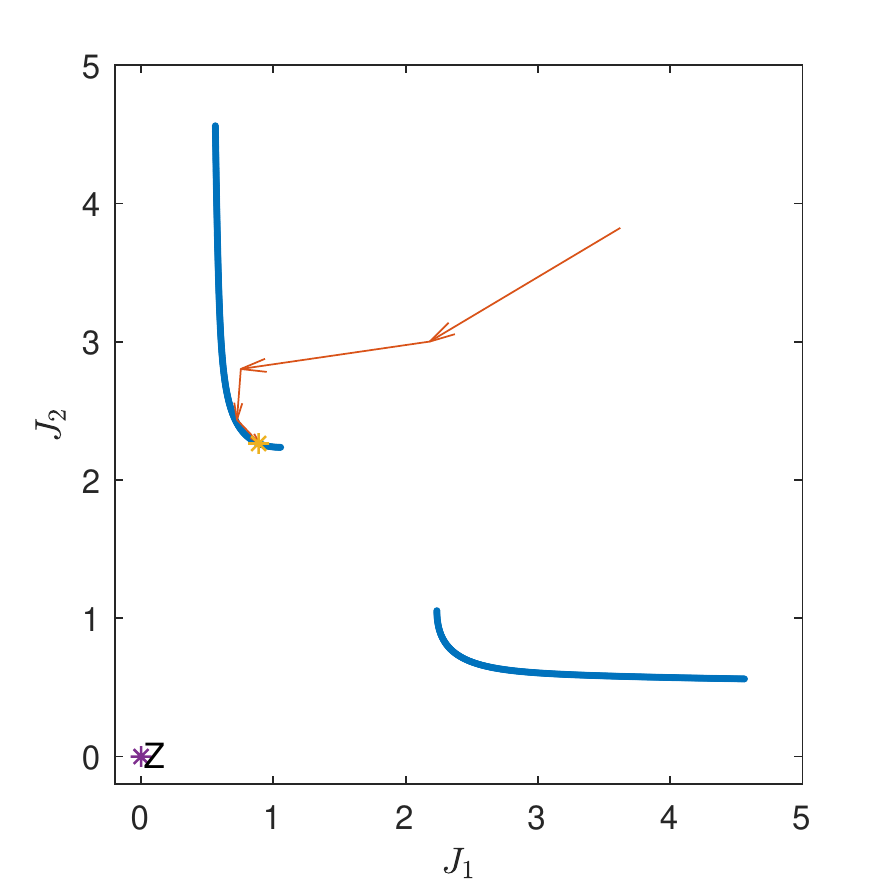}        
                \caption{Path in objective space.}
                \label{fig:exCS2}
            \end{subfigure}\hfil
            \begin{subfigure}[t]{.4\textwidth}
                \includegraphics[width=\textwidth]{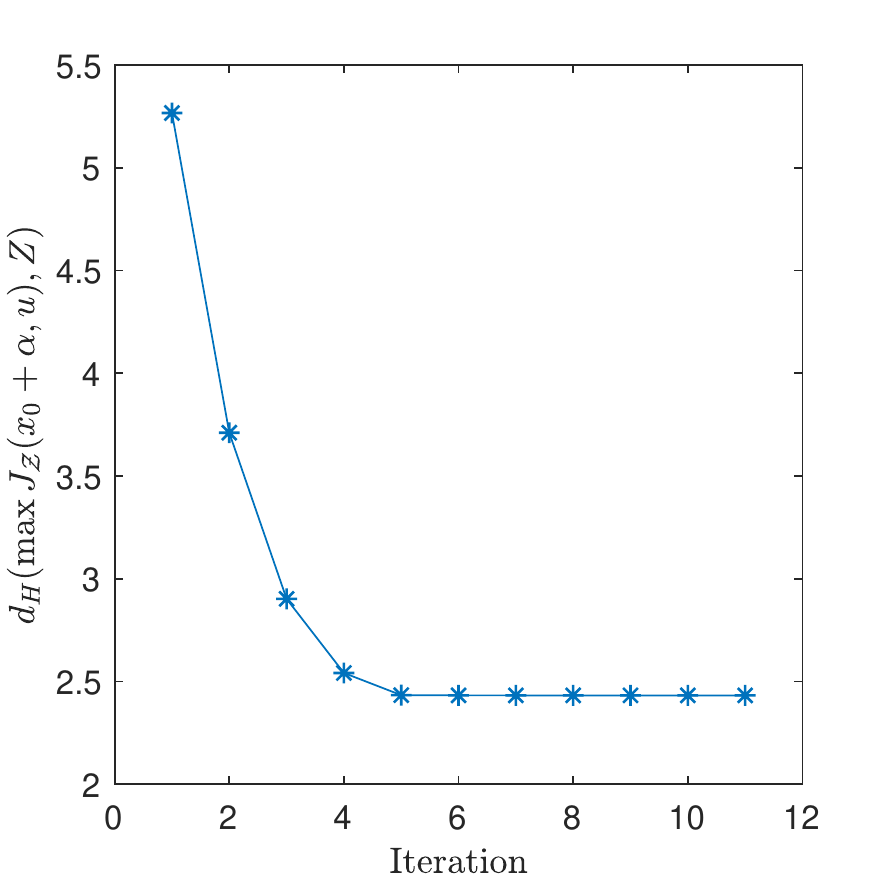}
                \caption{Convergence plot.}
                \label{fig:exCS3}
            \end{subfigure}            
            \caption{Example of the reference point method \eqref{eq:rhd} on Problem \eqref{eq:witting} with $u_0 = (-1.8, -1.6)^T$ and $Z = (0, 0)^T$.}
            \label{fig:exCS}
        \end{figure}

\section{Multi-objective car maneuvering}
    In this section, we present a car maneuvering problem with uncertainty in the state space that we will be dealing with throughout the rest of the paper. We first present the mathematical formulation of the problem and then perform a thorough analysis of the model with respect to the conflicting objective functions for a very coarse discretization in time by $N_u=2$ points. This way, we can gain insight into the problem structure and take this into consideration in our solution strategy.
    %can gain insight into the structure of the problem and then design and solve the problem accordingly. 
    Furthermore, we analyze the effect of the uncertainty in the different parameters by performing a sensitivity analysis of the problem. Finally, we conclude the section by applying our on-/offline optimization approach and present the numerical results.
    We will show that using an additional online optimization step results in a very efficient framework for feedback control of complex nonlinear systems with uncertainty.
    
    \subsection{Model definition}
        In this work, we consider the well-known bicycle model \cite{taheri1990investigation}. We are interested in optimally determining the steering angle for a vehicle with respect to the objectives secure and fast driving under uncertain initial conditions on a given track. 
        %This model was also considered in \cite{ober2018explicit} in the multi-objective case. 
        In this model -- which was also considered in \cite{ober2018explicit} -- the dynamics of the vehicle are approximated by representing the two wheels on each axis by one wheel on the centerline, see Figure \ref{fig:Car_Bicycle_a} for an illustration. When assuming a constant longitudinal velocity $v_x$, this leads to a nonlinear system of five coupled ODEs:
        \begin{equation}
            \begin{split}
            \dot{x}(t) &=
                \begin{pmatrix}
                    \dot{p_1}(t) \\
                    \dot{p_2}(t) \\
                    \dot{\varTheta}(t) \\
                    \dot{v_y}(t) \\
                    \dot{r}(t)
                \end{pmatrix}
                =
                \begin{pmatrix}
                    v_x(t)\cos(\varTheta(t)) - v_y(t)\sin(\varTheta(t))\\
                    v_x(t)\sin(\varTheta(t)) + v_y(t)\cos(\varTheta(t))\\
                    r\\
                    C_1(t)v_y(t) + C_2(t)r(t) + C_3(t)u(t)\\
                    C_4(t)v_y(t) + C_5(t)r(t) + C_6(t)u(t)                    
                \end{pmatrix}
                , t\in (t_0,t_e],\\
            x(t_0) &= x_0,
            \end{split}
            \label{eq:dyn}
        \end{equation}
where $x = (p_1,p_2,\varTheta,v_y,r)^T$ is the state state consisting of the position $p=(p_1,p_2)$, the angle $\varTheta$ between the horizontal axis and the longitudinal vehicle axis, the lateral velocity $v_y$ and the yaw rate $r$. The vehicle is controlled by the front wheel angle $u$ and the variables
        \begin{equation}
            \begin{array}{ll}
                C_1(t) = -\frac{C_{\alpha,f}\cos(u(t))+C_{\alpha,r}}{mv_x(t)},
                &C_2(t) = -\frac{L_fC_{\alpha,f}\cos(u(t))+L_rC_{\alpha,r}}{I_zv_x(t)},\\
                C_3(t) = \frac{C_{\alpha,f}\cos(u(t))}{m},
                &C_4(t) = -\frac{L_fC_{\alpha,f}\cos(u(t))+L_rC_{\alpha,r}}{mv_x(t)}-v_x(t),\\
                C_5(t) = -\frac{L^2_fC_{\alpha,f}\cos(u(t))+L^2_rC_{\alpha,r}}{I_zv_x(t)},
                &C_6(t) = \frac{L_fC_{\alpha,f}\cos(u(t))}{I_z},\\
            \end{array}
        \end{equation}
        have been introduced for abbreviation. The physical constants of the vehicle model are presented in Table \ref{tab:phyconst}.
        
        \begin{figure}
        	\centering
            \begin{subfigure}[t]{.3\textwidth}
                \includegraphics[width=\textwidth]{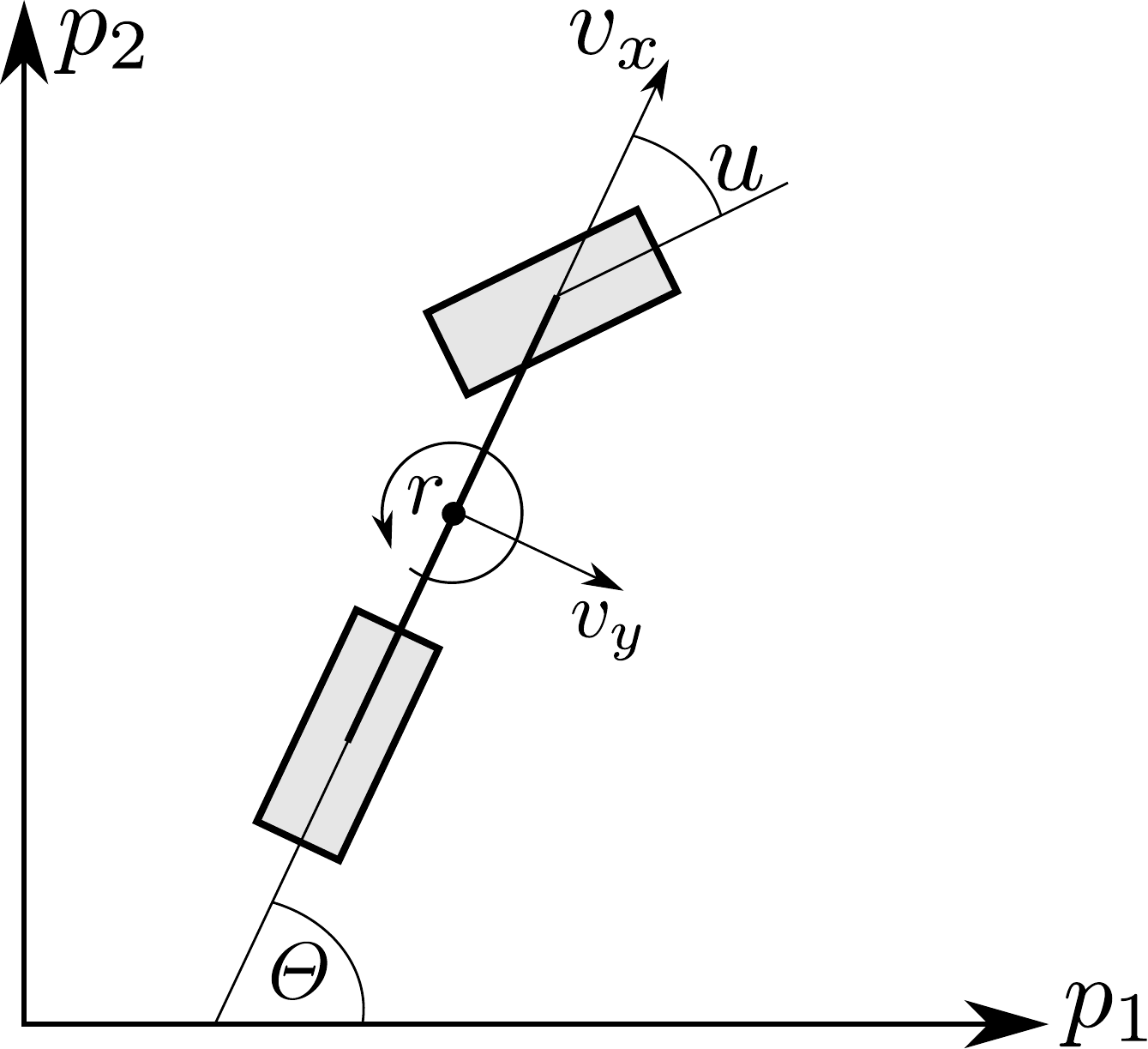}
                \caption{Bicycle model.}
                \label{fig:Car_Bicycle_a}
            \end{subfigure}
            \hfil
            \begin{subfigure}[t]{.5\textwidth}
                \includegraphics[width=\textwidth]{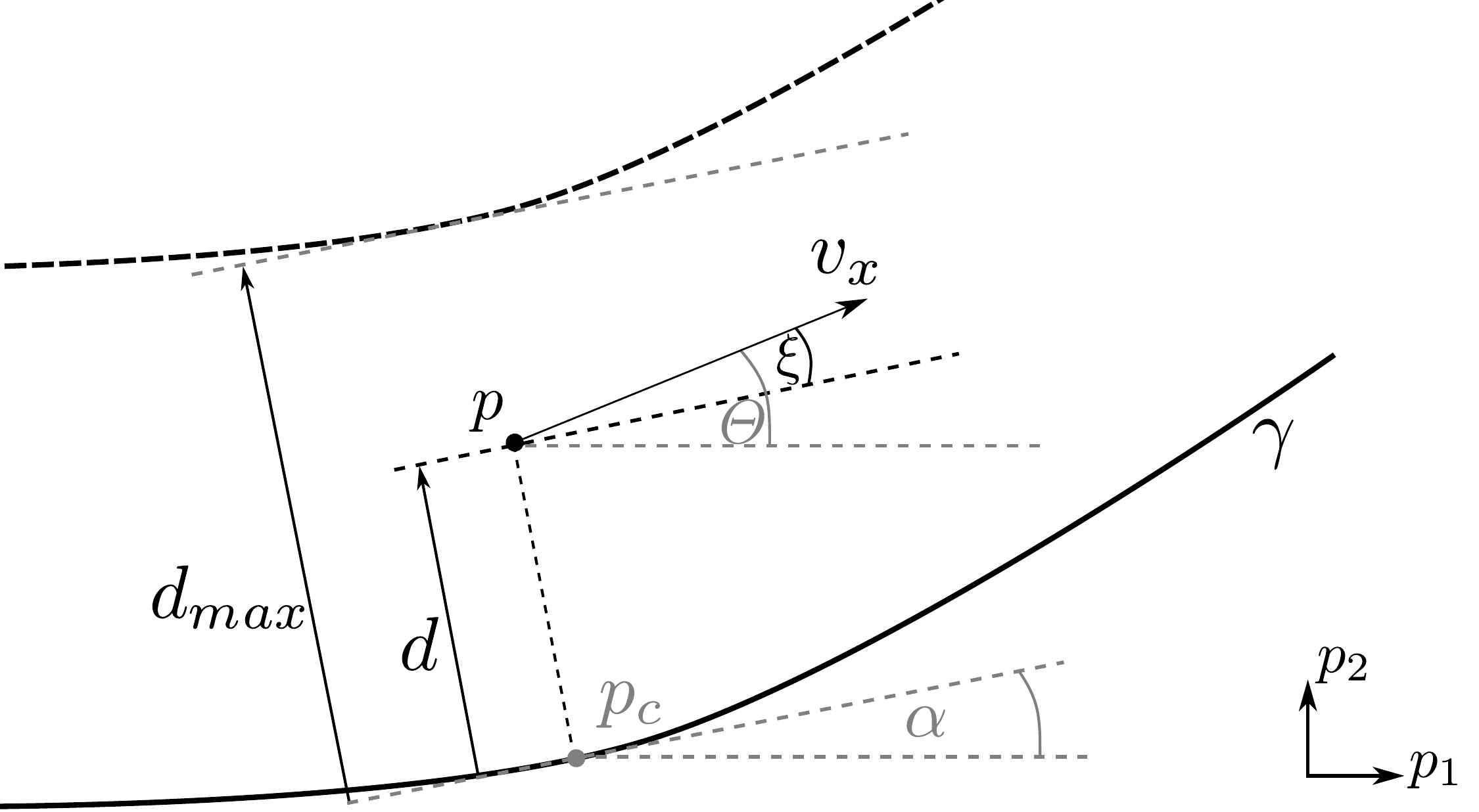}
                \caption{Track linearized in $p_c$.}
                \label{fig:Car_Bicycle_b}
            \end{subfigure}
        	\caption{Bicycle model and coordinates relative to the track at position $p$ \cite{ober2018explicit}.}
        	\label{fig:Car_Bicycle}
        \end{figure}
        
        \begin{table}
        	\centering
            \caption{Physical constants of the vehicle model.}
            \label{tab:phyconst}
            \begin{tabular}{|l|l|l|}
                \hline
                Variable    &Physical property  &Numerical value\\ \hline
                $C_{\alpha,f}$   &Cornering stiffness coefficient (front)    &65100  \\
                $C_{\alpha,f}$   &Cornering stiffness coefficient (rear)    &54100  \\
                $L_f$   &Distance front wheel to center of mass &1  \\
                $L_r$   &Distance rear wheel to center of mass  &1.45   \\
                $m$   &Vehicle mass   &1275   \\
                $I_z$   &Moment of inertia    &1627   \\ \hline
            \end{tabular}
        \end{table}

        The first objective measures the distance $d$ to the centerline (cf.~Figure \ref{fig:Car_Bicycle_b}), while the second objective corresponds to the distance driven along a track $\gamma$. For both objectives, we use projections of the position $p$ of the vehicle onto the centerline ($p_c$):
        \begin{equation}
            \Pi_\gamma(p(t)) = \arg\min\limits_{p_c\in\gamma}\|p(t) - p_c\|_2,
        \end{equation}
        and the corresponding distance to the centerline is defined as
        \begin{equation}
            d(t) = \min\limits_{p_c\in\gamma}\|p(t) - p_c\|_2 = \|p(t) - \Pi_\gamma(p(t))\|_2.
        \end{equation}
        
        In its current form, the system cannot be parameterized, as the track $\gamma$ is defined by a function. To this end, a local approximation of the track was proposed in \cite{ober2018explicit} (cf.~Figure \ref{fig:Car_Bicycle_b}), where the current angle $\alpha$ and curvature $\kappa = \frac{d \alpha}{ds}$, where $s$ is the coordinate along the centerline, are assumed to be constant. This way, the track can be described by the four parameters $\{p_{c,1}, p_{c,2}, \alpha, \kappa \}$. In combination with the initial condition $x_0$, we have, in total, a nine-dimensional parameter. When introducing a numerical grid for this parameter, this results in prohibitively expensive offline-phase. Fortunately, we can exploit several symmetries in the system. As the entire problem setup is invariant under translation as well as rotation, only the relative position between the vehicle and the track is of interest, by which we can replace the parameters $\{p_1,p_2,\varTheta,p_{c,1}, p_{c,2}\}$ by the distance $d$ and the relative angle $\xi$ (Figure \ref{fig:Car_Bicycle_b}, for details on the symmetry reduction we refer to \cite{ober2018explicit}). In summary,
        the problem can be parametrized by five parameters:
        \begin{equation}
            \tilde x_0 = (v_y,r,\xi,d,\kappa)^T,
        \end{equation}
        where we use the notation $\tilde{x}_0$ to indicate that this contains both the initial condition of the dynamics ($v_y$ and $r$) as well as the relation to the track ($\xi$, $d$ and $\kappa$) at $t=t_0$.
        %where $\xi$ is the angle between the track and the vehicle direction, $d$ is the distance to the centerline, and $\kappa$ is the curvature of the track and the initial orientation and position for the vehicle dynamics become:
        Exploiting the symmetries, the initial conditions for \eqref{eq:dyn} become
        \begin{equation}
            \varTheta = \xi, \; p(t_0)=(0,d) \text{ with } \Pi_\gamma(p(t_0)) = (0,0).
        \end{equation}
        
        For this problem, we consider uncertainties that appear due to the precision and resolution of the sensors used and that are given by intervals. In particular, we focus on treating uncertainty in the distance to the centerline $d$ (cf. Section 5.2).
        
        Thus, the uMOCP can be formulated as
        \begin{equation}
            \begin{split}                        
                \min\limits_{u\in\mathcal{U}} \sup_{\alpha\in \mathcal{Z}}(x_0+\alpha,u) &= \min\limits_{u\in\mathcal{U}} \sup_{\alpha\in \mathcal{Z}}
                    \begin{pmatrix}
                        \int^{t_e}_{t_0}d(t)^2 dt \\
                        -\int^{\pi_\gamma(p(t_e))}_{\pi_\gamma(p(t_0)}1 ds
                    \end{pmatrix}
                \\
                \mbox{s.t. }\;\;\;\;\; &\mbox{Dynamics (\ref{eq:dyn})}\\
                &\sup\limits_{\alpha\in \mathcal{Z}} d+\alpha \leq d_{\max}.
            \end{split}
        \end{equation}
        
        In order to solve the problem numerically, we use a \emph{direct approach}, i.e., we introduce a discretization in time on an equidistant grid with step size $h = 0.05$ sec. This way, the control function $u$ becomes a finite-dimensional input with $N_u = (t_e-t_0)/h + 1$ entries, and the above problem becomes a (potentially high-dimensional) parameter-dependent MOP with uncertainty.
        
    \subsection{Thorough analysis of the model}
        In this section, we analyze the model for $N_u=2$. Namely, we show the search space for each objective function, and then we examine the basins of attraction with the so-called \emph{cell mapping technique} \cite{hsu:87,sun2018cell}.

        Cell mapping methods transform the point-to-point dynamics into a cell-to-cell mapping by discretizing the space by hypercubes of finite size. Using the common descent direction $q \in \mathbb{R}^n$ for all objectives (see, e.g., \cite{FS00}), a discrete dynamical of the form $u^{(i+1)} = u^{(i)} + h q(u^{(i+1)})$ can be formulated with a suitable step length $h$. The generalized cell mapping method (GCM) now allows us to analyze these dynamics and thus leads to the discovery of invariant sets, stable and unstable manifolds of saddle-like responses, domains of attraction, and their boundaries. The invariant sets (known as the \emph{persistent group} in the Markov chain literature) represent equilibrium points, periodic or chaotic motions. In the context of multi-objective optimization, one can compute the local/global Pareto set and its basin of attraction (among other features) that can be interesting for exploratory landscape analysis \cite{kerschke2014cell}.
        
        For our analysis, we use $\tilde x_0 = (-3,-6,-\pi/4,0,-0.1)^T$. Figure \ref{fig:bm_ld_surf} shows the objective function values for $N_u=2$, and we observe that both functions are multimodal, i.e., $\hat J_1$ and $\hat J_2$ both have an optimum near $[-0.5, -0.5]^T$ and a second one near $[0.5, 0.5]^T$, respectively.
        \begin{figure}          
        	\centering
	        \includegraphics[width=.45\textwidth]{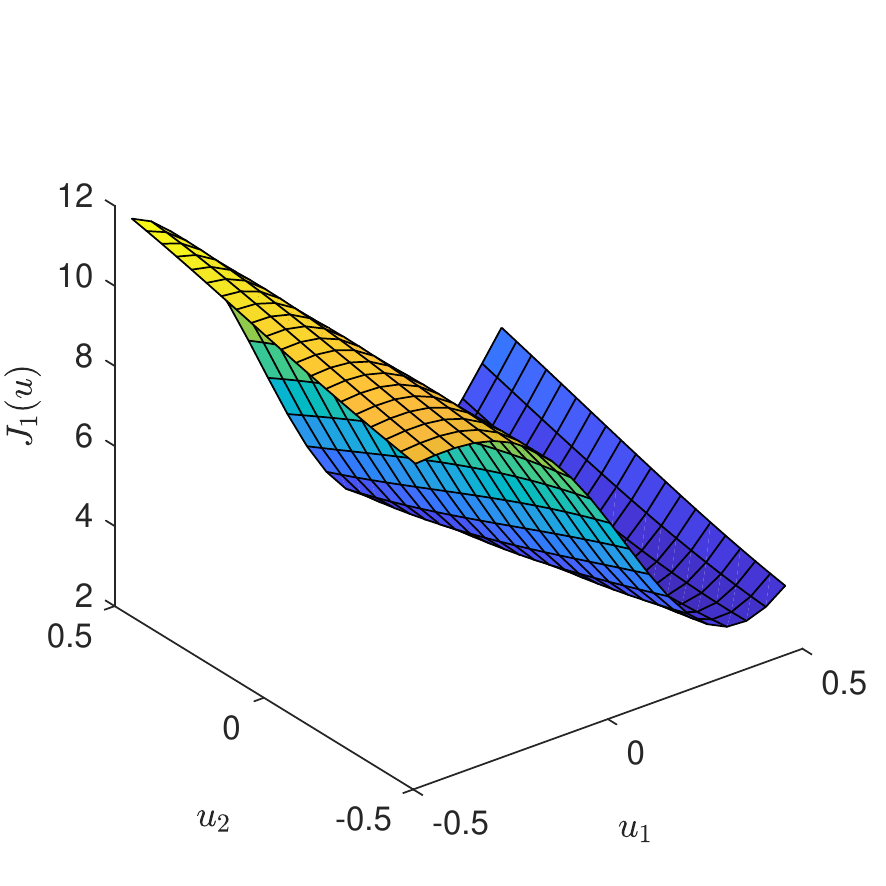} \hfil
	        \includegraphics[width=.45\textwidth]{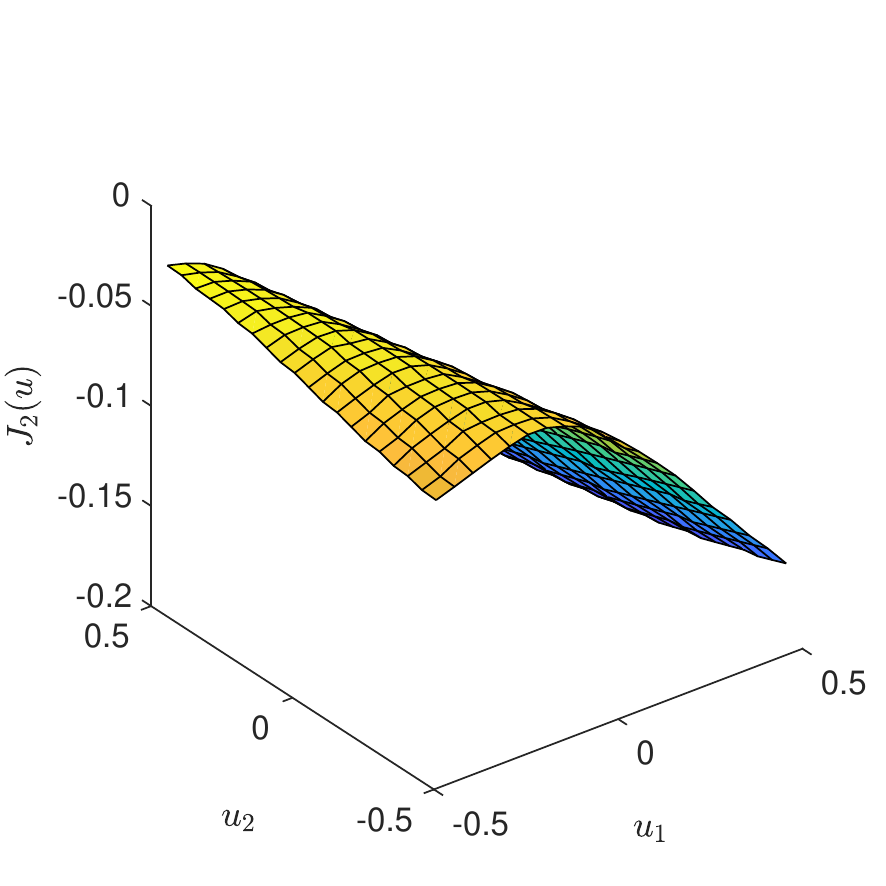}
	        \caption{Surface for $\hat J_1$ (left) and $\hat J_2$ (right).}
	        \label{fig:bm_ld_surf}
        \end{figure}
        
        Figure \ref{fig:bm_ld_attr} shows the attractors (blue) and basins of attraction (arrows) computed by the GCM with the center point method for MOCPs and the corresponding mapping to the objective space. The blue section in the graph corresponds to the attracting regions, and we observe that the function has two attractors. Thus it is a multi-modal function, although the basin of attraction for the global optimum is significantly larger. Notice that this behavior is maintained as the number $N_u$ of dimensions increases.
        
        \begin{figure}        
        	\centering                    
	        \includegraphics[width=.49\textwidth]{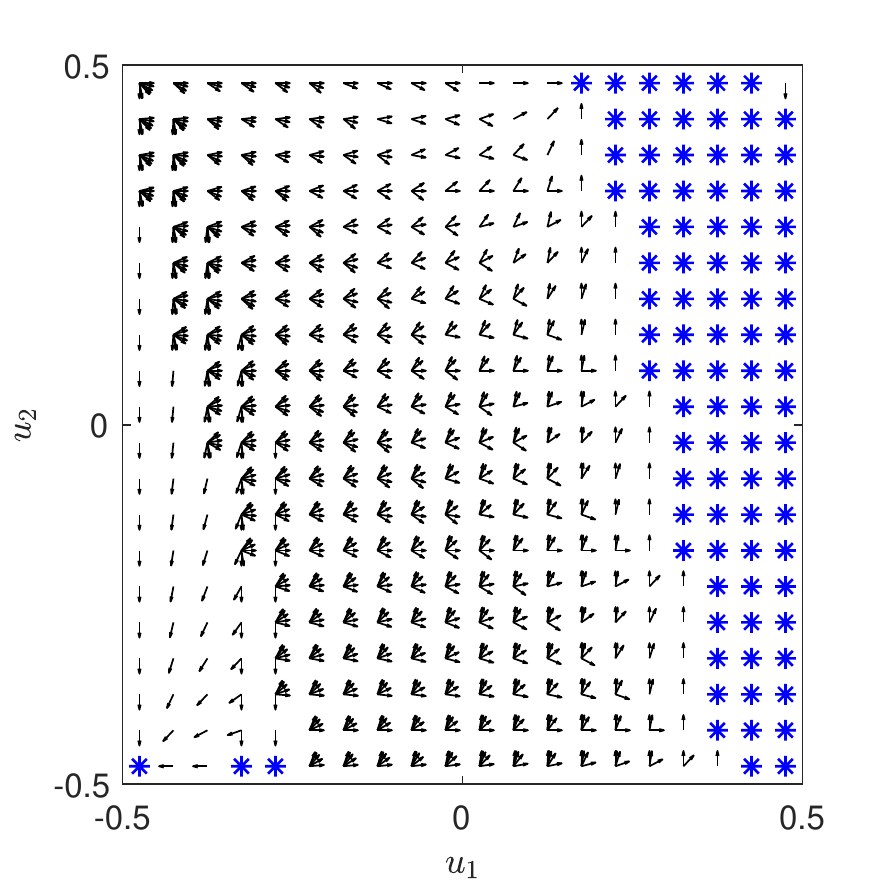} \hfil
	        \includegraphics[width=.49\textwidth]{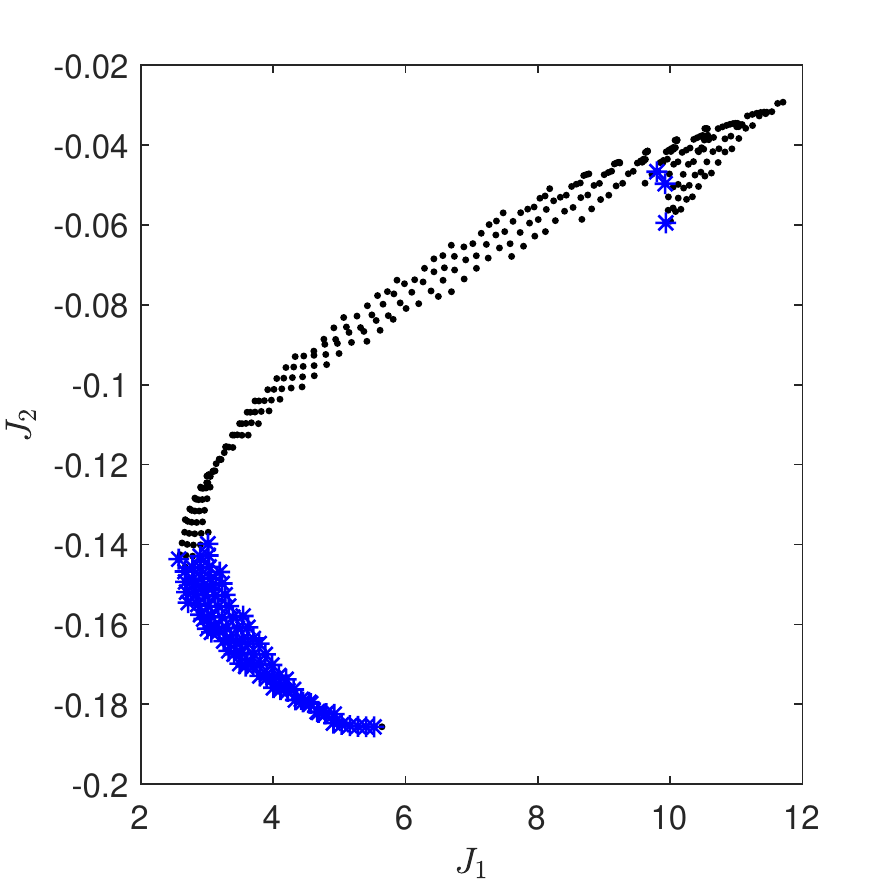}
	        \caption{Decision space (left) and objective space (right). In blue, we show the attractors of $\hat J$ found by GCM with the center point method. In decision space, the arrows represent the directions where there is an improvement in the objective functions. Finally, in objective space, the black dots represent the image of the center points of the cells.}
	        \label{fig:bm_ld_attr}
        \end{figure}
                
        In the following, we present an empirical study on the effect of the uncertainty on lateral velocity, yaw rate, the distance to the centerline, and the mass of the vehicle for $N_u=11$. In this case, we study the uncertainty for $\tilde x_0 = (-0.1801,    0.4349,         0,   -0.0694,   -0.0222)^T$. First, we generated $10,000$ uniform random control variables and computed the set $\hat J_Z(\tilde x_0,u)$ for all $u \in P$. Then we computed their corresponding efficient set. The results are visualized in Figure \ref{fig:ex3unc}. In the figure, each color represents a different feasible point. The first column shows $\hat J_Z(\tilde x_0,u)$ for ten feasible points. There it is possible to observe all the realizations of a feasible point when uncertainty is present. In general, the ``longer'' the lines, the more uncertain a feasible point is. Next, the second column shows the corresponding efficient set approximation. To help visualization, we added dotted lines to join feasible points in the same worst-case. Next, Figure \ref{fig:fronts} shows two efficient sets for initial conditions $(-0.1801,    0.4349,         0,   -0.0694,   -0.0222)^T$ (left) and $(0.9842,   -0.9982,         0,    0.0783,   -0.0222)^T$ (right). In this example the uncertainty is in the distance to the centerline $d$.

        Further, the third column shows $\hat J_Z(\tilde x_0,u)-\mathbb{R}^2_{\succeq}$ for all efficient feasible points. From Figure \ref{fig:ex3unc}, we can observe that in the first three cases, the uncertainty can cause a considerable deterioration in the objective functions. Further, it is possible to visualize that feasible points that were dominated in the nominal case are efficient when uncertainty is introduced. This behavior is especially apparent for the lateral velocity ($v_y$), the yaw rate ($r$) and the distance to the centerline ($d$).
    
    \begin{figure}
    	\centering
        \begin{subfigure}[b]{\columnwidth}
            \includegraphics[width=.27\columnwidth]{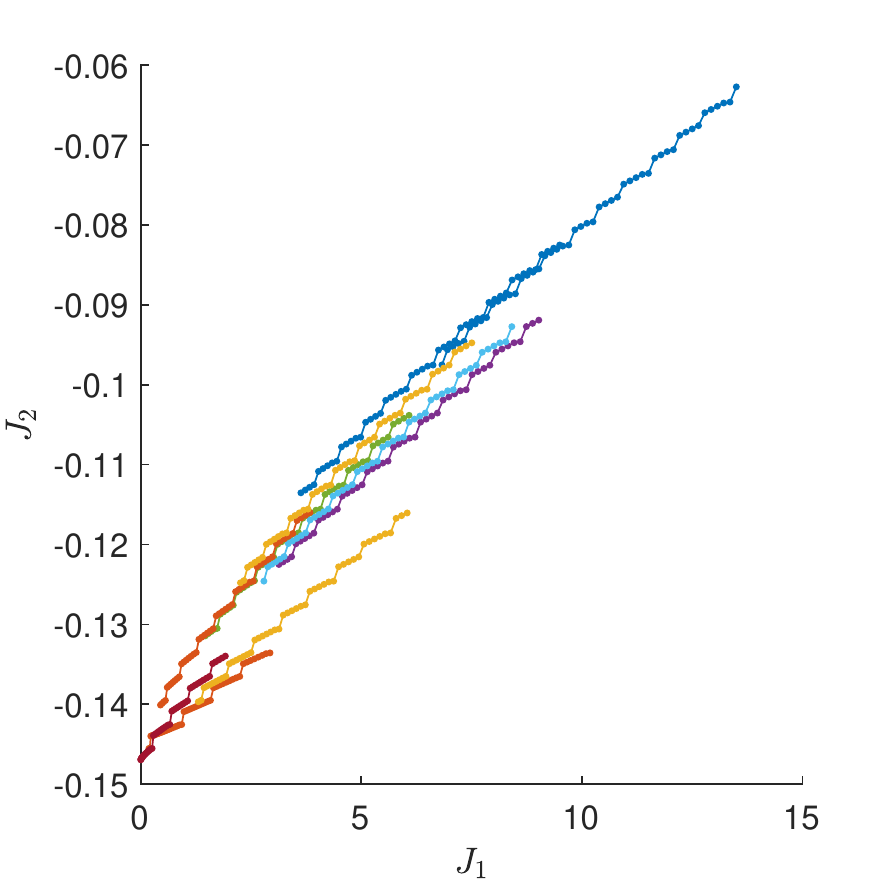} \hfil
            \includegraphics[width=.27\columnwidth]{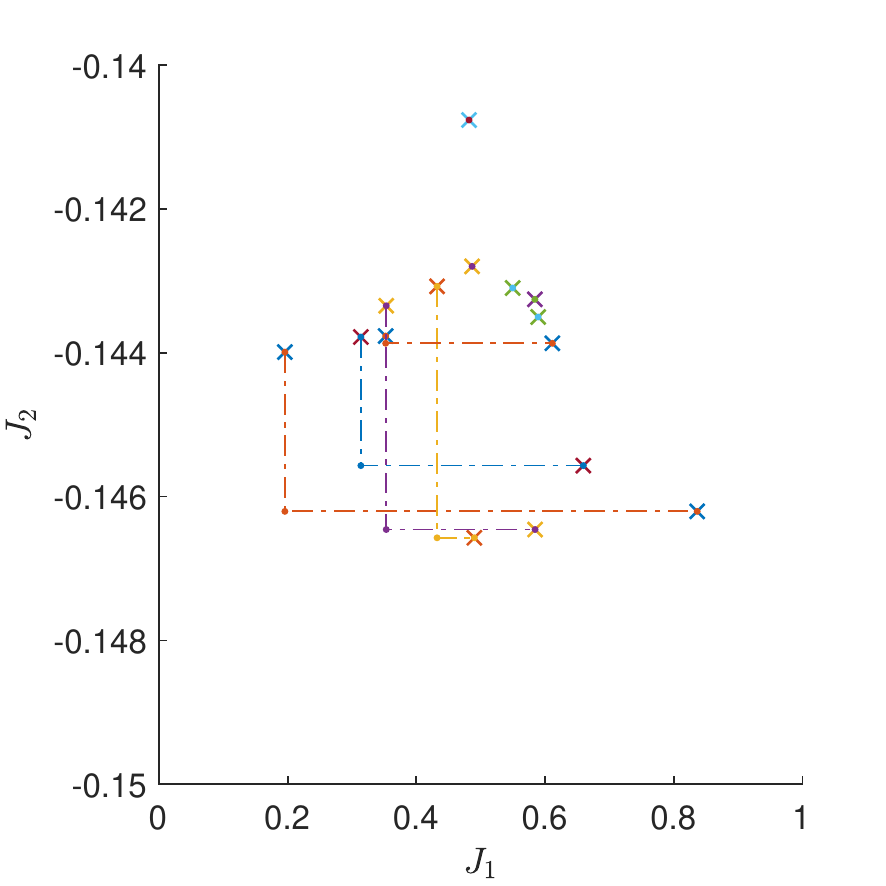} \hfil
            \includegraphics[width=.27\columnwidth]{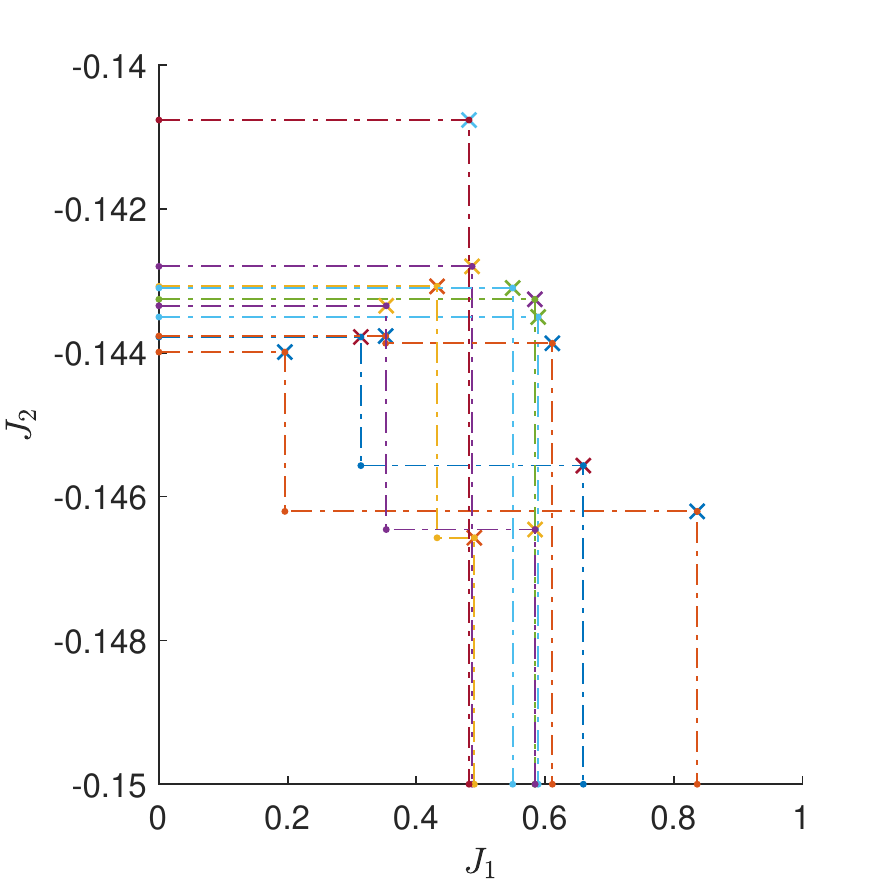}
            \caption{Uncertainty on the lateral velocity ($v_y$).}
        \end{subfigure}
        \begin{subfigure}[b]{\columnwidth}
            \includegraphics[width=.27\columnwidth]{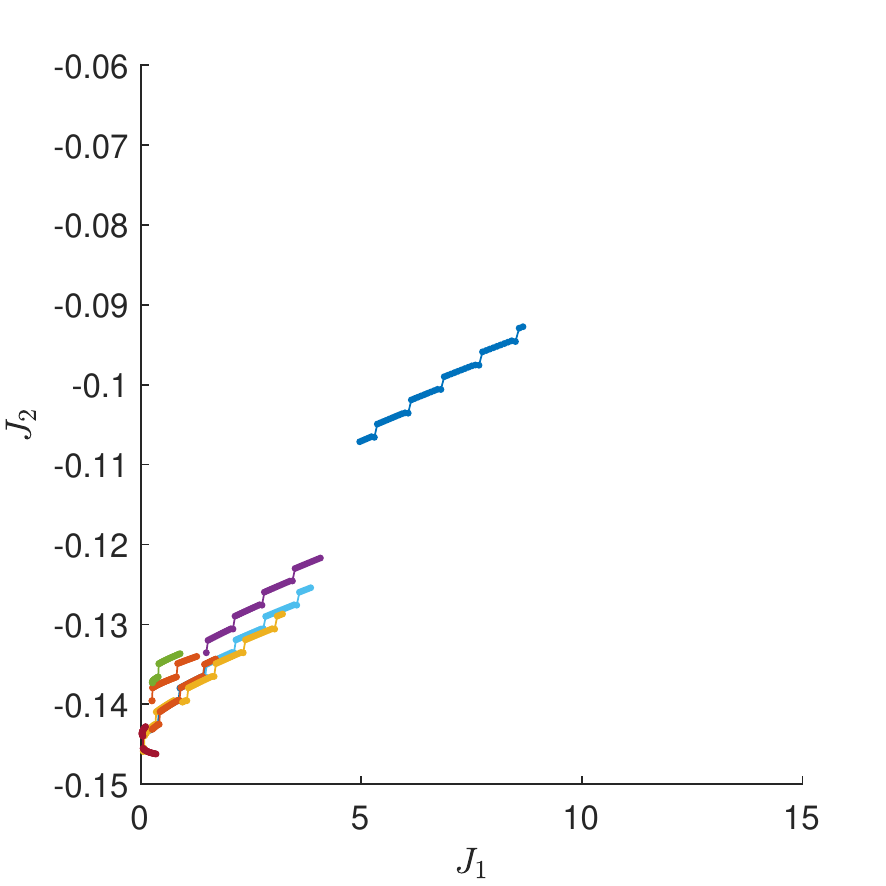} \hfil
            \includegraphics[width=.27\columnwidth]{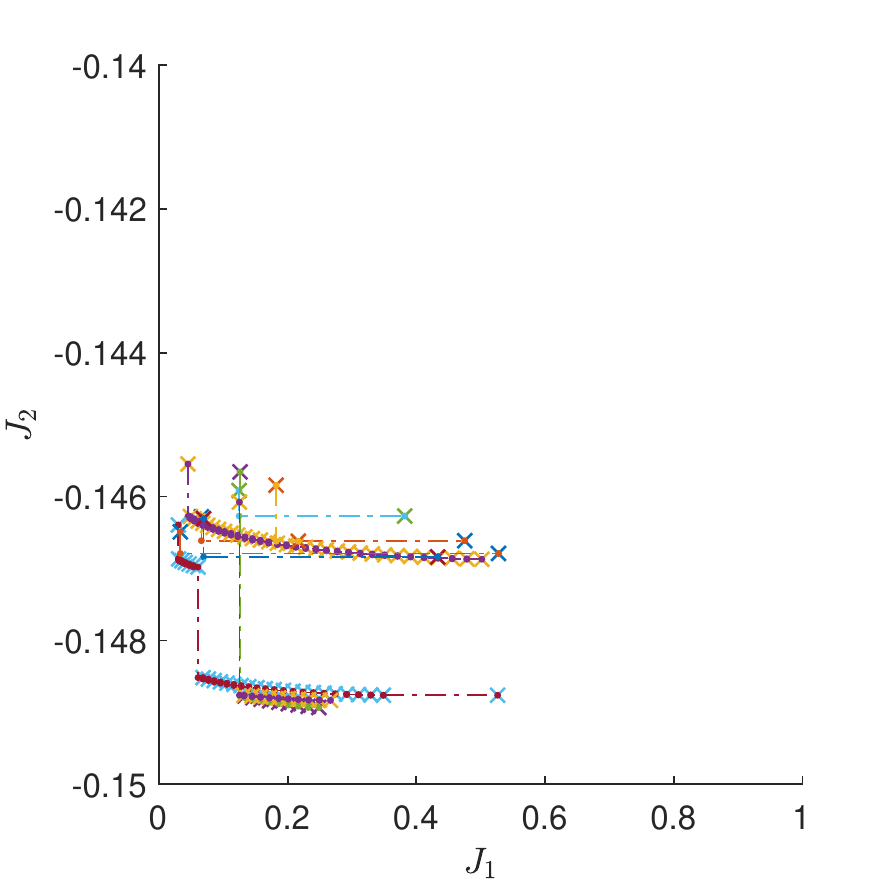} \hfil
            \includegraphics[width=.27\columnwidth]{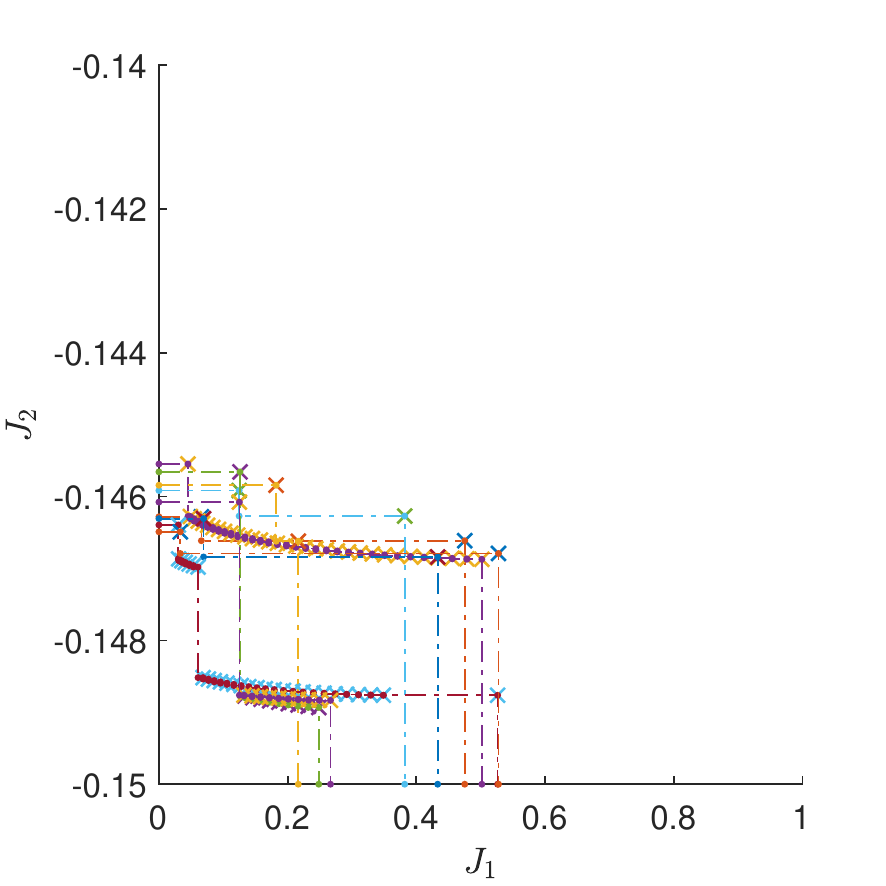}
            \caption{Uncertainty on yaw rate ($r$).}
        \end{subfigure}
        
        \begin{subfigure}[b]{\columnwidth}
            \includegraphics[width=.27\columnwidth]{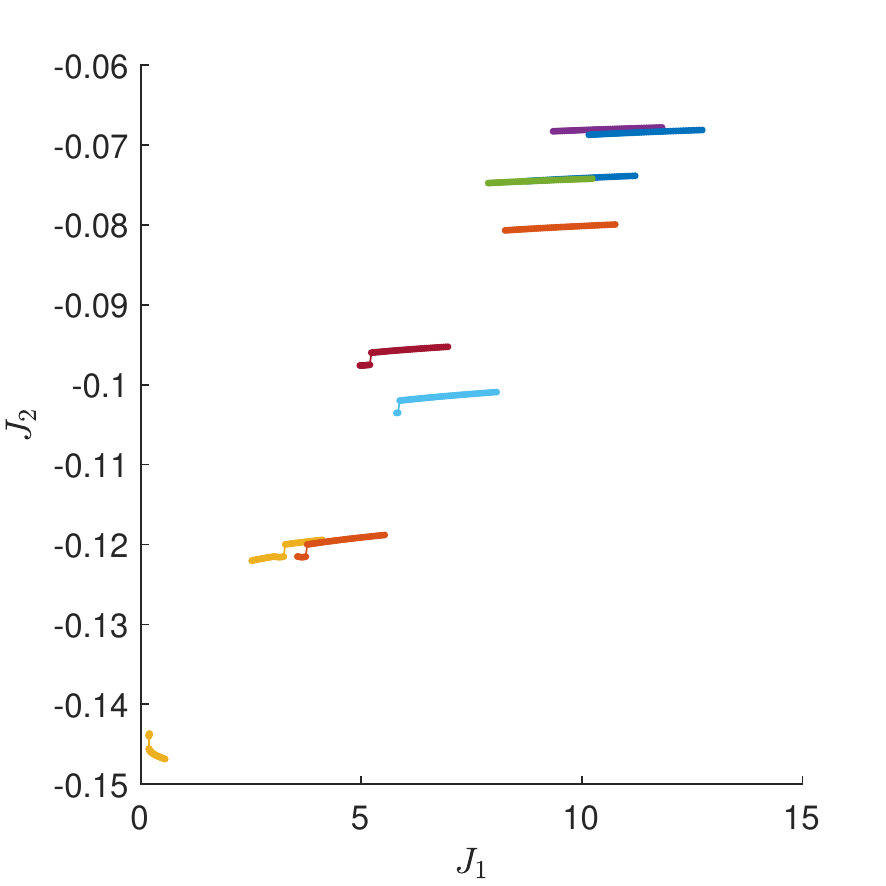} \hfil
            \includegraphics[width=.27\columnwidth]{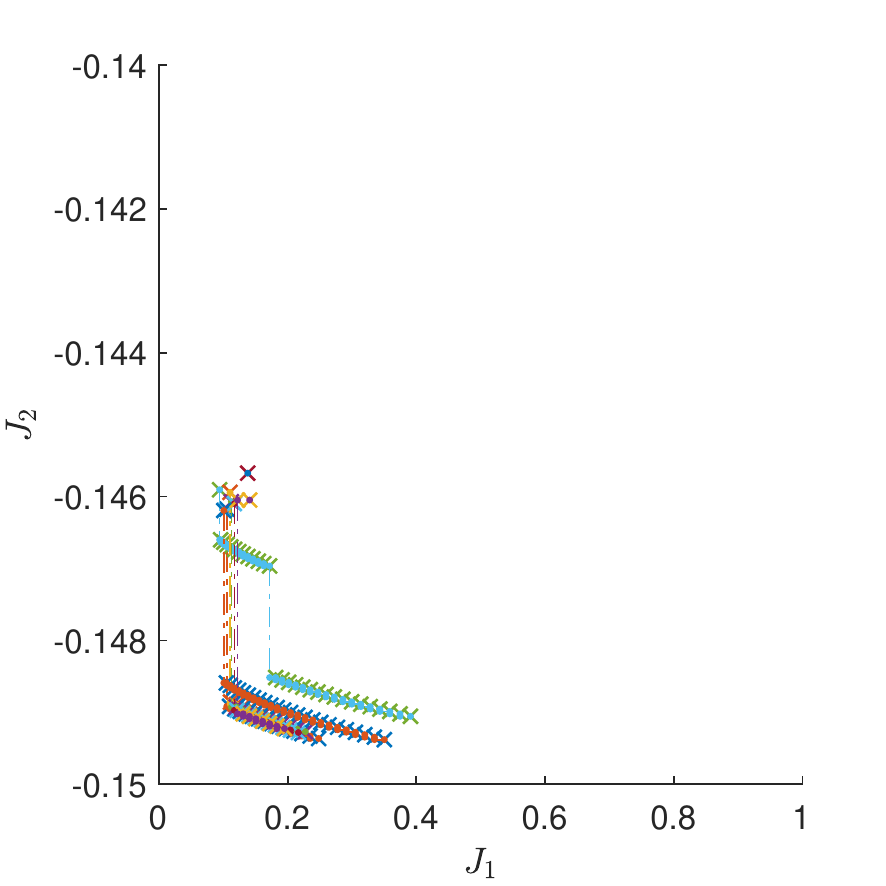} \hfil
            \includegraphics[width=.27\columnwidth]{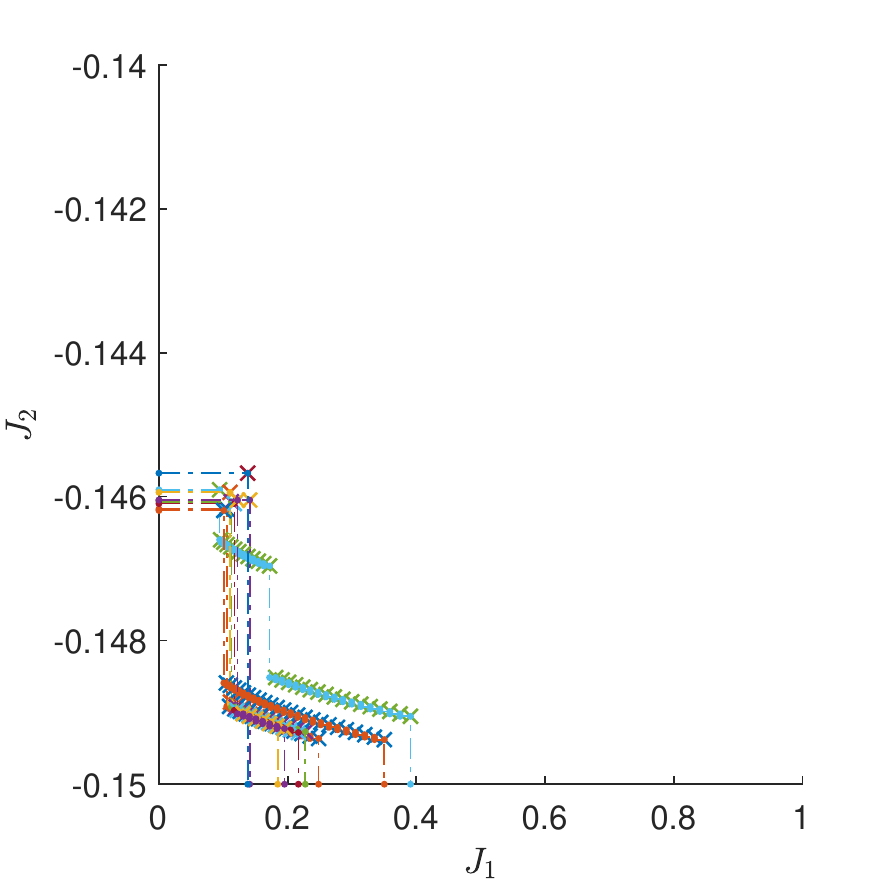}
            \caption{Uncertainty on the distance to the centerline ($d$).}
        \end{subfigure}    
        
        \begin{subfigure}[b]{\columnwidth}
            \includegraphics[width=.27\columnwidth]{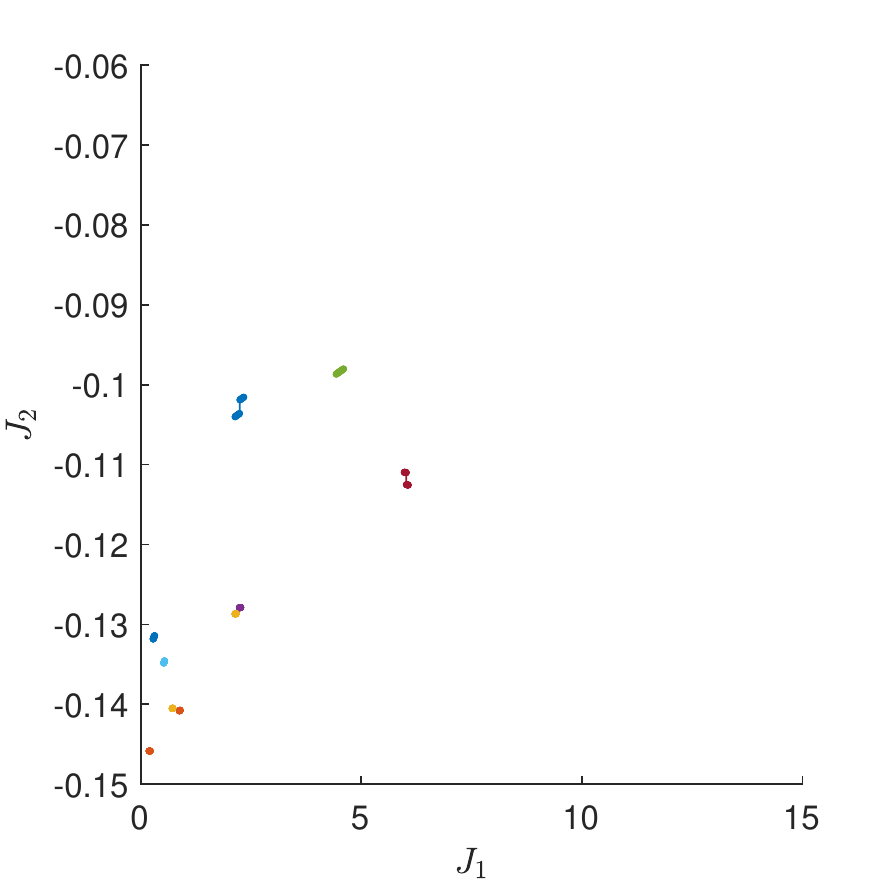} \hfil
            \includegraphics[width=.27\columnwidth]{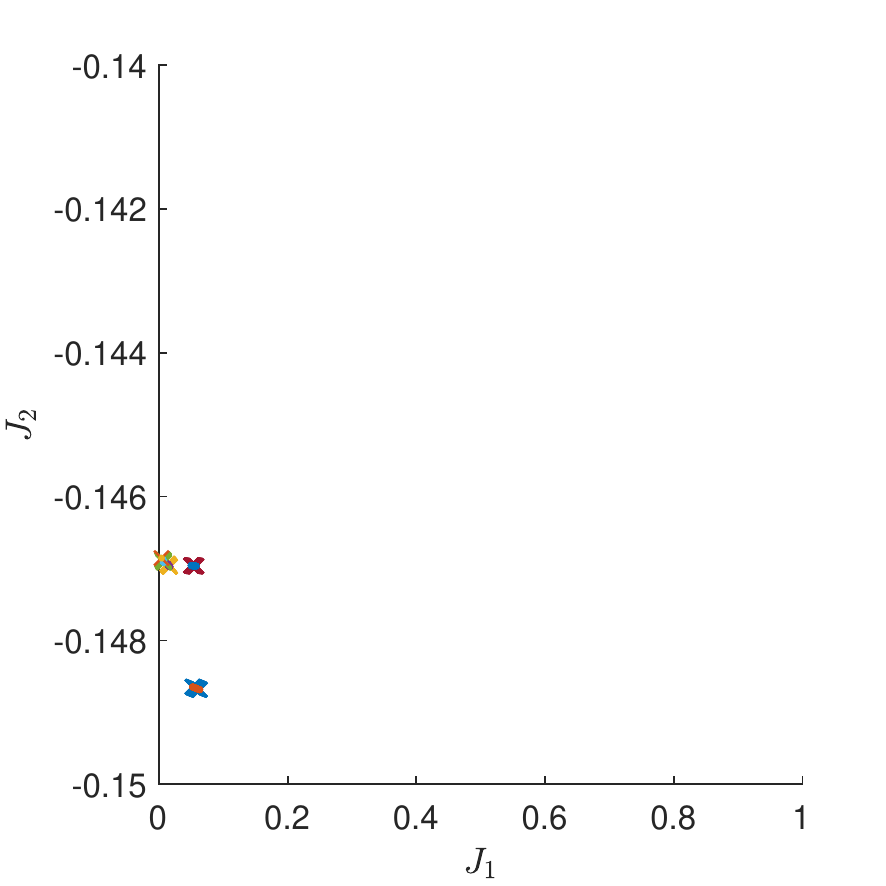} \hfil
            \includegraphics[width=.27\columnwidth]{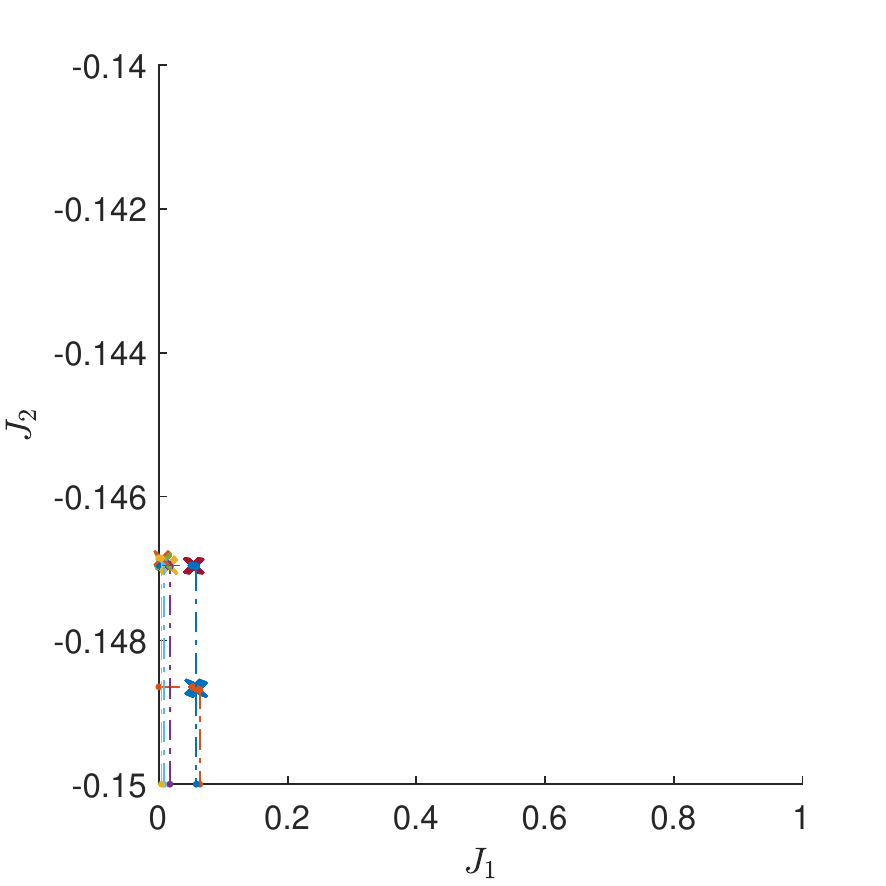}
            \caption{Uncertainty on the mass.}
        \end{subfigure}
        \caption{Scenario with initial conditions $\tilde x_0 = (-0.1801,    0.4349,         0,   -0.0694   -0.0222)^T$.}
        \label{fig:ex3unc}
    \end{figure}

    \begin{figure}
    	\centering
        \includegraphics[width=.45\columnwidth]{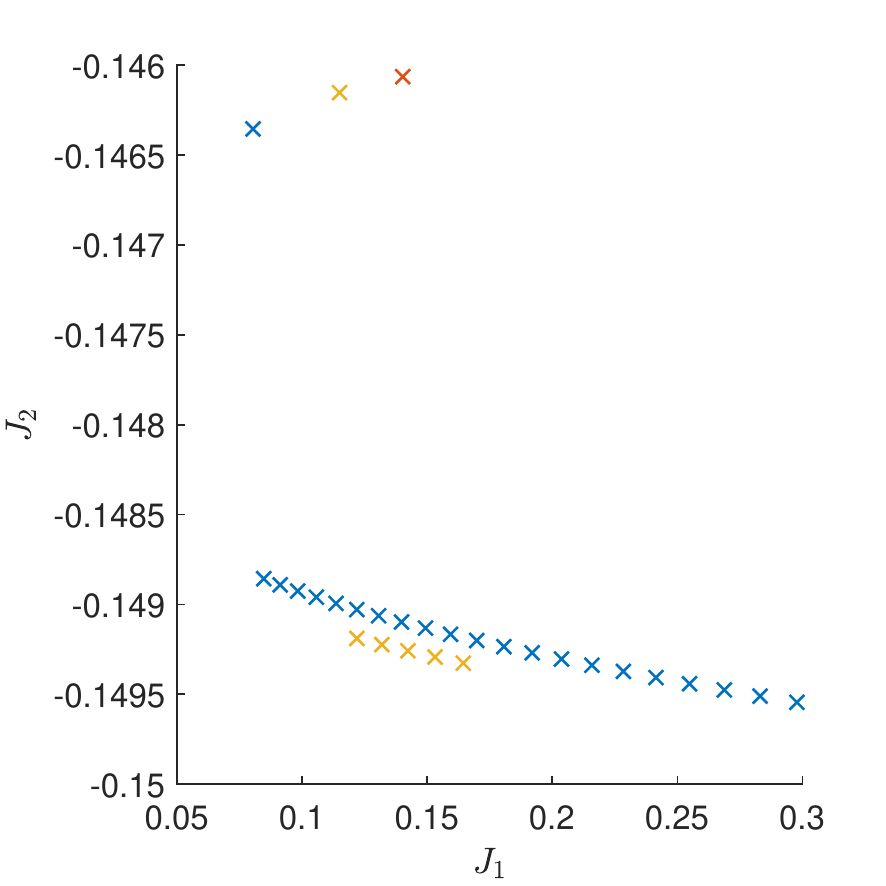} \hfil
        \includegraphics[width=.45\columnwidth]{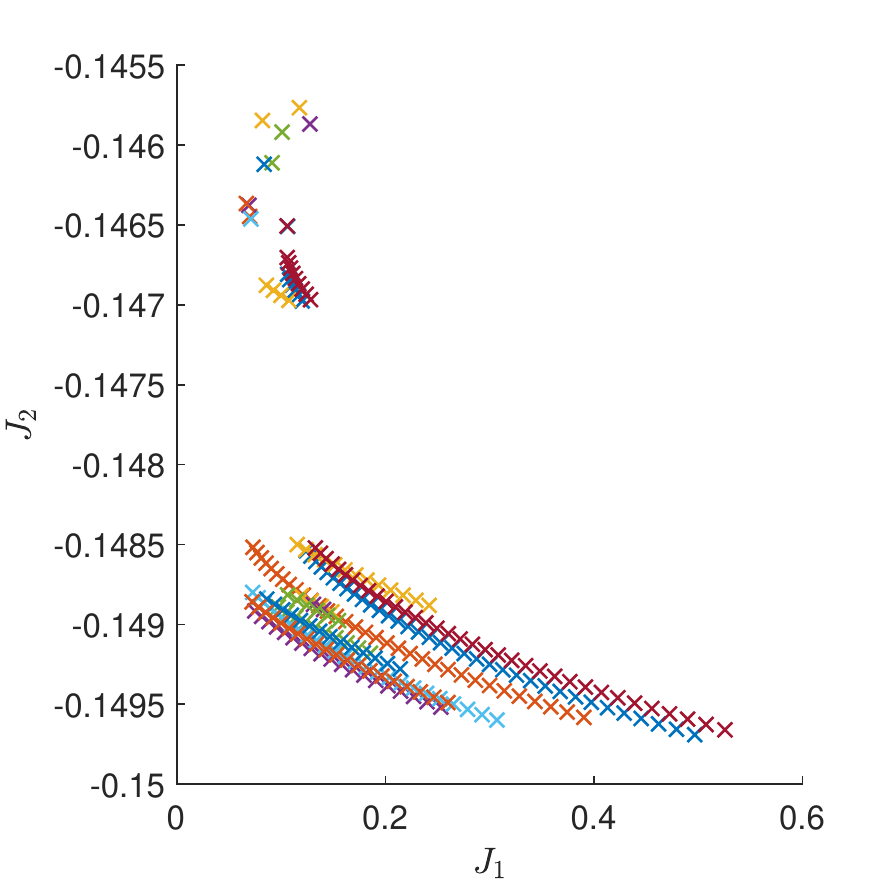}
        \caption{Efficient sets for the car maneuvering problem with uncertainty in the distance to the centerline $d$ for initial conditions $(-0.1801,    0.4349,         0,   -0.0694,   -0.0222)^T$ (left) and $(0.9842,   -0.9982,         0,    0.0783,   -0.0222)^T$ (right). The colors represent different efficient solutions.}
        \label{fig:fronts}
    \end{figure}

        To finish this section, we present a global sensitivity analysis on the parameters $\tilde x_0$. Global sensitivity analysis studies the amount of variance that would be neglected if one or more parameters were fixed. For our proposes, we computed the first-order sensitivity index \cite{sobol1993sensitivity}:
        \begin{equation}
            S_i = \frac{var\{E[\hat J|y_i]\}}{var\{\hat J\}},
        \end{equation}
        where $\hat J$ is the evaluation of the model at point $y = (\tilde x_0, m, L_f)$ and $E[y|x]$ is the conditional expectation. This index measures the contribution of input $y_i$ on the output variance without considering the interactions with other parameters.         
        
        Figure \ref{fig:si} shows the sensitivity index for the parameters for each objective functions. From the results, we can observe that the distance $d$ is the most sensitive parameter for $\hat J_1$ and $\kappa$ for $\hat J_2$. On the other hand, the mass and $L_f$ have almost no effect on the problem.
        \begin{figure}
            \centering
            \includegraphics[width=.5\columnwidth]{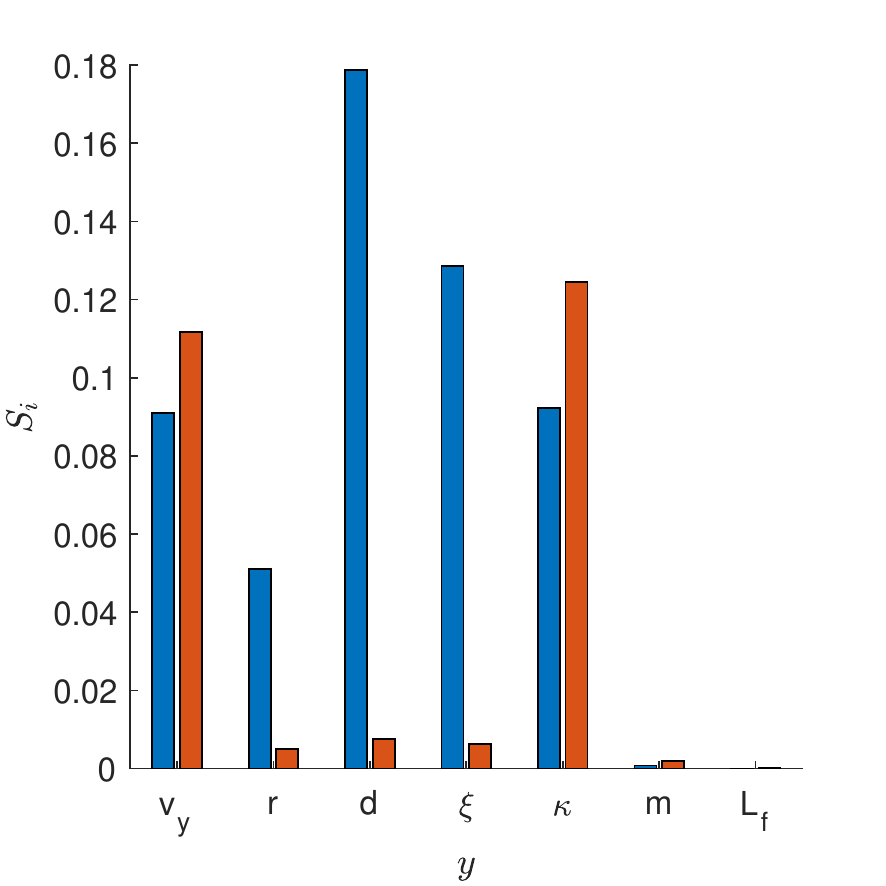}            
            \caption{Sensitivity index for each parameter. The first bar (blue) represents the first objective and the second one (red) the second objective.}        
            \label{fig:si}
        \end{figure}
        
    \subsection{Numerical Results}    
    In this section, we apply our approach to the multi-objective car maneuvering problem.
    %In the following, we present the parameters for the study. 
    Table~\ref{tab:param} shows the parameters for the library $\mathcal{L}$ that we use for the study.
    
    \begin{table}
        \centering
        \caption{Parameters for the library $\mathcal{L}$.}
        \label{tab:param}
        \begin{tabular}{|l|l|l|l|l|}
            \hline
            Variable    &Minimal value  &Maximal value  &Step size  &Number of grid points  \\  \hline
            $d$   &0      &10     &0.5    &21 \\
            $\xi$       &$-\pi/4$  &$\pi/4$   &$\pi/12$  &7  \\
            $v_y$   &-3     &3      &0.5    &13 \\
            $r$     &-6     &6      &1      &13 \\            
            $\kappa$&-0.1   &0.1    &0.025  &9  \\
            \hline
        \end{tabular}
    \end{table}        
    Further, we choose $v_x = 30, t_0=0$ sec, $t_e = 0.5$, and a time step of $h=0.05$ sec. Consequently, we have $u \in [u_{\min}, u_{\max}]^{11}$, where $u_{\min} = -0.5$ and $u_{\max} = -0.5$. In the online phase, we then apply the first three entries to the system. For our study, we assume uncertainty in the distance $d$ to the track centerline in the interval $[-0.25, 0.25]$. We selected the distance to the track as it the most sensitive parameter according to the sensitivity analysis from the previous section. 
            
    To evaluate the quality of the proposed hybrid approach, we first compare it to the approach proposed in \cite{ober2018explicit} (Opt Off/on). Therein, a library of Pareto fronts is constructed in the offline phase, cf.~Figure \ref{fig:bd}. In the online phase, a suitable input is computed via interpolation between library entries and according to the decision maker's preference (i.e., according to a fixed weighting vector). Alternatively, the weights may be adjusted automatically depending on the situation. One possibility is to increase the weight of the first objective (minimize the distance to the center of the lane) on straight parts and of the second objective (time to complete a lap) in curves. Note that this method addresses the problem without taking the uncertainty into account 
    %(Figure \ref{fig:bd}, Online Method) even though the model returns the initial conditions with uncertainty (Figure \ref{fig:bd}, Model). 
    This yields a good basis for comparison, as it helps to emphasize the effect of uncertainty when it is not considered in the optimization methods. 
    
    We then test the different components of our method separately. %First, we test the approach using the generic stochastic algorithm, the archiver, and the first step of the online phase (SBR-Off/on). This approach is the natural extension of the previous one to handle uncertainty.
    In the first step (SBR-Off/on), the only extension is to solve the Problem uMOCP, i.e., to take the uncertainty into account.
    In the next step, we only use online optimization via the reference point method (RPM), which was also used in \cite{ZF12} for deterministic problems.
    %Next, we tested the reference point method (RPM) based on the Hausdorff distance (SBR-$d_H$-RPM). 
     In Figure \ref{fig:bd}, this means that the database of optimal solutions computed offline is not used.
    This method requires a starting point to perform the optimization of the driving strategy. For this propose, the first starting point is drawn at random, and for subsequent starting points, we selected the driving strategy found in the previous iteration as the initial condition. 
    Finally, in the Hybrid approach we use both techniques. This means that starting from an initial guess within the library, we use the RPM to compute a feasible Pareto optimum in the online phase. The properties of the different algorithms are compared in Table~\ref{tab:characterization} with respect to the explicit treatment of uncertainties as well as offline and online computations. To decide which control strategy is going to be used at each iteration, we used a static approach. We set $Z=(0, 0.7125)^T$ for all cases, which corresponds to the vector formed with the minimum of each objective function (this is also called the ideal vector).
    %The comparison between this approach and the hybrid method, allows us to study the potential advantages that come from using the library of efficient solutions to select a starting point.
    
    \begin{table}
        \centering
        \caption{Offline and online computation parts of the different methods.}
        \label{tab:characterization}
        
        \begin{tabular}{|l|c|c|c|c|}
            \hline
             & Opt Off/on & SBR Off/on & SBR-$d_H$-RPM & Hybrid \\ 
            \hline
            Uncertainty  & ~ & X & X & X \\
            Offline comp. & X & X & ~ & X \\
            Online comp. & X & X & ~ & X \\
            RPM & ~ & ~ & X & X \\
            \hline
        \end{tabular}        
    \end{table} 
    
    Figure~\ref{fig:track_all} shows the trajectories created by the respective algorithms on a test track. Figure~\ref{fig:comp} then shows the corresponding driving strategy (i.e., the control input), the distance to the center of the track, and the initial conditions for each step of the simulation (from left to right). Table~\ref{tab:results} compares the methods on the test track for the two objectives (note that the objective of maximizing the driven distance for a fixed time horizon can be transformed to minimizing the lap time, i.e., the driving time for a fixed distance). 
    
    From the results, we can observe that there is a clear advantage of actively addressing uncertainty with respect to both objectives (compared to Opt Off/on). In particular, the strong zigzag behavior (Figure \ref{fig:comp}a in the middle) can be avoided, which results in lower lap times as well as a more secure driving behavior.
    As it was expected, the hybrid method has the best performance.
    %among the methods that handle uncertainty. 
    However, it has higher computational cost since an optimization problem needs to be solved in every step. Thus, it can only be used if the time windows are sufficiently large to complete the optimization process. Also, both the hybrid and SBR-$d_H$-RPM show less control effort, which is unexpected since a conservative driving strategy, in general, would tend to require more control effort. In this case, the strategy selected allows getting close to the centerline gradually while advancing on the track. The behavior results in less abrupt changes in the direction when compared to the Opt Off/on method.
    
    \begin{figure}
    	\centering
        \includegraphics[width=.45\textwidth]{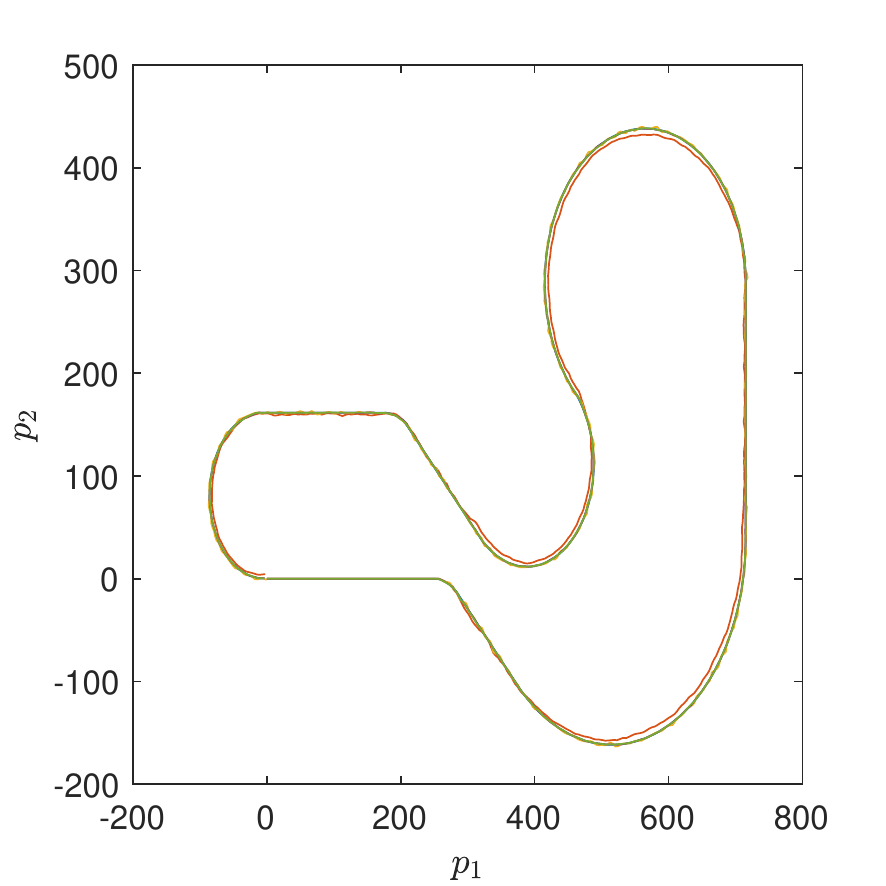} \hfil
        \includegraphics[width=.45\textwidth]{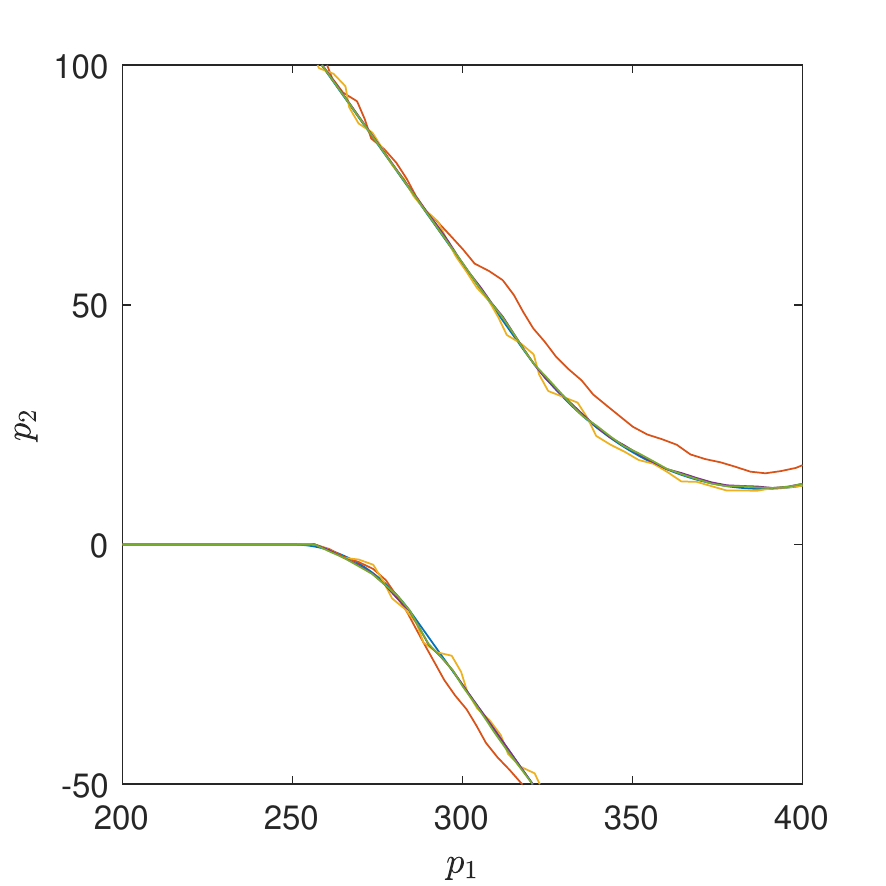} \\
        \includegraphics[width=.45\textwidth]{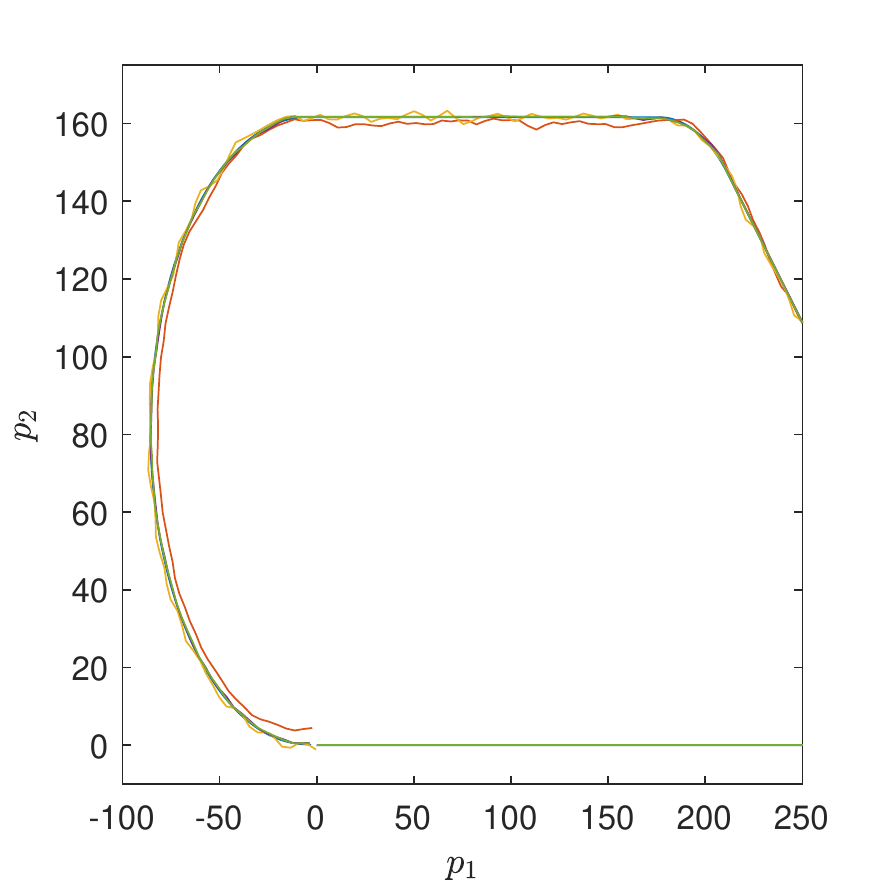} \hfil
        \includegraphics[width=.45\textwidth]{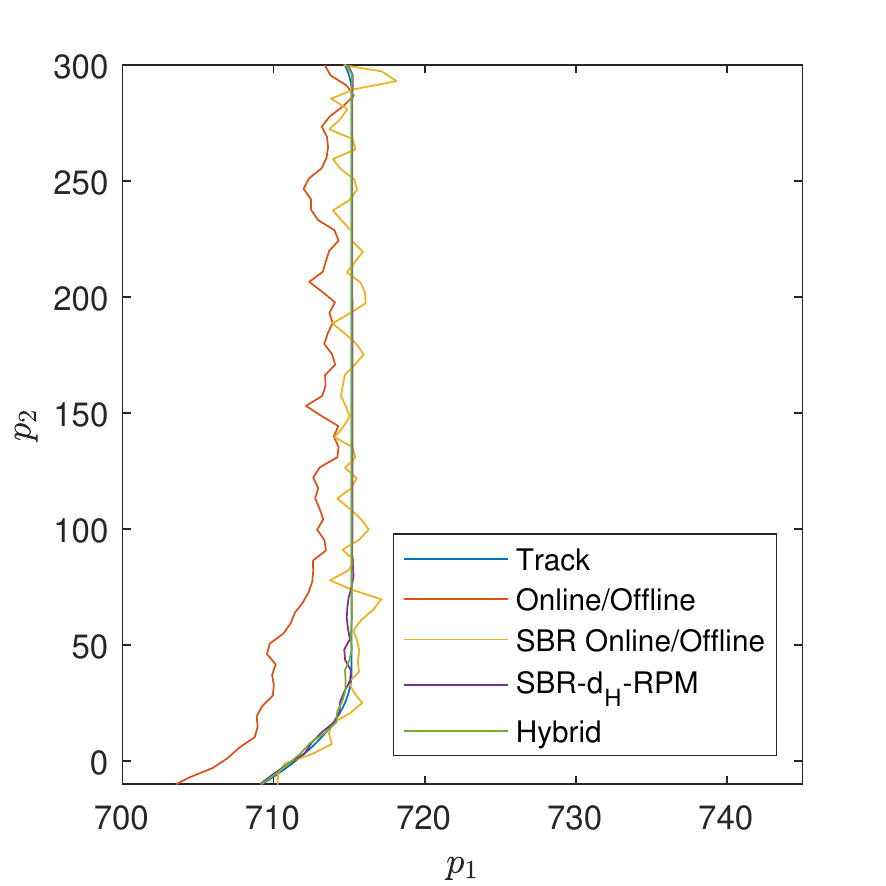}
        \caption{Test track results based on different approaches.}
        \label{fig:track_all}
    \end{figure}
    % \begin{figure}        
    %     \begin{subfigure}[t]{\textwidth}                 
    %         \includegraphics[width=.27\textwidth]{Bicycle-Model-7-3}
    %         \includegraphics[width=.27\textwidth]{Bicycle-Model-7-4}
    %         \includegraphics[width=.27\textwidth]{Bicycle-Model-7-5}\\
    %     \caption{Opt Off/on}
    %     \end{subfigure}\\
    %     \begin{subfigure}[t]{\textwidth}                 
    %         \includegraphics[width=.27\textwidth]{Bicycle-Model-5-3}
    %         \includegraphics[width=.27\textwidth]{Bicycle-Model-5-4}
    %         \includegraphics[width=.27\textwidth]{Bicycle-Model-5-5}
    %     \caption{SBR Off/on }
    %     \end{subfigure}\\
    %     \begin{subfigure}[t]{\textwidth}                 
    %         \includegraphics[width=.27\textwidth]{Bicycle-Model-4-3}
    %         \includegraphics[width=.27\textwidth]{Bicycle-Model-4-4}
    %         \includegraphics[width=.27\textwidth]{Bicycle-Model-4-5}\\                                        
    %     \caption{SBR-$d_H$-RPM}
    %     \begin{subfigure}[t]{\textwidth}                 
    %         \includegraphics[width=.27\textwidth]{Bicycle-Model-6-3}
    %         \includegraphics[width=.27\textwidth]{Bicycle-Model-6-4}
    %         \includegraphics[width=.27\textwidth]{Bicycle-Model-6-5}
    %         \caption{Hybrid}
    %     \end{subfigure}\\
    %     \end{subfigure}\\
    %     \caption{Control signal (left), distance to the track (center) and initial conditions of each method.}
    %     \label{fig:comp}
    % \end{figure}
    
    \begin{figure}        
    	\centering
        \includegraphics[width=.27\textwidth]{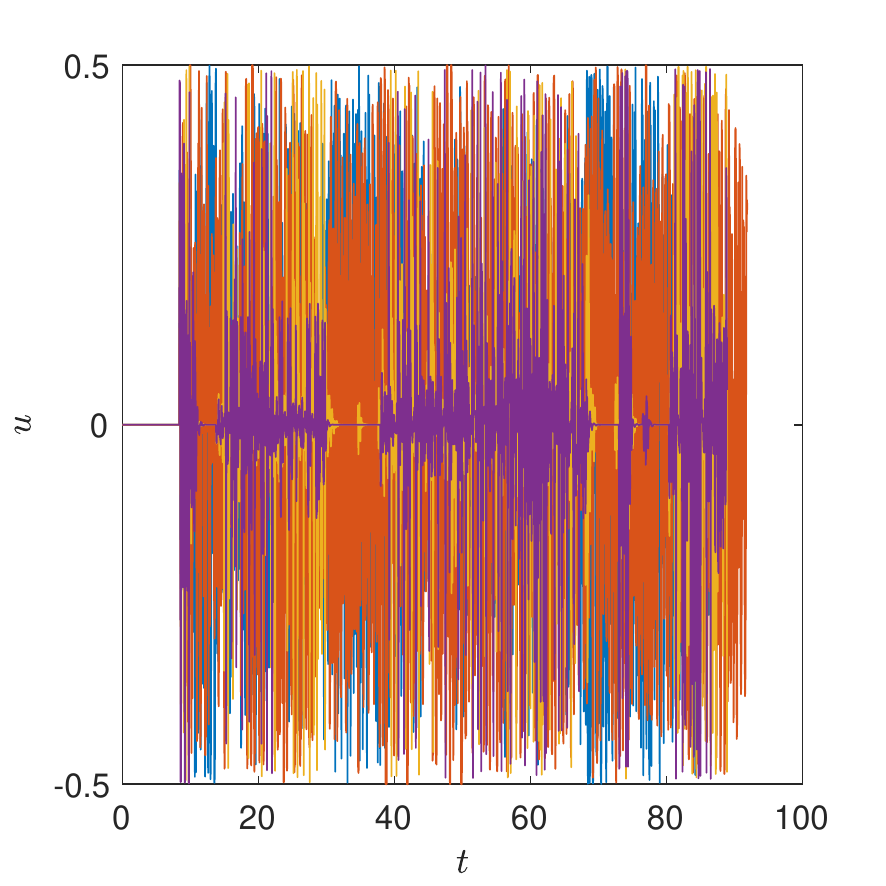}\hfil
        \includegraphics[width=.27\textwidth]{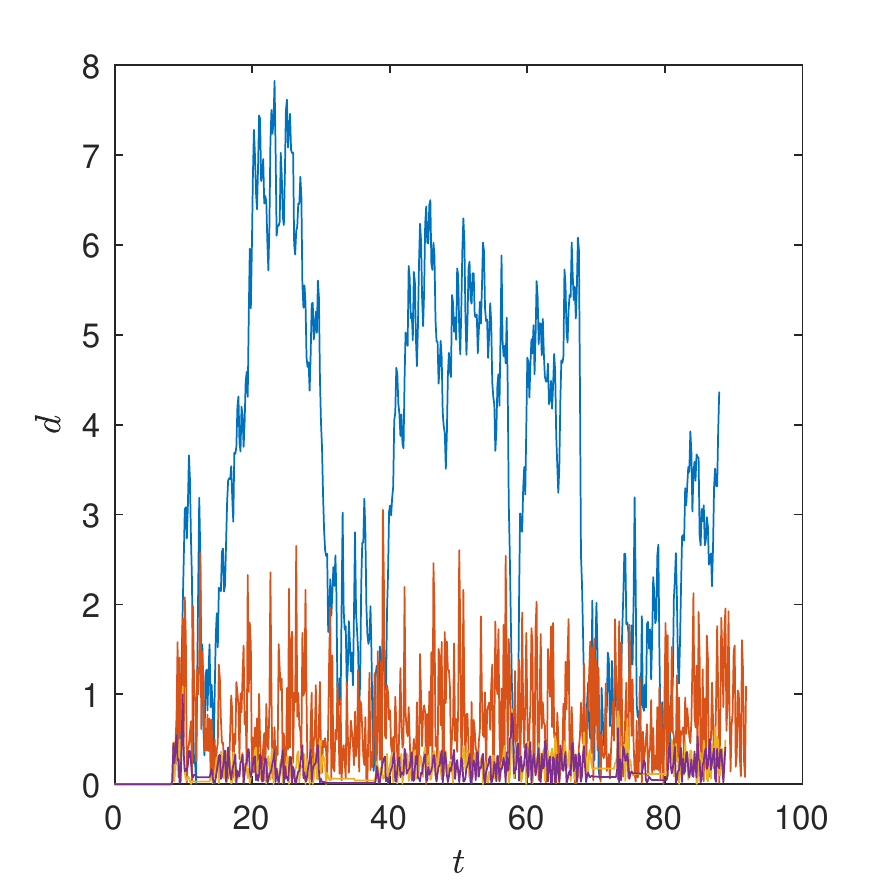}\hfil
        \includegraphics[width=.27\textwidth]{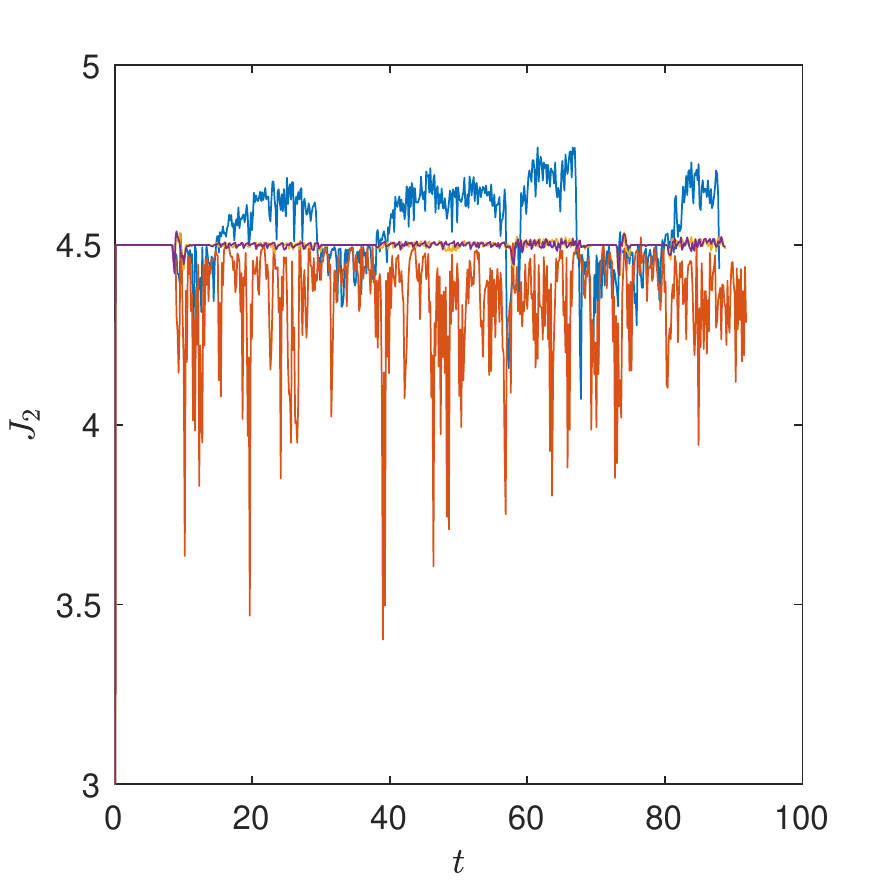}
        \caption{Control signal (left), distance to the track (center) and distance driven along the track (right) for Opt Off/on (blue), SBR Off/on (Orange), SBR-$d_H$-RPM (yellow) and the Hybrid (purple).}
        \label{fig:comp}
    \end{figure}
    
    \begin{table}
    \centering
        \caption{Comparison of the methods on the test track in terms of overall distance to the centerline and the time to complete a lap in the track.}
        \label{tab:results}
        
        \begin{tabular}{|p{2.5cm}|l|l|}
            \hline
            Method                  &Distance to center line    &Lap time\\ \hline            
            Opt Off/on              &396.46   &90.75    \\
            SBR Off/on              &454.95   &91.35   \\
            SBR-$d_H$-RPM            &82.474   &89.1    \\ 
            Hybrid                  &{\bf 74.435}   &{\bf 89.1} \\
            \hline
        \end{tabular}        
    \end{table}    
    
%%%%%%%%%%%%%%%%%%%%%%%%%%%%%%%%%%%%%%%%%%%%%%%%%%%%%%%%%%%%%%%%%%%%%%%%%%%%%%%%%%%%%%%%
    Finally, we compare the methods on five other tracks inspired by different racing circuits. We have taken the circuits images from Alastaro \cite{wiki:1}, Abudhabi \cite{wiki:2}, Catalunya \cite{wiki:3}, Melbourne \cite{wiki:4}, and Mexico \cite{wiki:5} and treated them to obtain their contours. In this case, besides the comparison in terms of distance to the centerline (Table \ref{tab:results2}) and lap time (Table \ref{tab:results3}), we also study the maximum distance to the centerline (Table \ref{tab:results4}). From the results, we can observe that in terms of overall distance to the centerline and maximum distance, the hybrid method yields the best results in at least four of the six test cases. In terms of lap time, the hybrid provides better results in three out of the six tracks. This result shows that in some cases the strategies can be over-conservative but at the same time yielding ``safer'' driving strategies.
    
    \begin{table}
    	\centering
        \caption{Comparison of the methods on the test tracks in terms of overall distance to the centerline.}
        \label{tab:results2}
        \begin{tabular}{|p{2.5cm}|l|l|l|l|l|l|}
            \hline
            Method               &Test &Alastaro &Abudhabi &Catalunya &Melburne &Mexico\\ \hline            
            Opt Off/on           &5484.4       &2357.1         &5903       &6625.7       &8131.8       & 3617.3\\
            SBR Off/on           &1231.6 &832.55  &1953.2  &2068.9  &2067.7  &1430.1 \\
            SBR-$d_H$-RPM         &307.42 &320.42  &715.08  &783.38  &{\bf 635.91}  &546.54 \\
            Hybrid               &{\bf 284.39}&{\bf 308.56} &{\bf 672.42}    &{\bf 644.01}  &660.96  &{\bf 514.7} \\
            \hline
        \end{tabular}        
    \end{table}
    \begin{table}
    	\centering
        \caption{Comparison of the methods on the test tracks in terms of lap time.}
        \label{tab:results3}
        \begin{tabular}{|p{2.5cm}|l|l|l|l|l|l|}
            \hline
            Method               &Test &Alastaro &Abudhabi &Catalunya &Melburne &Mexico\\ \hline            
            Opt Off/on           &{\bf 88.05}       & 50.25       & 119.4      &  {\bf 128.7}     & 141.75     &   85.2 \\
            SBR Off/on           &91.95 &47.25  &121.8  &133.2  &145.5  &84.6 \\
            SBR-$d_H$-RPM         &88.95 &45.75  &118.5  &129.6  &{\bf 141.45}  &82.05 \\
            Hybrid               &88.95 &{\bf 45.6}   &{\bf 117.6}  &129.15  &{\bf 141.45}  &{\bf 81.9} \\
            \hline
        \end{tabular}        
    \end{table}
    \begin{table}
    	\centering
        \caption{Comparison of the methods on the test tracks in terms of maximum distance to the centerline.}
        \label{tab:results4}
        \begin{tabular}{|p{2.5cm}|l|l|l|l|l|l|}
            \hline
            Method               &Test &Alastaro &Abudhabi &Catalunya &Melburne &Mexico\\ \hline            
            Opt Off/on           &7.8253 &12.519 &21.191  &7.7452  &8.1868  &15.809\\
            SBR Off/on           &3.3051 &4.5436 &9.1894  &{\bf 3.1675}  &3.6291  &7.035 \\
            SBR-$d_H$-RPM         &1.0887 &2.1066 &4.5819  &8.9224  &{\bf 3.0509}  &6.0147\\
            Hybrid               &{\bf 0.94629}&{\bf 3.1904} &{\bf 3.0061}  &3.8522  &3.7554  &{\bf 3.6958} \\
            \hline
        \end{tabular}        
    \end{table}
    
%%%%%%%%%%%%%%%%%%%%%%%%%%%%%%%%%%%%%%%%%%%%%%%%%%%%%%%%%%%%%%%%%%%%%%%%%%%%%%%%%%%%%%%%
\section{Conclusions and Future Work}
    In this work, we have introduced uncertainty to multi-objective optimal control problems in the initial state in the sense of set-based minmax robustness. 
    In order to achieve the necessary performance for feedback control, an offline/online strategy has to be used.
    For this purpose, we exploit symmetries in the control problems to reduce the complexity of the offline phase. Further, we have proposed a hybrid method to obtain feasible and more robust and efficient solutions. Therein, an additional optimization is performed in the online phase using a reference point scalarization approach.
    
    For the first step of the hybrid algorithm, we have presented a generic stochastic algorithm with an external archiver. This method can find an approximation of the set of efficient solutions in a single run of the algorithm. Further, we have proved that the algorithm converges in the limit to the set of efficient solutions in the Hausdorff sense. 
    In the second step, the reference point method is applied to improve further a solution selected from the library computed in the offline phase.
    %we proposed an extension of reference point methods to the problem at hand. 
    Given a reference point, the algorithm is capable of finding the closest solution when the worst-case regarding the uncertainty is considered. Moreover, under some assumption on the reference point, we proved that the solution found by the method is also efficient.
    
    Next, we have studied an application for autonomous driving to demonstrate the behavior of the methods. In our experiments, we found there is an advantage in considering uncertainty during the optimization process, and also in performing an additional online optimization. However, solving such problems becomes more expensive. Nevertheless, this hybrid approach yields very efficient and robust feedback signals while avoiding large parts of the expensive online computations. For the deterministic case, numerical experiments have shown that the performance is comparable to a globally optimal solution \cite{PSO+17}.

    For future work, it will be interesting to allow for adaptive weighting by a decision maker, for instance in order to allow for reactions to changing priorities or to external conditions. Moreover, it might be beneficial
    to study other stochastic algorithms to solve the problems more efficiently. Further, in the case of the reference point methods, it would be interesting to consider other distance measures for sets such as $\Delta_p$ \cite{SELC12} as well as study achievement scalarizing functions. Finally, it would be interesting to test the approaches in other real-world applications as well as to study equality constraints handling techniques for this kind of problem.

 \section*{Acknowledgments}
 CIHC acknowledges Conacyt for funding no. 711172. 
 SP acknowledges support by the DFG Priority Programme 1962.
 The calculations were performed on resources provided by the Advanced Research Computing (ARC) of the University of Oxford.

%%%%%%%%%%%%%%%%%%%%%%%%%%%%%%%%%%%%%%%%%%%%%%%%%%%%%%%%%%%%%%%%%%%%%%%%%%%%%%%%%%%%%%%%%%%%%%
%% Bibliography
%%%%%%%%%%%%%%%%%%%%%%%%%%%%%%%%%%%%%%%%%%%%%%%%%%%%%%%%%%%%%%%%%%%%%%%%%%%%%%%%%%%%%%%%%%%%%%
\bibliographystyle{alpha}
\bibliography{Bibliography}

\newcommand{\etalchar}[1]{$^{#1}$}
\begin{thebibliography}{PSOB{\etalchar{+}}17}

\bibitem[Amd67]{Amdahl:1967}
G.~M. Amdahl.
\newblock Validity of the single processor approach to achieving large scale
  computing capabilities.
\newblock In {\em Proceedings of the April 18-20, 1967, Spring Joint Computer
  Conference}, AFIPS '67 (Spring), pages 483--485, New York, NY, USA, 1967.
  ACM.

\bibitem[BBM03]{bemporad2003min}
A.~Bemporad, F.~Borrelli, and M.~Morari.
\newblock Min-max control of constrained uncertain discrete-time linear
  systems.
\newblock {\em IEEE Transactions on automatic control}, 48(9):1600--1606, 2003.

\bibitem[BF06]{BF06}
A.~Bemporad and C.~Filippi.
\newblock {An Algorithm for Approximate Multiparametric Convex Programming}.
\newblock {\em Computational Optimization and Applications}, 35(1):87--108,
  2006.

\bibitem[BM09]{BP09}
A.~Bemporad and D.~{Mu{\~{n}}oz de la Pe{\~{n}}a}.
\newblock {Multiobjective model predictive control}.
\newblock {\em Automatica}, 45(12):2823--2830, 2009.

\bibitem[BMDP02]{bemporad2002explicit}
A.~Bemporad, M.~Morari, V.~Dua, and E.~N. Pistikopoulos.
\newblock The explicit linear quadratic regulator for constrained systems.
\newblock {\em Automatica}, 38(1):3--20, 2002.

\bibitem[Bow76]{Bowman76}
V.~J. Bowman.
\newblock On the {R}elationship of the {T}chebycheff {N}orm and the {E}fficient
  {F}rontier of {M}ultiple-{C}riteria {O}bjectives.
\newblock In {\em Multiple Criteria Decision Making}, volume 130 of {\em
  Lecture Notes in Economics and Mathematical Systems}, pages 76--86. Springer
  Berlin Heidelberg, 1976.

\bibitem[BS07]{Beyer20073190}
H.-G. Beyer and B.~Sendhoff.
\newblock Robust optimization: A comprehensive survey.
\newblock {\em Computer Methods in Applied Mechanics and Engineering},
  196(33:34):3190 -- 3218, 2007.

\bibitem[CLV07]{Coello:07}
C.~A. {Coello Coello}, G.~B. Lamont, and D.~A. {Van Veldhuizen}.
\newblock {\em {Evolutionary Algorithms for Solving Multi-Objective Problems}}.
\newblock Springer, New York, second edition, September 2007.
\newblock ISBN 978-0-387-33254-3.

\bibitem[DB12]{DB12}
C.~Danielson and F.~Borrelli.
\newblock {\em {Symmetric Explicit Model Predictive Control}}, volume~4.
\newblock IFAC, 2012.

\bibitem[DD98]{das:98}
I.~Das and J.~Dennis.
\newblock {Normal-boundary intersection: A new method for generating the Pareto
  surface in nonlinear multicriteria optimization problems}.
\newblock {\em {SIAM Journal of Optimization}}, 8:631--657, 1998.

\bibitem[Deb01]{deb:01}
K.~Deb.
\newblock {\em Multi-{O}bjective {O}ptimization using {E}volutionary
  {A}lgorithms}.
\newblock John Wiley \& Sons, Chichester, UK, 2001.
\newblock ISBN 0-471-87339-X.

\bibitem[DKW18]{doolittle2018robust}
E.~K. Doolittle, H.~L. Kerivin, and M.~M. Wiecek.
\newblock Robust multiobjective optimization with application to internet
  routing.
\newblock {\em Annals of Operations Research}, 271(2):487--525, 2018.

\bibitem[DSH05]{dsh:05}
M.~Dellnitz, O.~Sch\"utze, and T.~Hestermeyer.
\newblock Covering {P}areto sets by multilevel subdivision techniques.
\newblock {\em Journal of Optimization Theory and Applications}, 124:113--155,
  2005.

\bibitem[Ehr05]{ehrgott:05}
M.~Ehrgott.
\newblock {\em Multicriteria Optimization}.
\newblock Springer, 2005.

\bibitem[EIS14]{ehrgott2014minmax}
M.~Ehrgott, J.~Ide, and A.~Sch{\"o}bel.
\newblock Minmax robustness for multi-objective optimization problems.
\newblock {\em European Journal of Operational Research}, 239(1):17--31, 2014.

\bibitem[EKS17]{Eichfelder2017}
G.~Eichfelder, C.~Kr{\"u}ger, and A.~Sch{\"o}bel.
\newblock Decision uncertainty in multiobjective optimization.
\newblock {\em Journal of Global Optimization}, 69(2):485--510, Oct 2017.

\bibitem[FDF05]{frazzoli2005maneuver}
E.~Frazzoli, M.~A. Dahleh, and E.~Feron.
\newblock Maneuver-based motion planning for nonlinear systems with symmetries.
\newblock {\em IEEE transactions on robotics}, 21(6):1077--1091, 2005.

\bibitem[Fra01]{frazzoli2001robust}
E.~Frazzoli.
\newblock {\em Robust hybrid control for autonomous vehicle motion planning}.
\newblock PhD thesis, Massachusetts Institute of Technology, 2001.

\bibitem[FS00]{FS00}
J.~Fliege and B.~F. Svaiter.
\newblock {Steepest descent methods for multicriteria optimization}.
\newblock {\em Mathematical Methods of Operations Research}, 51(3):479--494,
  2000.

\bibitem[FW14]{fliege2014robust}
J.~Fliege and R.~Werner.
\newblock Robust multiobjective optimization \& applications in portfolio
  optimization.
\newblock {\em European Journal of Operational Research}, 234(2):422--433,
  2014.

\bibitem[GP17]{grune2017nonlinear}
L.~Gr{\"u}ne and J.~Pannek.
\newblock Nonlinear model predictive control.
\newblock In {\em Nonlinear Model Predictive Control}, pages 45--69. Springer,
  2017.

\bibitem[Han99]{Hanne1999}
T.~Hanne.
\newblock On the convergence of multiobjective evolutionary algorithms.
\newblock {\em European Journal Of Operational Research}, 117(3):553--564,
  1999.

\bibitem[HD19]{hu2019efficient}
J.~Hu and B.~Ding.
\newblock An efficient offline implementation for output feedback min-max mpc.
\newblock {\em International Journal of Robust and Nonlinear Control},
  29(2):492--506, 2019.

\bibitem[HM79]{hwang1979multiple}
C.~Hwang and A.~Masud.
\newblock Multiple objective decision making-methods and applications: a
  state-of-the-art survey.
\newblock {\em Lecture notes in economics and mathematical systems}, 164, 1979.

\bibitem[HSS17]{hernandez2017global}
C.~Hern{\'a}ndez, O.~Sch{\"u}tze, and J.-Q. Sun.
\newblock Global multi-objective optimization by means of cell mapping
  techniques.
\newblock In {\em EVOLVE--A Bridge between Probability, Set Oriented Numerics
  and Evolutionary Computation VII}, pages 25--56. Springer, 2017.

\bibitem[Hsu87]{hsu:87}
C.~S. Hsu.
\newblock {\em {Cell-to-cell mapping: a method of global analysis for nonlinear
  systems}}.
\newblock Applied mathematical sciences. Springer-Verlag, 1987.

\bibitem[IS16]{ide2016robustness}
J.~Ide and A.~Sch\"{o}bel.
\newblock Robustness for uncertain multi-objective optimization: A survey and
  analysis of different concepts.
\newblock {\em OR Spectrum}, 38(1):235--271, January 2016.

\bibitem[Joh02]{Joh02}
T.~A. Johansen.
\newblock {On multi-parametric nonlinear programming and explicit nonlinear
  model predictive control}.
\newblock In {\em 41st IEEE Conference on Decision and Control}, volume~3,
  pages 2768--2773, 2002.

\bibitem[KL12]{kuroiwa2012robust}
D.~Kuroiwa and G.~M. Lee.
\newblock On robust multiobjective optimization.
\newblock {\em Vietnam J. Math}, 40(2-3):305--317, 2012.

\bibitem[KPH{\etalchar{+}}14]{kerschke2014cell}
P.~Kerschke, M.~Preuss, C.~Hern{\'a}ndez, O.~Sch{\"u}tze, J.-Q. Sun, C.~Grimme,
  G.~Rudolph, B.~Bischl, and H.~Trautmann.
\newblock Cell mapping techniques for exploratory landscape analysis.
\newblock In {\em EVOLVE-A Bridge between Probability, Set Oriented Numerics,
  and Evolutionary Computation V}, pages 115--131. Springer, 2014.

\bibitem[Lof03]{lofberg2003approximations}
J.~Lofberg.
\newblock Approximations of closed-loop minimax mpc.
\newblock In {\em 42nd IEEE International Conference on Decision and Control
  (IEEE Cat. No. 03CH37475)}, volume~2, pages 1438--1442. IEEE, 2003.

\bibitem[LSKVI10]{logist2010efficient}
F.~Logist, S.~Sager, C.~Kirches, and J.~Van~Impe.
\newblock Efficient multiple objective optimal control of dynamic systems with
  integer controls.
\newblock {\em Journal of Process Control}, 20(7):810--822, 2010.

\bibitem[LTDZ02]{LTDZ2002b}
M.~Laumanns, L.~Thiele, K.~Deb, and E.~Zitzler.
\newblock Combining convergence and diversity in evolutionary multiobjective
  optimization.
\newblock {\em Evolutionary Computation}, 10(3):263--282, 2002.

\bibitem[MM95]{miettinen1995interactive}
K.~Miettinen and M.~M{\"a}kel{\"a}.
\newblock Interactive bundle-based method for nondifferentiable multiobjective
  optimization: Nimbus.
\newblock {\em Optimization}, 34(3):231--246, 1995.

\bibitem[OBP18]{ober2018explicit}
S.~Ober-Bl{\"o}baum and S.~Peitz.
\newblock Explicit multiobjective model predictive control for nonlinear
  systems with symmetries.
\newblock {\em arXiv preprint arXiv:1809.06238}, 2018.

\bibitem[OBRzF12]{ober2012solving}
S.~Ober-Bl{\"o}baum, M.~Ringkamp, and G.~zum Felde.
\newblock Solving multiobjective optimal control problems in space mission
  design using discrete mechanics and reference point techniques.
\newblock In {\em 2012 IEEE 51st IEEE Conference on Decision and Control
  (CDC)}, pages 5711--5716. IEEE, 2012.

\bibitem[Par27]{pareto:71}
V.~Pareto.
\newblock {\em {Manual of Political Economy}}.
\newblock The MacMillan Press, 1971 original edition in French in 1927.

\bibitem[PD18a]{PD18}
S.~Peitz and M.~Dellnitz.
\newblock {A Survey of Recent Trends in Multiobjective Optimal Control -
  Surrogate Models, Feedback Control and Objective Reduction}.
\newblock {\em Mathematical and Computational Applications}, 23(2), 2018.

\bibitem[PD18b]{PD18b}
S.~Peitz and M.~Dellnitz.
\newblock {Gradient-based multiobjective optimization with uncertainties}.
\newblock In Y.~Maldonado, L.~Trujillo, O.~Sch{\"{u}}tze, A.~Riccardi, and
  M.~Vasile, editors, {\em NEO 2016}, volume 731, pages 159--182. Springer,
  2018.

\bibitem[POBD19]{POBD19}
S.~Peitz, S.~Ober-Bl{\"{o}}baum, and M.~Dellnitz.
\newblock {Multiobjective Optimal Control Methods for the Navier-Stokes
  Equations Using Reduced Order Modeling}.
\newblock {\em Acta Applicandae Mathematicae}, 161(1):171--199, 2019.

\bibitem[PSOB{\etalchar{+}}17]{PSO+17}
S.~Peitz, K.~Sch{\"{a}}fer, S.~Ober-Bl{\"{o}}baum, J.~Eckstein,
  U.~K{\"{o}}hler, and M.~Dellnitz.
\newblock {A Multiobjective MPC Approach for Autonomously Driven Electric
  Vehicles}.
\newblock {\em IFAC PapersOnLine}, 50(1):8674--8679, 2017.

\bibitem[SCM{\etalchar{+}}19]{schutze2019pareto}
O.~Sch{\"u}tze, O.~Cuate, A.~Mart{\'\i}n, S.~Peitz, and M.~Dellnitz.
\newblock Pareto explorer: a global/local exploration tool for many-objective
  optimization problems.
\newblock {\em Engineering Optimization}, pages 1--24, 2019.

\bibitem[SELCC12]{SELC12}
O.~Sch{\"u}tze, X.~Esquivel, A.~Lara, and C.~A. Coello~Coello.
\newblock Using the averaged {H}ausdorff distance as a performance measure in
  evolutionary multi-objective optimization.
\newblock {\em IEEE Transactions on Evolutionary Computation}, 16(4):504--522,
  2012.

\bibitem[SLT{\etalchar{+}}10]{schuetze_ecj:10}
O.~Sch\"utze, M.~Laumanns, E.~Tantar, C.~A.~C. Coello, and E.-G. Talbi.
\newblock Computing gap free {P}areto front approximations with stochastic
  search algorithms.
\newblock {\em Evolutionary Computation}, 18(1):65--96, 2010.

\bibitem[Sob93]{sobol1993sensitivity}
I.~M. Sobol.
\newblock Sensitivity analysis for non-linear mathematical models.
\newblock {\em Mathematical modelling and computational experiment},
  1:407--414, 1993.

\bibitem[SXSH18]{sun2018cell}
J.-Q. Sun, F.-R. Xiong, O.~Sch{\"u}tze, and C.~Hern{\'a}ndez.
\newblock {\em Cell Mapping Methods}.
\newblock Springer, 2018.

\bibitem[TL90]{taheri1990investigation}
S.~Taheri and E.~H. Law.
\newblock Investigation of a combined slip control braking and closed loop four
  wheel steering system for an automobile during combined hard braking and
  severe steering.
\newblock In {\em 1990 American Control Conference}, pages 1862--1867. IEEE,
  1990.

\bibitem[VB18]{vargas2018generalization}
A.~Vargas and J.~Bogoya.
\newblock A generalization of the averaged hausdorff distance.
\newblock {\em Computaci{\'o}n y Sistemas}, 22(2), 2018.

\bibitem[Wie80]{wierzbicki1980use}
A.~P. Wierzbicki.
\newblock The use of reference objectives in multiobjective optimization.
\newblock In {\em Multiple criteria decision making theory and application},
  pages 468--486. Springer, 1980.

\bibitem[{Wik}19a]{wiki:2}
{Wikipedia contributors}.
\newblock Abu dhabi grand prix --- {Wikipedia}{,} the free encyclopedia, 2019.
\newblock [Online; accessed 23-October-2019].

\bibitem[{Wik}19b]{wiki:1}
{Wikipedia contributors}.
\newblock Alastaro circuit --- {Wikipedia}{,} the free encyclopedia, 2019.
\newblock [Online; accessed 23-October-2019].

\bibitem[{Wik}19c]{wiki:4}
{Wikipedia contributors}.
\newblock Australian grand prix --- {Wikipedia}{,} the free encyclopedia, 2019.
\newblock [Online; accessed 23-October-2019].

\bibitem[{Wik}19d]{wiki:3}
{Wikipedia contributors}.
\newblock Circuit de barcelona-catalunya --- {Wikipedia}{,} the free
  encyclopedia, 2019.
\newblock [Online; accessed 23-October-2019].

\bibitem[{Wik}19e]{wiki:5}
{Wikipedia contributors}.
\newblock Mexican grand prix --- {Wikipedia}{,} the free encyclopedia, 2019.
\newblock [Online; accessed 23-October-2019].

\bibitem[WPGK16]{walton2016numerical}
C.~Walton, C.~Phelps, Q.~Gong, and I.~Kaminer.
\newblock A numerical algorithm for optimal control of systems with parameter
  uncertainty.
\newblock {\em IFAC-PapersOnLine}, 49(18):468--475, 2016.

\bibitem[Zad63]{Zadeh63}
L.~Zadeh.
\newblock Optimality and {N}on-{S}calar-{V}alued {P}erformance {C}riteria.
\newblock {\em {IEEE} Transactions on Automatic Control}, 8:59--60, 1963.

\bibitem[ZFT12]{ZF12}
V.~M. Zavala and A.~Flores-Tlacuahuac.
\newblock {Stability of multiobjective predictive control: A utopia-tracking
  approach}.
\newblock {\em Automatica}, 48(10):2627--2632, 2012.

\bibitem[ZKMS18]{Zhou:2017}
Y.~Zhou-Kangas, K.~Miettinen, and K.~Sindhya.
\newblock Interactive multiobjective robust optimization with nimbus.
\newblock In M.~Baum, G.~Brenner, J.~Grabowski, T.~Hanschke, S.~Hartmann, and
  A.~Sch{\"o}bel, editors, {\em Simulation Science}, pages 60--76, Cham, 2018.
  Springer International Publishing.

\end{thebibliography}
\end{document}